\documentclass [a4paper] {article}

\usepackage[latin1]{inputenc}
\usepackage[english]{babel}
\usepackage{amsmath,amssymb,amsthm,color,graphicx}

\usepackage{amsfonts}

\usepackage[left=2.5cm,right=2.5cm,top=4cm,bottom=4cm,a4paper]{geometry}
\usepackage{xspace}

\usepackage[arrow, matrix, curve]{xy}

\newcommand{\kommentar}[1]{}

\newcommand{\norm}[1]{\left\|{#1}\right\|}
\newcommand{\abs}[1]{\left|{#1}\right|}
\newcommand{\set}[1]{\left\{{#1}\right\}}
\newcommand{\lie}[1]{\left[{#1}\right]}
\def\sbs{\subseteq}
\def\sps{\supseteq}

\newcommand{\es}[1]{\begin{equation*}\begin{split}{#1}\end{split}\end{equation*}}

\newcommand{\kD}[1]{\begin{equation*}\begin{xy}\xymatrix{#1}\end{xy}\end{equation*}}

\theoremstyle{plain}
\newtheorem{thm}{Theorem}[section]
\newtheorem{lem}[thm]{Lemma}
\newtheorem{prop}[thm]{Proposition}
\newtheorem{cor}[thm]{Corollary}

\theoremstyle{definition}
\newtheorem{rem}[thm]{Remark}
\newtheorem{defi}[thm]{Definition}
\newtheorem{exa}[thm]{Example}

\def\1{\mathbf{1}}
\def\A{\mathcal A}
\def\ad{\operatorname{ad}}

\def\Aut{\operatorname{Aut}}

\def\Bc{\operatorname{\check{B}}}
\def\C{\mathbb C}
\def\Cent{\operatorname{Cent}}
\def\Ci{\operatorname{C}^{\infty}}
\def\Ck{\operatorname{C}^k}
\def\Co{\mathfrak{Cov}}
\def\D{\mathcal D}
\def\d{\partial}
\def\Der{\operatorname{Der}}
\def\Diff{\operatorname{Diff}}
\def\E{\mathbb E}
\def\End{\operatorname{End}}
\def\ep{\varepsilon}
\def\ev{\operatorname{ev}}
\def\g{\mathfrak g}
\def\gl{\mathfrak{gl}}
\def\GL{\operatorname{GL}}
\def\H{\mathbb H}
\def\h{\mathfrak h}
\def\Hc{\operatorname{\check{H}}}
\def\Hom{\operatorname{Hom}}
\def\I{\mathcal I}
\def\id{\operatorname{id}}
\def\ide{\trianglelefteq}
\def\im{\operatorname{im}}
\def\Iso{\operatorname{Iso}}
\def\K{\mathbb K}
\def\k{\mathfrak k}
\def\L{\mathbb L}
\def\L{\mathbb L}

\def\lra{\longrightarrow}
\def\N{\mathbb N}
\def\Nc{\mathcal N}
\def\Ni{\overline{\mathbb N}}
\def\O{\mathcal O}
\def\o{\mathfrak o}
\def\P{\mathfrak P}
\def\ph{\varphi}

\def\R{\mathbb R}
\def\S{\mathbb S}
\def\sl{\mathfrak{sl}}
\def\so{\mathfrak{so}}
\def\sp{\mathfrak{sp}}
\def\su{\mathfrak{su}}
\def\supp{\operatorname{supp}}
\def\U{\mathfrak U}
\def\V{\mathfrak V}
\def\W{\mathfrak W}

\def\wt{\widetilde}
\def\Y{\mathcal Y}
\def\Z{\mathbb Z}
\def\z{\mathfrak z}
\def\Zc{\operatorname{\check{Z}}}

\def\diffem{diffeomorphism\xspace}
\def\diffec{diffeomorphic\xspace}

\def\La{Lie algebra\xspace}

\def\vs{vector space\xspace}
\def\Kvs{$\K$-vector space\xspace}

\def\iff{if and only if\xspace}
\def\wrt{with respect to\xspace}

\title{\huge{Lie algebras of smooth sections}}
\author{Hasan Gündo\u gan}
\date{July 2007}

\begin{document}

\maketitle

\subsection*{Abstract}
Lie algebras of smooth sections are Lie algebras obtained
from bundles of Lie algebras, where the latter are vector bundles of which
the fibers are Lie algebras. We also consider the $\Ck$-sections for $k \in
\N$. This paper, which is essentially my diploma thesis from May
2007 at Technische Universität Darmstadt,  studies the derivations, the
centroid and the isomorphisms of such Lie algebras and generalizes some facts
from Lecomte's publications \cite{Le79} and \cite{Le80} to the case where the
fiber is perfect or centerfree and it gives some more explicit
proofs.

\tableofcontents

\section{Introduction}
\subsection{Motivation and requirements}
There are two main goals in the analysis of Lie algebras of smooth sections: On 
the one hand, the infinite-dimensional Lie algebras are not yet as well studied
as the finite-dimensional ones and the methods used in finite-dimensional Lie
theory are difficult to adapt to the infinite-dimensional case. I will discuss
some properties of the Lie algebras of smooth sections in Chapter 3, especially
their derivations, centroids and isomorphisms.

A solid knowledge of analysis, linear algebra, topology and some knowledge of
the theory of manifolds, e.g. a priori to be acquired in \cite{Ne05}, is
required. However, there is no knowledge of Lie algebra theory required, for
the relevant parts will be explained.

\subsection{Remarks concerning notation}
I will follow some notational conventions which ought to be clarified.
\begin{itemize}
\item The formulae $A \sbs B$ and $B \sps A$ will mean that each element of the set $A$ is an
element of the set $B$ and equality is not excluded. In order to describe a subset relation
excluding equality I will write $A \subsetneq B$ or $B \supsetneq A$. \item All vector spaces and
(bi)linear maps are considered over a field $\K \in \set{\R,\C}$. \item The set $\set{0,1,2, \ldots
}$ is denoted by $\N$ and the set $\set{0,1,2, \ldots } \cup \set{\infty}$ is denoted by $\Ni$.
\item A $\Ck$-map, $k \in \Ni$, is a $k$-times continuously differentiable function. A function is
also called ``smooth'' instead of $\Ci$. \item A manifold $M$ will always be smooth,
finite-dimensional over $\R$, hausdorff, paracompact and connected. \item Let $(e_1, \ldots , e_m)$
denote the canonical basis of $\R^m$. Then $(U,\xi)$ being a chart of a manifold $M$ with $\dim M =
m$, $x\in U$ and $f: U \to E$ being a $\operatorname{C}^1$-map between $U$ and a finite-dimensional
vector space $E$, we define the notation \[ \left(\partial_{\xi_i}\right)_x {f} := \partial_{\xi_i}
{f}(x):=\left. \frac{d}{dt} \right|_{t=0}  {f}\left(\xi^{-1}(\xi(x) + t e_i)\right). \] As a basis
of $T_xM$ we can take $\left( \left(\d_{\xi_1}\right)_x, \ldots , \left(\d_{\xi_m}\right)_x
\right)$, where $\left(\d_{\xi_i}\right)_x:= T_{\xi(x)}\left(\xi^{-1}\right)(e_i)$. We write
$\d_{\xi_i}$ for the map $U \to T_xM$, $x \mapsto \left(\partial_{\xi_i}\right)_x$. Any vector field
$X \in \mathcal V(M)$ takes the local form $\sum_{i=1}^m X^i \d_{\xi_i}$ for certain functions $X^1,
\ldots X^m \in \Ci(U,\R)$. For $f \in \Ci(M,\R)$ we define $X.f \in \Ci(M,\R)$ by $(X.f)(x) := (T_x
f) (X_x)$ and we see that the two different definitions of $\left(\partial_{\xi_i}\right)_x$ are
linked by the formula $\left(\partial_{\xi_i}\right)_x (f) = \left(\left(\partial_{\xi_i}\right).
f\right)(x).$

If $f$ is a $\Ck$-function and $\alpha=(\alpha_1, \ldots , \alpha_m) \in \N^m$  a multi-index with $\abs{\alpha}:=\sum_{i=1}^m \alpha_i \le k$, then we write:
\es{
\partial_{\xi}^{\alpha}f := \partial_{\xi_1}^{\alpha_1} \cdots \partial_{\xi_m}^{\alpha_m} f.
}
For a multi-index $\alpha \in \N^m$ we also define $\alpha!:= \prod_{i=1}^m \alpha_i$. Then, if
$\xi(U)$ is convex, the Taylor Formula can be stated as follows for $x,y \in U$ and some $z \in
[0,1]\cdot (x-y) + y \sbs U$: \es{ f(x+y) = \sum_{\abs{\alpha} < k} \frac{\prod_{i=1}^m
\xi_i(y)^{\alpha_i}}{\alpha !} \cdot \partial_{\xi}^\alpha f(x) + \sum_{\abs{\alpha} = k}
\frac{\prod_{i=1}^m \xi_i(y)^{\alpha_i}}{\alpha !} \cdot \partial_{\xi}^\alpha f(z). } Furthermore,
we define the notation $\gamma \le \alpha$ for multi-indices $\alpha, \gamma \in \N^m$, which means
$\gamma_i \le \alpha_i$ for all $i \in \set{1, \ldots , m}$, and we define the multinomial
coefficients for multi-indices $\gamma \le \alpha \in \N^m$: \[ \begin{pmatrix}\alpha \\
\gamma\end{pmatrix}:=\frac{\alpha !}{\gamma !(\alpha-\gamma)!}. \] For multi-indices $\gamma \le
\alpha \in \N^m$ and any canonical basis vector $e_i$ of $\R^m$, the multinomial coefficients
satisfy the formula \[ \begin{pmatrix}\alpha + e_i\\ \gamma+e_i\end{pmatrix} = \begin{pmatrix}\alpha
\\ \gamma\end{pmatrix} + \begin{pmatrix}\alpha \\ \gamma+e_i\end{pmatrix}. \]
\item A Lie group is a manifold equipped with a smooth group multiplication whose inversion is smooth.
\item The neutral element of a group is denoted by 1 and if $g$ is an element of the group, then $\lambda_g$ is the multiplication by $g$ from the left and $\rho_g$ is the multiplication by $g$ from the right.
\item Depending on the context, the symbol $\1$ denotes the identity map of a vector space, the endomorphism of a vector bundle where each $\1_x$ is the identity of the $x$-fiber, or a constant map with value 1 on a manifold.

\item For numbers $m,n \in \N \backslash \set{0}$, the space of all $m \times n$-matrices with entries in $\K$ is denoted by $M_{m,n}(\K)$ and is canonically identified with the space of linear functions $\K^n \to \K^m$. Note that, if we have $m=0$ or $n=0$, then $M_{m,n}(\K)$ is the vector space containing the unique linear map $\K^n \to 0$ or $0 \to \K^n$, respectively.
\item If $f: M \to E$ is a $\Ck$-map between a manifold and a finite-dimensional vector space, $n \in \N$ with $n \le k$ and $x \in M$, then we define the notation
\[
j^n_x(f) = 0 \quad :\Longleftrightarrow \quad \begin{cases}f(x)=0 & \text{ if } n=0 \\ j^{n-1}_x(f)=0 \text{ and } T^n_x f = 0 & \text{ if } n>0. \end{cases}
\]
\item If, for each point $x \in M$ of a set, $T_x=T(x): N \to P$ is a map and $A_x=A(x) \in N$ is a point, then we denote the map $M \to P$, $x \mapsto T_x (A_x)$ by $T \cdot A$.
If, for each point $x \in M$ of a set, $T_x=T(x): N \to P$ is a map and $S_x=S(x): L \to N$ is a map, then we denote the map $M \to P^L$, $x \mapsto T_x \circ S_x$ by $T \circ S$.
\end{itemize}

\section{Definitions, notions, former results}
\subsection{Associative algebras and Lie algebras}
\begin{defi} Let $A$ be a \vs.
\begin{enumerate}
\item $A$ is called an \emph{algebra} if equipped with a bilinear map $A \times A \to A, (x,y) \mapsto x\cdot y$, which we call the \emph{multiplication} of the algebra $A$.
\item An algebra $A$ is called \emph{associative}, if $x \cdot (y \cdot z) = (x \cdot y) \cdot z$ is satisfied for all $x,y,z \in A$. In this case the expression $x\cdot y \cdot z$ is clear without ambiguity. An associative algebra $A$ may possess an element $1 \in A$ with $1\cdot x = x \cdot 1 = x$ for all $x \in A$, called \emph{unity element}.
\item An algebra $\g$ is called a \emph{\La}and its multiplication map is also written
$\lie{\cdot,\cdot}: \g \times \g \to \g$, $(x,y) \mapsto \lie{x,y}$, if the following two conditions are satisfied:
\begin{enumerate}
\item $\lie{\cdot,\cdot}$ is alternating, i.e. $\lie{x,x}=0$ for all $x \in \g$.
\item the \emph{Jacobi identity}, i.e. $\lie{\lie{x,y},z}+\lie{\lie{z,x},y}+\lie{\lie{y,z},x}=0$ for all $x,y,z \in \g$.
\end{enumerate}
\item If $S,T$ are subsets of an algebra $A$, we define $S \cdot T$ to be the \vs generated by the set $\set{s\cdot t: s \in S, t \in T}$. In the Lie algebra case we also write $\lie{S,T}$ instead of $S \cdot T$.
\item Let $B$ be a vector subspace of an algebra $A$. It is called \emph{left ideal} of $A$, if $A \cdot B \sbs B$, i.e. $x \cdot y\in B$ for all $x \in A$, $y \in B$. It is called \emph{right ideal} of $A$, if $B \cdot A \sbs B$, i.e. $y \cdot x\in B$ for all $x \in A$, $y \in B$. It is called \emph{(two-sided) ideal} of $A$, symbolically $B \ide A$, if it is a left ideal and a right ideal of $A$.
In the Lie algebra case, the conditions $A \cdot B \sbs B$ and $B \cdot A \sbs B$ are equivalent.
\item A vector subspace $B$ of an algebra $A$ is called a \emph{subalgebra} of $A$, symbolically $B \leq A$, if $x \cdot y \in B$ for all $x,y \in B$, i.e. the multiplication of $A$ induces a multiplication of $B$. Of course a subalgebra of $A$ is an algebra, too.
\item Let $\g$ be a Lie algebra. The vector subspace $\lie{\g,\g}$ is called the \emph{commutator} of $\g$ and $\g$ is \emph{perfect} if $\g = \lie{\g,\g}$. The vector subspace $\set{\left.x \in \g \right|\lie{x,y}=0 \text{ for all }y\in \g}$ is called the \emph{center} of $\g$, denoted by $\z(\g)$, and $\g$ is \emph{abelian} if $\z(\g)=\g$ or, equivalently, $\lie{\g,\g}=0$. The Lie algebra $\g$ is called \emph{centerfree}, if $\z(\g)=0$.
\item A Lie algebra $\g$ is \emph{nilpotent}, if its \emph{lower central series} $\g^0, \g^1, \g^2, \ldots$, which is defined by $\g^0:=\g$ and $\g^{n+1}:=\lie{\g,\g^n}$ for $n \in \N$, becomes zero eventually.
\item A Lie algebra $\g$ is \emph{solvable}, if its \emph{derived series} $\g^{(0)}, \g^{(1)}, \g^{(2)}, \ldots$, which is defined by $\g^{(0)}:=\g$ and $\g^{(n+1)}:=\lie{\g^{(n)},\g^{(n)}}$ for $n \in \N$, becomes zero eventually.
\end{enumerate}
\end{defi}

\begin{rem}
In our case of $\K$ being a field of characteristic 0, an algebra multiplication
is alternating \iff it is skew-symmetric: On the one hand we have, for all $x,y 
\in \g:$ \[ \lie{x,y} = - \lie{y,x} \implies \lie{x,x} = - \lie{x,x} \implies 2
\lie{x,x}=0 \implies  \lie{x,x}=0. \]
On the other hand, for all $x,y,z \in \g$:
\es{
\lie{z,z}=0 &\implies 0=\lie{x+y,x+y}=\lie{x,x}+\lie{x,y}+\lie{y,x}+\lie{y,y}=\lie{x,y}+\lie{y,x}\\
&\implies \lie{x,y} = - \lie{y,x}.
}
\end{rem}

The next lemma (of which I leave out the elementary proof) shows four ``canonical'' ways to construct algebras.

\begin{lem} $ $
\begin{enumerate}
\item Let $V$ be an arbitrary \vs. Then its \vs endomorphisms form an associative algebra where the multiplication of the algebra is $\circ$, the composition of functions. This associative algebra is called the \emph{endomorphism algebra} of $V$, denoted by $\End(V)$\footnote{The symbol $\End(V)$ will also be used to describe only the \textsl{set} of \vs endomorphisms of $V$.}.
\item Let $A$ be an associative algebra. $A$ equipped with the map $\lie{\cdot,\cdot}: A \times A \to A$, $(a,b) \mapsto a\cdot b - b\cdot a$, called the \emph{commutator bracket} of $A$, is a \La, denoted by $A_L$.
\item $A$ being an algebra, the vector space $\set{\left.f \in \End(A)\right|f(a\cdot b)=f(a)\cdot b + a \cdot f(b) \text{ for all } a,b \in A}$ is a Lie subalgebra of $\End(A)_L$, called the \emph{Lie algebra of derivations} of $A$, denoted by $\Der(A)$.
\item $A$ being an algebra and $I \ide A$ an ideal, the quotient vector space $A / I$ equipped with the well-defined induced multiplication
\es{
A/I \times A/I &\lra A/I\\
(a + I, b + I) &\longmapsto (a \cdot b) + I
}
is an algebra, called \emph{quotient algebra} of $A$ modulo $I$, denoted by $A/I$.
\end{enumerate}
\end{lem}

\begin{defi}
The set of \emph{Lie algebra morphisms} of $\g$ and $\h$ is
\es{
\Hom(\g,\h):=\set{\left.f: \g \to \h \text{ linear map }\right|f(\lie{x,y})=\lie{f(x),f(y)}\text{ for all }x,y\in \g}.
}
The set of \emph{Lie algebra isomorphimsm} of $\g$ and $\h$ is
\es{
\Iso(\g,\h):=\set{\left.f: \g \to \h \text{ linear isomorphism }\right|f(\lie{x,y})=\lie{f(x),f(y)}\text{ for all }x,y\in \g}.
}
In the special case of $\g = \h$ we even have a group $\Aut(\g):=\Iso(\g,\g)$ with neutral element $\1=\id_\g$, which is called \emph{group of Lie algebra automorphisms} of $\g$.
\end{defi}

Now it is time to give some important examples of algebras.

\begin{exa}\label{exampAlg} \
\begin{enumerate}
\item Every \vs $V$ can be equipped with the \emph{trivial Lie bracket}: $\lie{x,y}:=0$ for all $x,y \in V$.
\item $S$ being a set and $A$ being an algebra, the set $A^S$ of all mappings $S \to A$ can be pointwisely equipped with an algebra structure. An important subalgebra of such an algebra is $\operatorname{C}^k(M,\K) \le \K^M$, where $M$ is a smooth manifold, $\K$ is equipped with the field multiplication as algebra multiplication and $k \in \Ni$.
\item If $\left(\g,\lie{\cdot,\cdot}\right)$ is a Lie algebra and $\ph: V \to \g$ is an isomorphism of \vs{}s, then we can define a Lie bracket on $V$ as follows:
\es{
\lie{\cdot,\cdot}_{\ph}: V \times V &\lra V\\
(v,w) &\longmapsto \ph^{-1}\lie{\ph(v),\ph(w)}.
}
Then $\ph$ becomes an isomorphism of the Lie algebras $\left(V,\lie{\cdot,\cdot}_{\ph}\right)$ and $\left(\g,\lie{\cdot,\cdot}\right)$.
\item If $V$ is a \vs, we call $\gl(V):=\End(V)_L$ the \emph{general \La of $V$}.
\item If $V$ is a \vs and $\dim(V) < \infty$, we call $\sl(V):=\set{\left. f \in \End(V)\right| \text{tr}(f)=0}$ the \emph{special \La of $V$}. In fact, $\sl(V)$ is a Lie subalgebra of $\gl(V)$:
\[
\text{tr}\left(\lie{f,g}\right) = \text{tr}(f\circ g - g \circ f) = \text{tr}(f\circ g) - \text{tr}(g \circ f)=\text{tr}(f\circ g) -\text{tr}(f\circ g) =0
\]
This calculation even shows that $\lie{\gl(V),\gl(V)} \sbs \sl(V)$.
\item The squared matrices over $\K$ with $n$ rows form the associative algebra $M_{n,n}(\K)$ \wrt the multiplication of matrices. It is isomorphic to $\End\left(\K^n\right)$. The corresponding \La $\left(M_{n,n}(\K)\right)_L$ isomorphic to $\gl\left(\K^n\right)$ is denoted by $\gl_n(\K)$ and called \emph{general linear \La of order $n$}.

Analogously, $\sl_n(\K):=\set{\left.A\in M_n(\K)\right|\text{tr}(A)=0}$ is called \emph{special linear \La of order $n$} and is isomorphic to $\sl(\K^n)$. We identify $\End\left(\K^n \right),\gl\left(\K^n \right),\sl\left(\K^n \right)$ with $M_n(\K)$, $\gl_n(\K)$, $\sl_n(\K)$, respectively.
\item Let $\beta: V \times V \to \K$ be a bilinear map. The set
\[
\o(V,\beta):=\set{\left.f\in \End(V)\right|\beta(f(x),y)+\beta(x,f(y))=0}
\]
is a $\K$-Lie subalgebra of $\gl(V)$, called the \emph{Lie algebra of
$\beta$-skew-symmetric endomorphisms} of $V$. There are some important examples
of this type of Lie algebra: \begin{enumerate}
\item For $V=\K^n$ and
$\beta(x,y)=\sum_{i=1}^n x_i y_i$ we write $\o_n(\K):=\o(V,\beta)$. This is the
\emph{orthogonal \La of order $n$}, which can be identified with \es{
\set{\left.A\in M_n(\K)\right|A+A^T=0}, } the set of skew-symmetric $n \times
n$-matrices. The space $\so_n(\K):=\o_n(\K) \cap \sl_n(\K)$ is a Lie subalgebra
of $\o_n(\K)$, called \emph{special orthogonal \La of order $n$}.\footnote{In
our considered case of $\K \in \set{\R,\C}$, there is no difference between
$\so_n(\K)$ and $\o_n(\K)$. The case $\so_n(\K) \subsetneq \o_n(\K)$ is only
possible if $\text{char}(\K) =2$.}
\item For $V=\C^n$ and
$\beta(x,y)=\sum_{i=1}^n x_i \overline{y_i}$ we write ${\mathfrak
u}_n(\C):=\o(V,\beta)$. This is the \emph{unitary \La of order $n$}, which can
be identified with \es{ \set{\left.A\in M_n(\C)\right|A+\overline{A^T}=0}, } the
set of complex skew-hermitian $n \times n$-matrices.
Note that this is \textsl{not} a complex, but a real \La, since it is not a
$\C$-\vs. The space
$\su_n(\C):={\mathfrak u}_n(\C) \cap \sl_n(\C)$ is a Lie subalgebra of
${\mathfrak u}_n(\C)$, called \emph{special unitary \La of order $n$}. \item For
$V=\K^{2n}$ and $\beta(x,y)=\sum_{i=1}^n x_i y_{n+i} - x_{n+i} y_i$ we write
$\sp_{2n}(\K):=\o(V,\beta)$. This is the \emph{symplectic \La of order $n$},
which can be identified with \[ \set{\left.\begin{pmatrix}A & B\\C &
-A^T\end{pmatrix}\in M_2\left(M_n(\K)\right) \cong M_{2n}(\K)\right|B=B^T,
C=\operatorname{C}^T}. \]

\end{enumerate}
\end{enumerate}
\end{exa}

\begin{rem}\
\begin{enumerate}
\item One can show that for any Lie algebra $\g$ there is an associative algebra with unity $\mathcal{U}(\g)$, such that $\g$ can be embedded into $\left(\mathcal{U}(\g)\right)_L$ via an injective morphism of Lie algebras $\eta_{\g}: \g \to \left(\mathcal{U}(\g)\right)_L$ (cf. Part III of \cite{Ne94}). If one demands a universal property\footnote
{
If $A$ is an associative algebra with unity and $\alpha: \g \to A_L$ is a morphism of Lie algebras, then there is a unique morphism of associative algebras $\alpha': \mathcal{U}(\g) \to A$ such that $\alpha'(1)=1$ and $\alpha' \circ \eta_{\g} = \alpha$.
} of $\mathcal{U}(\g)$ and $\eta_{\g}$, then $(\mathcal{U}(\g),\eta_{\g})$ is unique up to algebra isomorphism, i.e. an isomorphism of vector spaces preserving the multiplications, and $(\mathcal{U}(\g),\eta_{\g})$ is called \emph{universal enveloping algebra} of $\g$.

\item The \emph{Ado Theorem} states that each finite-dimensional Lie algebra is isomorphic to a Lie subalgebra of $\gl_n(\K)$ for certain $n \in \N$, i.e. each finite-dimensional Lie algebra is, up to isomorphism, a Lie algebra of squared matrices.
\end{enumerate}
\end{rem}

\begin{defi}  Let $\g$ be a \La.
\begin{enumerate}
\item If $V$ is a \vs, then a \emph{representation of $\g$ on $V$} is a morphism of Lie algebras
\es{
\rho: \g \lra \gl(V).
}
\item A \emph{$\g$-module} is a \vs $V$ with a $\K$-bilinear map
\es{
\mu: \g \times V &\lra V\\
(x,v) &\longmapsto \mu(x,v)=:x . v
}
satisfying the equation
\[
\lie{x,y}. v = x . (y . (v)) - y . (x . (v))
\]
for all $x,y \in \g$ and $v \in V$.

It is not difficult to obtain a bijective correspondence between the concepts of $\g$-module and representation of $\g$: Given a module with multiplication $\mu: \g \times V \to V$ one can define a representation by
\es{
r(\mu): \g &\lra \gl(V)\\
x &\longmapsto (v \mapsto \mu(x,v)).
}
Given a representation $\rho: \g \lra \gl(V)$ one can define a module structure on $V$ by
\es{
m(\rho): \g \times V &\lra V\\
(x,v) &\longmapsto \rho(x)(v).
}
It is easy to verify that $r(m(\rho))=\rho$ and $m(r(\mu))=\mu$ for every representation $\rho$ and every module multiplication $\mu$.

A linear map $f: V \to W$ between two modules is called \emph{module morphism}, if $f(x.v)=x.f(v)$ for all $x \in \g, v \in V$. The vector space of module morphisms $V \to W$ is denoted by $\Hom_\g(V,W)$ and we write $\End_\g(V):=\Hom_\g(V,V)$. A vector subspace $W \sbs V$ of a module which is invariant under $\g$ is itself a module, called \emph{submodule}. Images and kernels of module morphisms, quotients of modules modulo submodules and intersections of submodules are again modules. Note that the direct sum (in the sense of vector subspaces) of two submodules is also a submodule.
\item A module $V\neq0$ is called
\begin{enumerate}
\item \emph{simple}, if $0$ is its only proper submodule.
\item \emph{semisimple}, if for each of its submodules $W$ there is a complementary submodule $W' \sbs V$ such that $W\oplus W' = V$.
\item \emph{trivial}, if $\g.V=0$.
\end{enumerate}
\end{enumerate}
\end{defi}
Some useful facts about simple and semisimple modules are presented in the following two lemmas, shown e.g. in \cite{Ne02}:
\begin{lem}  \
\begin{enumerate}
\item Submodules and quotient modules of semisimple modules are semisimple, too.
\item The following statements are equivalent for a Lie algebra $\g$ and a $\g$-module $V$:
\begin{enumerate}
\item $V$ is semisimple.
\item $V$ is a sum of simple modules.
\item $V$ is a direct sum of simple modules.
\end{enumerate}
\end{enumerate}
\end{lem}

\begin{lem}  \textbf{\upshape(Schur Lemma)} Let $V,W$ be simple $\g$-modules of a \La $\g$.
\begin{enumerate}
\item $\Hom_\g(V,W)=0$, if $V$ is not isomorphic to $W$.
\item Each non-zero element of $\End_\g(V)$ is invertible.
\item If $\dim(V) < \infty$ and $\K=\C$, then $\End_\g(V)=\C \cdot \1$.
\end{enumerate}
\end{lem}

\begin{defi} Let $\g$ be a \La.
We define the \emph{adjoint representation}\footnote{$\ad$ represents $\g$ on $\g$: $\ad_{\lie{x,y}}(z)=\lie{\lie{x,y},z}=-\lie{y,\lie{x,z}}+\lie{x,\lie{y,z}}=\lie{\ad_x,\ad_y}(z)$\\
$\implies \ad_{\lie{x,y}}=\lie{\ad_x,\ad_y}$.} of $\g$ to be the map
\es{
\ad: \g &\lra \Der(\g)\\
x &\longmapsto \ad_x:=\left(y \mapsto \lie{x,y} \right).
}
The Jacobi identity for $\g$ ensures that the range is properly chosen. Note that $\z(\g)=\ker(\ad)$.
\end{defi}
The adjoint representation of $\g$ corresponds to the $\g$-module structure on $\g$, where the multiplication is the Lie bracket of $\g$. In this sense the ideals of $\g$ are the submodules of $\g$. We use this change of perspective in order to define the notions ``simple'' and ``semisimple'' for \La{}s.
\begin{defi} Let $\g$ be a \La and consider it as a $\g$-module \wrt the adjoint representation.
\begin{enumerate}
\item $\g$ is a \emph{simple} \La, if it is a \textsl{non-trivial} simple $\g$-module.
\item $\g$ is a \emph{semisimple} \La, if it is the direct sum of \textsl{non-trivial} simple $\g$-modules, i.e. of simple Lie algebras.
\item $\g$ is a \emph{reductive} \La, if it is the direct sum of simple $\g$-modules., i.e. ideals of $\g$.
\end{enumerate}
\end{defi}

\begin{exa}
The Lie algebras $\sl_n(\K)$ for $n \ge 2$ and $\so_n(\K)=\o_n(\K)$ for $n \ge 5$ and $\sp_{2n}(\K)$ for $n \ge 1$ are simple. The Lie algebras $\gl_n(\K)$ for $n \ge 1$ are reductive and $\dim \z\left(\gl_n(\K)\right)=1$. For proofs, cf. e.g. \cite{Ne02}.
\end{exa}

\begin{rem}
If $\g=\mathfrak{a} \oplus \mathfrak{b}$ is a decomposition of a Lie algebra into a direct sum of ideals, then $\lie{\mathfrak{a}, \mathfrak{b}} \sbs \mathfrak{a} \cap \mathfrak{b} = 0$.
\end{rem}

The relation between reductive and semisimple Lie algebras is described in the following lemma. For a proof, cf. e.g. \cite{Ne02}.
\begin{lem}
Let $\g$ be a Lie algebra.
\begin{enumerate}
\item If $\g$ is semisimple, then it is reductive, perfect, centerfree and
$\Der(\g)=\ad(\g)$. So $\ad: \g \to \Der(\g)$ is an isomorphism of Lie algebras.
\item If $\g$ is reductive, then $\g=\lie{\g,\g} \oplus
\z(\g)$ and $\lie{\g,\g}$ is semisimple. In particular, a reductive Lie algebra
is semisimple \iff it is centerfree \iff it is perfect. \end{enumerate}
\end{lem}

\begin{exa}\label{nonreduLa}
Let $\g$ be a two-dimensional \Kvs with basis $\left(x_1,x_2\right)$. The unique
skew-symmetric bilinear map $\lie{\cdot,\cdot}: \g \times \g \to \g$ determined
by the relation $\lie{x_1,x_2} := x_1$ is a Lie
bracket. The Lie algebra $\left(\g,\lie{\cdot,\cdot}\right)$ is centerfree, but
not perfect, thus not reductive.
Let $\h$ be a five-dimensional \Kvs with basis $\left(y_1, \ldots, y_5\right)$.
The unique skew-symmetric bilinear map $\set{\cdot,\cdot}: \h \times \h \to \h$
determined by the relations
\es{
\set{y_1,y_2} := y_1, \quad \set{y_1,y_3} := y_2, \quad \set{y_1,y_4} &:= y_3,
\quad \set{y_2,y_3} := y_4, \quad \set{y_2,y_4} := y_5, \\
\set{y_1,y_5}:=\set{y_2,y_5}:=\set{y_3,y_4}&:=\set{y_3,y_5}:=\set{y_4,y_5}:=0.
}
is a Lie bracket. The Lie algebra $\left(\h,\set{\cdot,\cdot}\right)$
possesses the non-zero center $\K \cdot y_5$, but it is perfect, thus not
reductive.
\end{exa}

An interesting vector subspace of $\End(\g)$ is the centroid of $\g$.
\begin{defi}
The \emph{commutant of the adjoint representation} or \emph{centroid} of a Lie algebra $\g$ is defined as follows:
\es{
\Cent(\g) :&= \set{\left.f \in \End(\g)\right| \lie{f,\ad_x}= 0 \text{ for all } x \in \g} = \set{\left.f \in \End(\g)\right| f \circ \ad_x = \ad_x \circ f \text{ for all } x \in \g}\\
&= \set{\left.f \in \End(\g)\right| f\lie{x,y} = \lie{x,f(y)} \text{ for all } x,y \in \g}.
}
Note that for $f \in \Cent(\g)$ and $x,y \in \g$ we have:
\[
\lie{f(x),y}=-\lie{y,f(x)}=-f\lie{y,x}=f\lie{x,y}=\lie{x,f(y)}.
\]
\end{defi}

For a finer analysis of the structure of $\Cent(\g)$, we introduce special elements in an endomorphism algebra.
\begin{defi} Let $\End(V)$ be the endomorphism algebra of a \vs $V$ and $\g$ a Lie algebra.
\begin{enumerate}
\item A map $f \in \End(V)$ is \emph{nilpotent}, if there exists $n \in \N$ such that $f^n=0$.\footnote{Note that in the case of $\dim V = d < \infty$ this is equivalent to $f^d = 0$.} A map $f \in \End(V)$ is \emph{semisimple}, if for each $f$-invariant subspace $W \sbs V$ there exists a complementary $f$-invariant subspace $W' \sbs V$ such that $W \oplus W' = V$.
\item We define the subsets $\operatorname{N}(\g):=\set{\left.f \in \Cent(\g)\right|f \text{ nilpotent}}$, $\operatorname{S}(\g):=\set{\left.f \in \Cent(\g)\right|f \text{ semisimple}}$ and the vector subspace $\operatorname{J}(\g):= \set{\left.{\ph} \in \Cent(\g) \right| \ad_x \circ {\ph} = 0 = {\ph} \circ \ad_x \text{ for all } x \in \g}$.
\end{enumerate}
\end{defi}

\begin{lem}\label{Liealglemma}
Let $\g$ be a \La.
\begin{enumerate}
\item $\Cent(\g)$ is an associative subalgebra of $\End(\g)$.
\item We have the inclusions $\Cent(\g) \circ \Der(\g) \sbs \Der(\g)$ and $\lie{\Cent(\g), \Der(\g)} \sbs \Cent(\g)$ or, to say it in other words, $\Der(\g)$ is a $\Cent(\g)$-module and $\Der(\g)$ acts by module derivations on $\Cent(\g)$.
\item We have the inclusion $\lie{\Cent(\g),\Cent(\g)} \sbs \Hom(\g / \lie{\g,\g}, \z(\g))$. In particular, if $\g$ is perfect or centerfree, then $\Cent(\g)$ is abelian. In the latter case, if $\g$ is of finite dimension, then $\operatorname{N}(\g), \operatorname{S}(\g)$ are associative subalgebras of $\Cent(\g)$ with $\Cent(\g)=\operatorname{N}(\g) \oplus \operatorname{S}(\g)$.
\item If $\g$ is of dimension $d \in \N$ and semisimple, then $\operatorname{N}(\g)=0$, thus $\Cent(\g)=\operatorname{S}(\g)$. 
\end{enumerate}
\end{lem}

\begin{proof}\
\begin{enumerate}
\item Clearly, $\Cent(\g)$ is a vector subspace of $\End(\g)$. If $f,g \in \Cent(\g)$, then for all $x,y \in \g$ we have:
\[
(f  g)\lie{x,y} = f\left(g\lie{x,y} \right)=f\left( \lie{x,g(y)}\right) = \lie{x,f(g(y))} = \lie{x,(f  g)(y)}.
\]
So, $f g$ is also in $\Cent(\g)$. Therefore $\Cent(\g)$ is an associative subalgebra of $\End(\g)$.
\item Let $f \in \Cent(\g)$, $g \in \Der(\g)$ and $x,y \in \g$. Then:
\es{
\lie{f,g}\lie{x,y} &= f(g\lie{x,y}) - g(f\lie{x,y}) = f(\lie{gx,y}+\lie{x,gy}) - g(\lie{x,fy})\\
&= f(\lie{gx,y}) +f(\lie{x,gy}) - \lie{gx,fy} - \lie{x,gfy} = \lie{gx,fy} +\lie{x,fgy} - \lie{gx,fy} - \lie{x,gfy}\\
&=\lie{x,fgy} - \lie{x,gfy} = \lie{x,fgy - gfy} = \lie{x,\lie{f,g}(y)}.
}
This implies $\lie{\Cent(\g), \Der(\g)} \sbs \Cent(\g)$. We may also calculate:
\es{
(fg)\lie{x,y}= f(\lie{gx,y}+\lie{x,gy}) = \lie{(fg)x,y}+\lie{x,(fg)y}.
}
We thus get $\Cent(\g) \circ \Der(\g) \sbs \Der(\g)$.
\item Let $f,g \in \Cent(\g)$ and $x,y \in \g$. We calculate:
\es{
\lie{\lie{f,g}(x),y}&=\lie{f(gx)-g(fx),y}=f\lie{gx,y}-g\lie{fx,y}=fg\lie{x,y}-gf\lie{x,y}\\
&=\lie{fx,gy}-\lie{fx,gy}=0.
}
Therefore $\lie{f,g}\left(\g\right) \sbs \z(\g)$. If $z \in \lie{\g,\g}$ is arbitrarily chosen, then there exist finitely many elements $x_1, y_1, x_2, y_2, \ldots , x_n, y_n \in \g$ such that $z=\sum_{i=1}^n \lie{x_i,y_i}$ and thus we obtain
\es{
\lie{f,g}(z)=\sum_{i=1}^n fg \lie{x_i,y_i}-gf \lie{x_i,y_i}=0
}
implying $\lie{f,g}\left(\lie{\g,\g}\right) = 0$ and so $\lie{f,g}$ identifies with a \vs morphism $\g / \lie{\g,\g} \to \z(\g)$.

Now let $\Cent(\g)$ be abelian and $\g$ of finite dimension. Obviously, the subsets $\operatorname{N}(\g), \operatorname{S}(\g)$ are closed under the multiplication by scalars in $\K$.

Let $f,g \in \operatorname{N}(\g)$ and $n,m \in \N$ such that $f^n = g^m = 0$. Obviously, $f \circ g$ is also nilpotent. Since $\Cent(\g)$ is abelian, we have $\lie{f,g} = 0$ and so we can apply the Binomial Theorem, yielding
\begin{align*}
(f+g)^{n+m+1} &= \sum_{k=0}^{n+m-1} \begin{pmatrix}n+m-1\\k\end{pmatrix} f^k g^{n+m-1-k} \\
&=
\sum_{k=0}^{n-1} \begin{pmatrix}n+m-1\\n-1-k\end{pmatrix} f^{n-1-k}g^{m+k} +
\sum_{k=0}^{m-1} \begin{pmatrix}n+m-1\\n+k\end{pmatrix} f^{n+k}g^{m-k-1}
= 0+0 =0.
\end{align*}
Thus $f+g$ is also nilpotent. So $\operatorname{N}(\g)$ is an associative subalgebra of $\Cent(\g)$.

Let $f,g \in \operatorname{S}(\g)$. We want to show that $f+g$ and $f \circ g$ are also semisimple. If $\K = \R$, then we use Corollary V.2.8 of \cite{Ne94} to say that $h \in \End(\g)$ would be semisimple \iff its complexification $h_{\C} := \id_{\C \otimes \g} \otimes h$ was semisimple. So we only consider the case $\K = \C$. By Lemma V.3.3 of \cite{Ne94}, ``semisimple'' means the same as ``diagonalizable'' for $\K = \C$. Since $\Cent(\g)$ is abelian, we have $\lie{f,g} = 0$ and so we can apply a theorem about simultaneous diagonalzation (cf. e.g. Theorem 11.B.15 of \cite{StWi90}): Diagonalizable endomorphisms of a vector space of finite dimension are simultaneously diagonalizable \iff they commute.
So $f$ and $g$ are diagonal \wrt a fixed basis, thus also $f+g$ and $f \circ g$, so they are semisimple, too. This shows that $\operatorname{S}(\g)$ is an associative subspace of $\Cent(\g)$.

The Jordan Theorem V.3.5 of \cite{Ne94} yields that each element $f\in \Cent(\g)$ decomposes into the sum of a nilpotent \vs endomorphism $f_n$ and a semisimple one $f_s$ and because of the fact that $f$ commutes with any $\ad_x$ for $x \in \g$ this is
 also satisfied for $f_n$ and $f_s$. Since $\Cent(\g)$ is abelian, any two endomorphisms $n,s$ with $f=n+s$, where $n \in \operatorname{N}(\g)$ and $s \in \operatorname{S}(\g)$, commute with $f$ and so the theorem also implies that $n=f_n$ and $s=f_s$. This shows $\Cent(\g)=\operatorname{N}(\g) \oplus \operatorname{S}(\g)$.

\item If $f \in \operatorname{N}(\g)$, then $f^d(\g) = 0$. On the other hand, $f(\g)$ is an ideal of $\g$ because for $x,y \in \g$ we have $\lie{y,f(x)} = f\lie{x,y} \in f(\g)$. Since $f(\g)^n \sbs f^n(\g)$ for all $n \in \N$, this implies $f(\g)^d = 0$, so $f(\g)$ is a nilpotent ideal of $\g$. But $\g$ is semisimple, so this ideal is trivial and $f = 0$.
\end{enumerate}
\end{proof}

\begin{rem} \
\begin{enumerate}
\item Note that any finite-dimensional associative commutative $\C$-algebra $A$ with unity element 1 can be realized as the centroid of a perfect and centerfree $\C$-Lie algebra $\g$. We show this in several steps.
\begin{enumerate}
\item Let $\k$ be an finite-dimensional simple $\C$-Lie algebra, e.g. $\k =
\sl_2(\C)$. Then it is central, i.e. \es{\Cent(\k)=\End_{\k}(\k) = \K \cdot \1,}
by the Schur Lemma. We define $\g:= \k \otimes A$, understood as the tensor
product of \Kvs{s}. By setting $\lie{(x \otimes a),(y \otimes b) }:=\lie{x,y}
\otimes ab$ for all generating elements $x,y \in \k$ and $a,b \in A$ and
extending $\lie{\cdot,\cdot}$ to a $\C$-bilinear map $\g \times \g \to \g$, we
obtain a Lie algebra $(\g,\lie{\cdot,\cdot})$. \item Since $\k$ is simple, it is
perfect. This implies that  $\g$ is also perfect: An element $v \in \g$ takes
the form $v=\sum_{i=1}^n x_i \otimes a_i$ for certain $x_1, \ldots , x_n \in \k$
and $a_1, \ldots , a_n \in A$ and each $x_i$ can be written as $x_i = \sum_{j
=1}^{r_i} \lie{y_{ij}, z_{ij}}$ for some $y_{i1}, z_{i1}, \ldots , y_{ir_i},
z_{ir_i} \in \k$. So we have \[v =\sum_{i=1}^n \left(\sum_{j =1}^{r_i}
\lie{y_{ij}, z_{ij}}\right) \otimes a_i = \sum_{i=1}^n \sum_{j =1}^{r_i}
\lie{y_{ij} \otimes a_i, z_{ij}\otimes 1}. \] \item Furthermore, $\g$ is
centerfree because $\k$ is so: If $v = \sum_{j=1}^m x_j \otimes a_j \in \z(\g)$
for some $x_1, \ldots , x_m \in \k$ and $a_1, \ldots , a_m \in A$, then, for all
$x \in \k$, we have: \es{ 0=\lie{v,x \otimes 1}=\sum_{j=1}^m \lie{x_j \otimes
a_j, x \otimes 1}=\sum_{j=1}^m \lie{x_j,x} \otimes a_j. } But this is only
possible, if, for all $j \in \set{1, \ldots , m}$, we have $a_j = 0$ or $x_j \in
\z(\k) =0$, thus $v=0$. \item We may consider $\k \sbs \g$ by the embedding $\k
\hookrightarrow \g$, $x \mapsto x \otimes 1$ and $\g$ becomes a $\k$-module by
restriction of the adjoint representation. By \cite{Ne02}, there are $x_1,
\ldots , x_n, x^1, \ldots , x^n \in \k$ such that $\sum_{i=1}^n \ad(x_i) \circ
\ad(x^i)$ is an element in $\Cent(\k)$ and, by \cite{Hu72} on page 122, it is
equal to $\lambda \cdot \1$ for some $\lambda \in \C \backslash \set{0}$ and so
we may assume, after normalizing: \es{\sum_{i=1}^n \ad(x_i) \circ \ad(x^i) =
\1.} If $a \in A$ is an arbitrary element, then $f_a:=\sum_{i=1}^n \ad(x_i
\otimes a) \circ \ad(x^i \otimes a)$ is in $\Cent(\g)$ and $f_a(x \otimes b) = x
\otimes ab$ for all $x \in \k$ and $b \in A$. So $\Cent(\g)=\End_{\g}(\g)$ is
generated by endomorphisms of the type $\ad(x) \otimes \1$, where $x \in \k$,
and those of the type $f_a$, where $a \in A$. We obtain: \es{ \Cent(\g) =
\End_{\g}(\g) \cong \End_{\k}(\k) \otimes A \cong \K \otimes A \cong A. }
\end{enumerate}

\item
If $\g$ is an arbitary Lie algebra, then the subsets $\operatorname{N}(\g), \operatorname{S}(\g) \sbs \Cent(\g)$ are not necessarily closed under the addition in $\End(\g)$: Let $\g$ be the abelian complex Lie algebra $\C^2$. Then $\Cent(\g)=\End(\g)$ is, as associative algebra, isomorphic to $M_{2,2}(\C)$. We can write
\[
\begin{pmatrix}0 & 1 \\ 1 & 0 \end{pmatrix} = \begin{pmatrix}0 & 1 \\ 0 & 0 \end{pmatrix} + \begin{pmatrix}0 & 0 \\ 1 & 0 \end{pmatrix}
\]
and the two matrices on the right hand site are nilpotent, but the matrix on the left hand site is not.
And we can write
\[
\begin{pmatrix}0 & 2 \\ 0 & 0 \end{pmatrix} = \begin{pmatrix}0 & 1 \\ 1 & 0 \end{pmatrix} + \begin{pmatrix}0 & 1 \\ -1 & 0 \end{pmatrix}
\]
and the two matrices on the right hand site are semisimple, but the matrix on the left hand site is not.
\item
Let $\g$ and $\h$ be \La{s}. The quotient Lie algebra $\g / \lie{\g,\g}$ and the
 Lie subalgebra $\z(\h) \sbs \h$ are abelian and so the Lie algebra 
morphisms $\g / \lie{\g,\g} \to \z(\h)$ are just the linear maps $\g /
\lie{\g,\g} \to \z(\h)$. In particular, we have $\Hom(\g / \lie{\g,\g}, \z(\h))
= 0$ \iff $\g$ is perfect or $\h$ is centerfree.

There is a natural isomorphism between linear maps $\overline \ph: \g / \lie{\g,\g} \to \z(\g)$ and linear maps  ${\ph}: \g \to \z(\g)$ with $\ad_x \circ {\ph} = 0 = {\ph} \circ \ad_x$ for all $x \in \g$: Given $\overline{\ph}$, we set $\ph(x):=\overline{\ph}(x+\lie{\g,\g})$. Given $\ph$, the map $\overline \ph$ is well-defined by $\overline \ph(x+\lie{\g,\g}):=\ph(x)$ because $x+\lie{\g,\g} = y + \lie{\g,\g}$ implies the existence of $a_1, b_1, \ldots , a_n, b_n \in \g$ such that $x-y = \sum_{i=1}^n\lie{a_i,b_i}$ and thus \es{\ph(x)-\ph(y)=\sum_{i=1}^n\ph{\lie{a_i,b_i}}=\sum_{i=1}^n \ph \circ \ad_{a_i} (b_i) = 0.} We obtain:
\es{
\Hom(\g / \lie{\g,\g},\z(\g)) \cong \set{\left.{\ph} \in \Hom(\g,\z(\g)) \right| \ad_x \circ {\ph} = 0 = {\ph} \circ \ad_x \text{ for all } x \in \g} = \operatorname{J}(\g).
}
\item Let $\g$ be a Lie algebra such that $\z(\g) \sbs \lie{\g,\g}$. For all $f \in \Cent(\g)$, $g \in \operatorname{J}(\g)$, $x,y \in \g$ we have:
\es{
fg(\lie{x,y}) &= f(0) = 0\\
\lie{fg(x),y} &= f\lie{g(x),y}=f(0)=0.\\
gf\lie{x,y}&=g\lie{f(x),y}=0\\
\lie{gf(x),y}&=\lie{g(f(x)),y}=0\\
g^2(x)=g(g(x)) &\in g(\z(\g)) \sbs g(\lie{\g,\g})=0.
}
Thus, in this case, $\operatorname{J}(\g)$ is a two-sided ideal of $\Cent(\g)$ with $\operatorname{J}(\g)^2=0$.
\end{enumerate}
\end{rem}

\begin{defi}
A \La $\g$ is called \emph{decomposable} if it is the direct sum of two proper
ideals. If there is no such decomposition, $\g$ is called \emph{indecomposable}.
\end{defi}

\begin{rem}
A simple \La is automatically indecomposable, since 0 is its only proper
ideal. If an indecomposable \La $\g$ is reductive, then it is one-dimensional
and trivial or it is simple because
$\g=\z(\g) \oplus \lie{\g,\g}$ implies $\g=\z(\g)$ or $\g=\lie{\g,\g}$ and in the latter case
a decomposition of the semisimple Lie algebra $\lie{\g,\g}$ into simple Lie algebras may only have one summand.
In general, indecomposable \La{}s are not reductive, e.g. the
Lie algebra $\g$ in Example \ref{nonreduLa}.
\end{rem}

\begin{lem}\label{indecomposablemma}
Let $\g$ be a \La of finite dimension with abelian centroid, e.g. $\g$ is perfect or centerfree. Then there exists a decomposition $\g = \bigoplus_{i=1}^n \g_i$ into a direct sum of non-zero ideals $\g_i$, where each $\g_i$ is an indecomposable \La and this decomposition is unique except for the order.
\end{lem}
\begin{proof}
The existence of a decomposition into indecomposable non-zero ideals is proven by mathematical induction over the finite dimension of $\g$. The only difficult part is the one of the uniqueness: Let $\g=\bigoplus_{i=1}^n \g_i=\bigoplus_{j=1}^m \h_j$ be two such decompositions with corresponding systems of orthogonal projections $p_1, \ldots , p_n$ and $q_1, \ldots , q_m$. The projections are in $\Cent(\g)$, as we can show with the following easy calculation:
\es{
p_i(\lie{x,y})=p_i(\lie{x_i,y_i}+\lie{x',y'})=\lie{x_i,y_i}=\lie{x_i+x',y_i}=\lie{x,p_i(y)},
} where $z=z_i + z'$ is the decompositions of an element $z \in \g$ into $\g_i$ and its complement \wrt the decomposition $\bigoplus_{i=1}^n \g_i$. Since $\Cent(\g)$ is abelian, we have $p_i \circ q_j = q_j \circ p_i$ for $i \in \set{1, \ldots , n}$ and $j \in \set{1, \ldots , m}$. Now fix an $i$. Then $p_i$ induces a projection $\h_j \to \h_j$ for each $j \in \set{1, \ldots , m}$ and so, by the indecomposability of the $\h_j$, we obtain $p_i = 0$ on $\h_j$ or $p_i = \1$ on $\h_j$. But from this we deduce that $\g_i=\im(p_i)$ is the direct sum of some $\h_j$'s. We conclude, by the indecomposability of $\g_i$, that there is exactly one $j \in \set{1, \ldots , m}$ such that $\g_i = \h_j$.
\end{proof}

The following theorem is due to Médina and Revoy, cf. \cite{MeRe93}, and gives us information about the centroids of Lie algebras which are not necessarily perfect or centerfree.

\begin{thm}\label{MedinaRevoy}
Let $\g$ be a Lie algebra of finite dimension such that $\z(\g) \sbs \lie{\g,\g}$.
\begin{enumerate}
\item There are indecomposable idempotents $p_1, \ldots , p_n \in \Cent(\g)$ which are pairwise orthogonal,  i.e. $p_i \circ p_j = \delta_{ij} p_i$ for $i,j \in \set{1, \ldots , n}$, satisfying  $\sum_{i=1}^n p_i = \1$.
\item Setting $\g_i:=p_i(\g)$ for $i \in \set{1, \ldots , n}$ we have a decomposition $\g=\bigoplus_{i=1}^n \g_i$ into indecomposable non-zero ideals.
\item Each $\Cent(\g_i)$ is isomorphic to $p_i \circ \Cent(\g) \circ p_i$, its quotient $\Cent(\g_i)/\operatorname{J}(\g_i)$ modulo the maximal nilpotent ideal $\operatorname{J}(\g_i)$ is a field.
\item Setting $C_{ij}:=p_i \circ \Cent(\g) \circ p_j$ for $i,j \in \set{1, \ldots , n}$ we have the decomposition $\Cent(\g)=\bigoplus_{i,j=1}^n C_{ij}$ into ideals, each $C_{ij}$ is a $\Cent(\g_i)$-$\Cent(\g_j)$-bimodule and, if $i \ne j \in \set{1, \ldots , n}$, then the vector space $\Hom(\g_j / \lie{\g_j,\g_j}, \z(\g_i))$ is isomorphic to $C_{ij}$. Furthermore, $\operatorname{J}(\g)=\bigoplus_{i=1}^n \operatorname{J}(\g_i) \oplus \bigoplus_{i\ne j =1}^n C_{ij}$ is also a decomposition into ideals.
\end{enumerate}
\end{thm}

We also need the following result, Propsition 22.1 of \cite{La01}, to show a more general proposition about the uniqueness of the decomposition of a finite-dimensional Lie algebra into indecomposable non-zero ideals.
\begin{thm}\label{Lam} Let $R$ be a ring with unity element 1. Suppose there exists a decomposition of $1 \in R$ into a sum of indecomposable idempotents, say $1=c_1 + \ldots + c_r$, such that $c_i c_j = \delta_{ij} c_i$ and $c_i \in Z(R)$ for all $i,j \in \set{1, \ldots , n}$. Then the decomposition is unique except for the order.
\kommentar{%
Let $I \ide R$ be a nilpotent ideal, $\pi: R \to R/I := \overline R$, $x \mapsto \overline x$ the canonical projection and $e \in R$ an idempotent. Then $e$ is indecomposable in $R$ \iff $\overline e$ is indecomposable in $\overline R$.
}%
\end{thm}

\begin{prop}\label{uniquedecompo}
Let $\g$ be a finite-dimensional \La such that $\z(\g) \sbs \lie{\g,\g}$ with a decomposition into indecomposable non-zero ideals $\bigoplus_{i=1}^n \g_i$ such that $\Hom(\g_i / \lie{\g_i,\g_i} , \z(\g_j)) = 0$ if $i \ne j \in \set{1, \ldots , r}$, e.g. $\g$ is reductive and $\dim \z(\g) \le 1$. Then this decomposition is unique except for the order.
\end{prop}

\begin{proof}
By Theorem \ref{MedinaRevoy}, we have $R:=\Cent(\g) \cong \bigoplus_{i=1}^n \Cent(\g_i)$, $\operatorname{J}(\g) = \bigoplus_{i=1}^n \operatorname{J}(\g_i)$ and the quotient $\Cent(\g) / \operatorname{J}(\g) \cong  \bigoplus_{i=1}^n \Cent(\g_i) / \operatorname{J}(\g_i)$ is commutative. Now suppose there are two decompositions of $\g$ into indecomposable non-zero ideals, say $\g=\bigoplus_{i=1}^n \g_i = \bigoplus_{j=1}^m \h_j$ with corresponding systems of orthogonal projections $p_1, \ldots p_n$ and $q_1, \ldots , q_m$. Let $f \in \Cent(\g)$ and $f=f_1 + \ldots + f_n$ its decomposition into elements of the $\Cent(\g_i) \sbs \Cent(\g)$. Then, for fixed $i \in \set{1, \ldots , n}$, we have
\es{
p_i \circ f = \sum_{k=1}^n p_i \circ f_k = p_i \circ f_i = f_i = f_i \circ p_i = \sum_{k=1}^n f_k \circ p_i = f \circ p_i.
}
So the $p_i$ are in $Z(R)$ and, by a dual argument, all the $q_j$ are so. By Theorem \ref{Lam}, this implies $\set{p_1, \ldots , p_n}=\set{q_1, \ldots , q_m}$, thus $\set{\g_1, \ldots , \g_n}=\set{\h_1, \ldots , \h_m}$.
\kommentar{%
By the facts that $\operatorname{J}(\g) \ide \Cent(\g)$ is nilpotent and $\Cent(\g) / \operatorname{J}(\g)$ is commutative and by Theorem \ref{Lam}, we have $m=n$ and there is   $\sigma \in S_n$ such that $\overline{p_i}=\overline{q_{\sigma(i)}}$, i.e. there is $\ph_i \in \operatorname{J}(\g)$ with $p_i = q_{\sigma(i)} + \ph_i$, for each $i\in \set{1, \ldots , n}$. Since $\Hom(\g_i / \lie{\g_i,\g_i} , \z(\g_j)) = 0$ if $i \ne j \in \set{1, \ldots , n}$, there cannot be two distinct indices $i \ne j \in \set{1, \ldots , n}$ such that $\operatorname{J}(\g_i)$ and $\operatorname{J}(\g_j)$ both are non-trivial. Therefore there is an index $k\in \set{1, \ldots , n}$ with $\ph_i \in \operatorname{J}(\g_k)$ for each $i \in \set{1, \ldots , n}$. This yields $\h_{\sigma(i)} = \im q_{\sigma(i)} = \im p_i + \im \ph_i = \g_i + \im \ph_i \sbs \g_i + \z(\g_k)$ for all $i \in \set{1, \ldots , n}$ and, in particular, $\h_{\sigma(k)} = \g_k$. Since $\h_i \cap \h_j = 0$ for $i \ne j \in \set{1, \ldots , n}$, we may conclude $\h_{\sigma(i)} = \g_i$ for all $i \in \set{1, \ldots , n}$.
}%
\end{proof}

\begin{rem} In general, the decomposition of a finite-dimensional Lie algebra $\g$ into indecomposable non-zero ideals is not unique and the requirements of the preceding proposition are sharp: Let $\g = \g_1 \oplus \g_2$ be a decomposition into indecomposable non-zero ideals with corresponding indecomposable projections $p_1$, $p_2$ and $\Hom(\g_1 / \lie{\g_1,\g_1},\z(\g_2))\ne0$. If $0 \neq \ph: \g \to \g$ is a morphism mapping $\g_1$ to $\z(\g_2)$ and the spaces $\lie{\g_1,\g_1}$ and $\g_2$ to $0$, then $p_1':=(\1+\ph) \circ p_1 \circ (\1+\ph)^{-1}=p_1+\ph$ is also an indecomposable idempotent of $\Cent(\g)$. Thus $\g=\im p_1' + \im (\1 - p_1')$ is another decomposition of $\g$ into indecomposable non-zero ideals.
\end{rem}

\begin{rem}\label{Hopen}
If $\g_1, \ldots , \g_n$ are the unique indecomposable non-zero ideals of a Lie algebra $\g$ and $f \in \Aut(\g)$ is a Lie algebra automorpism, then of course $\g_1, \ldots , \g_n$ are the unique indecomposable non-zero ideals of $f(\g)$ and thus there is a permutation $\sigma \in S_n$ such that $f(\g_i) = \g_{\sigma(i)}$ for all $i \in \set{1, \ldots, n}$.
\end{rem}

\begin{lem}\label{indecomplexreallemma}
Let $\g$ be a \La of finite dimension.
\begin{enumerate}
\item If $\K = \C$, then $\g$ is indecomposable \iff $\operatorname{S}(\g)=\C \cdot \1$.
\item If $\K = \R$, then $\g$ is indecomposable \iff either $\operatorname{S}(\g)=\R \cdot \1$ or $\operatorname{S}(\g)=\R \cdot \1 + \R \cdot J$ for a complex structure $J$ on $\g$, i.e $J^2=-\1$. In the second case, $J$ and $-J$ are the only complex structures on $\g$ compatible with the Lie bracket.
\end{enumerate}
\end{lem}
\begin{proof} \
\begin{enumerate}
\item For any $f \in \Cent(\g)$ the eigenspaces of $f$ are ideals of $\g$: If $y \ne 0$ is an eigenvector of $f$ for the eigenvalue $\lambda$ and $x \in \g$ then $f\lie{x,y}=\lie{x,fy}=\lie{x,\lambda y}=\lambda \lie{x,y}$, i.e. $\lie{x,y}$ is also an eigenvector of $f$ for the eigenvalue $\lambda$.
If $\K = \C$, then the semisimple endomorphisms in $\Cent(\g)$ are exactly the diagonalizable endomorphisms in $\Cent(\g)$. Thus $\g$ is indecomposable \iff each endomorphism in $\operatorname{S}(\g)$ has exactly one eigenvalue \iff $\operatorname{S}(\g)=\C \cdot \1$.
\item
If $\K = \R$ and $\g$ is decomposable into the direct sum of non-trivial ideals $\g_1$ and $\g_2$ with corresponding orthogonal projections $p_1$ and $p_2$, then $p_1 p_2=0$ although $p_1 \ne 0 \ne p_2$, thus $\Cent(\g)$ has zero-divisors and is neither isomorphic as associative $\R$-algebras to $\R$ nor to $\C$.

If $\K = \R$ and $\g$ is indecomposable, then any endomorphism in $\operatorname{S}(\g)$ admits exactly one real eigenvalue or exactly two non-real, complex conjugate eigenvalues. If every endomorphism in $\operatorname{S}(\g)$ admits exactly one real eigenvalue, then $\operatorname{S}(\g)=\R \cdot \1$. In the other case, if any $f \in \operatorname{S}(\g)$ admits the non-real complex eigenvalues $a+ib$ and $a-ib$, then $J:=\frac 1 b (f - a\1) \in \operatorname{S}(\g)$ is a complex structure of $\g$ because of the following calculation for the eigenvector $x \in \g$ for the eigenvalue $a\pm ib$:
\es{
J^2(x)&=(\frac {1} {b^2} (f - a\1)^2) (x) = \frac{1}{b^2} (f^2(x) - 2af(x) + a^2x)\\
&= \frac{1}{b^2} ((a\pm ib)^2 x -2a(a \pm ib)x +a^2x ) = \frac{1}{b^2} (a^2 \pm 2iab - b^2 - 2a^2 \mp 2iab + a^2) (x)\\
&= -x.
} The space of semisimple endomorphims in $\Cent(\g)$ in this case is the space of real linear combination of $\1$ and $J$, thus isomorphic to $\C$ as associative algebras over $\R$. Since the complex conjugation is the only $\R$-linear algebra automorphism on $\C$ different from the identity, the only complex structures on $\operatorname{S}(\g)$ are $J$ and $-J$.
\end{enumerate}
\end{proof}

\subsection{Fiber bundles}
The definitions of various bundles in this subsection are based on \cite{tD91},
\cite{Hu75} and \cite{BiCr01}. Our notion ``bundle'' is what sometimes is called
``fiber bundle'', i.e. the fibers of a bundle are all ``of the same type''.

\begin{defi} \
\begin{enumerate}
\item Let $E,M,F$ be manifolds such that there is surjective submersion $\pi: E \to M$ and each fiber $E_x:=\pi^{-1}(x)$ for $x \in M$ is diffeomorphic to $F$. If there is an open cover $\U=\left(U_i\right)_{i \in I}$ of $M$ and a family of smooth maps $\Phi=\left(\ph_i\right)_{i \in I}$, each $\ph_i: \pi^{-1}(U_i) \to F$ inducing a \diffem $\left(\pi,\ph_i\right): \pi^{-1}(U_i) \to U_i \times F$ (\emph{local triviality}), then the sextuple $\left(E,M,\pi,F,\U,\Phi\right)$ is called a \emph{bundle} with \emph{bundle space} $E$ over the \emph{base space} $M$ with \emph{projection} $\pi$, \emph{fiber} $F$ and \emph{bundle atlas} $(\U,\Phi)$. A pair $(U_i, \ph_i)$ is called \emph{bundle chart}. As a shorter formulation we will say ``$\pi: E \to M$ is a bundle with fiber $F$'' without explicit mention of the bundle atlas.
\item A bundle $\pi': E' \to M'$ is a \emph{subbundle} of a bundle $\pi: E \to M$ provided $E' \sbs E$ and $M' \sbs M$ and $\pi'(x)=\pi(x)$ for all $x \in E'$.
\end{enumerate}
\end{defi}

\begin{rem}
Since we only consider paracompact manifolds as base spaces, we may always assume that the bundle atlas of a bundle and a corresponding\footnote{The cover of the base space is the same.} atlas of the base spaces are locally finite.
\end{rem}

For each bundle there is a class of interesting functions, the so-called sections.

\begin{defi}
Let $\pi: E \to M$ be a bundle and $k \in \Ni$. A $\Ck$-map
\es{
X: M &\lra E\\
x &\longmapsto X(x) = X_x
}
is called $\Ck$-\emph{section} of the bundle, if $\pi \circ X = \id_M$. The set of all $\Ck$-sections of the bundle $\pi: E \to M$ is denoted by $\Gamma^k(E)$ and we also write $\Gamma(E)$ instead of $\Gamma^\infty(E)$.
\end{defi}

We now give a first definition of a vector bundle.

\begin{defi}\label{Defvb1}
Let $V$ be a finite-dimensional \vs and $\pi: E \to M$ a bundle with fiber $V$ with a \vs structure on each fiber $E_x$ for $x \in M$, isomorphic to $V$ by $\left(\ph_i\right)_{|_{E_x}}: E_x \to V$ for each $i \in I$. Then the sextuple $\left(E,M,\pi,V,\U,\Phi\right)$ is called a \emph{vector bundle}. As a shorter formulation we will say ``$\pi: E \to M$ is a vector bundle with fiber $V$''.
\end{defi}

\begin{defi}\label{bundlemorph} Let $\pi_i: E_i \to N$ be a vector bundle with fiber $V_i$ for $i=1,2$.
\begin{enumerate}
\item A \emph{vector bundle morphism} from $\pi_1: E_1 \to N$ to $\pi_2: E_2 \to N$ is a map $\kappa: E_1 \to E_2$ such that the following diagram
\kD{
E_1 \ar[r]^\kappa \ar[d]_{\pi_1} & E_2\ar[dl]^{\pi_2}\\
N
}
commutes and $\kappa: (E_1)_x \to (E_2)_x$ is a \vs morphism for all $x \in N$. The space of vector bundle morphisms from $\pi_1: E_1 \to N$ to $\pi_2: E_2 \to N$ is denoted by $\Hom(E_1,E_2)$. A vector bundle morphism from $\pi_1: E_1 \to N$ to $\pi_1: E_1 \to N$ is also called \emph{vector bundle endomorphism} of $\pi_1: E_1 \to N$.
\item A vector bundle morphism is called \emph{vector bundle map}, if fibers are bijectively mapped on each other and, thus, the fibers are isomorphic \vs{s}.
\item A vector bundle morphism $\kappa$ is a \emph{vector bundle isomorphism}, if there exists a vector bundle morphism $\iota \in \Hom(E_2,E_1)$, such that $\kappa$ is inverse to $\iota$. In this case the vector bundles $\pi_1: E_1 \to N$ and $\pi_2: E_2 \to N$ are called \emph{isomorphic}. Note that vector bundle isomorphisms are vector bundle maps.
\end{enumerate}
\end{defi}

A second definition of a vector bundle needs the definition of a principal bundle and a bundle associated to a principal bundle.

\begin{defi}
Let $G$ be a Lie group and $\pi: P \to M$ a bundle with fiber $G$ such that $\pi$ is identic to the canonical projection of a smooth right action
\es{
R: P \times G &\lra P\\
(p,g) &\longmapsto p\cdot g = R_g(p)=R^p(g).
}
If the following diagram commutes for all $i \in I$, $g \in G$
\kD{
\pi^{-1}(U_i) \ar[r]^{R_g} \ar[d]_{\ph_i} & \pi^{-1}(U_i)\ar[d]^{\ph_i}\\
G \ar[r]_{\rho_g}& G
,
}
then the septuple $\left(P,M,\pi,G, R,\U,\Phi\right)$ is called a \emph{principal bundle} with \emph{structure group} $G$ and \emph{bundle action} $R$. As a shorter formulation we will say ``$\pi: P \to M$ is a $G$-principal bundle'' without explicit mention of the bundle action.
\end{defi}

\begin{rem} Let $\pi: P \to M$ be a $G$-principal bundle.
\begin{enumerate}
\item $G$ acts freely on $P$ by $R$, i.e. $p \cdot g = p$ for some $p \in P$, $g \in G$ implies $g=1$. To see this, let $p \cdot g = p$ for some $p \in P$, $g \in G$ and let $(U,\ph)$ be a bundle chart in $x:=\pi(p)$. We calculate:
\es{
\ph(p) = \ph(p \cdot g) = \ph(p) g \quad \implies \quad 1=g.
}
It is easy to show now that $R$ acts simply transitively on its orbits, i.e. for all $p \in P$ the map $R^p: G \to P$ is injective.
\item The condition that the surjective submersion $\pi: P \to M$ is identic to the canonical projection $\pi^R: P \to P/G$ of the action $R$ can be weakened in the following sense: If they both are surjective submersions, then $M$ and $P/G$ are already \diffec. To show this, let $h$ be a map $P/G \to M$ such that the following diagram commutes:
\kD{
P\ar[r]^{\pi} \ar[d]_{\pi^R} & M\\
P/G\ar[ru]_h & .
}
The map $h$ is smooth because $\pi^R$ is a submersion and $h$ is surjective because $\pi$ is. It is even bijective because $h\left(\pi^R(p)\right)=h\left(\pi^R(q)\right)$ implies $\pi(p)=\pi(q)$ and there exists a bundle chart $(U, \ph)$ and an element $g \in G$ such that $(\pi(p),\ph(p))=(\pi(q),\ph(q)\cdot g)=(\pi(q),\ph(q\cdot g))$
yielding $p=q \cdot g$, thus $\pi^R(p)=\pi^R(q)$. Furthermore, for any $x \in M$ there is an open neighbourhood $U \sbs M$ such that there exists a section $X: U \to \pi^{-1}(U)$. Then $h^{-1}$ maps $y \in U$ on $\pi^R(X(y)) \in P/G$ and is smooth. So $h$ is a \diffem $P/G \to M$.
\end{enumerate}
\end{rem}

If $\pi: P \to M$ is a $G$-principal bundle and $F$ is a manifold on which $G$ acts from the left, then we can construct a new bundle with fiber $F$ by replacing the structure group $G$ by $F$.

\begin{defi}
Let $\pi: P \to M$ be a $G$-principal bundle and $F$ a manifold with a left action
\es{
L: G \times F &\lra F\\
(g,f) &\longmapsto g\cdot f = L_g(f)=L^f(g).
}
By defining a left action as follows
\es{
L': G \times (P \times F) &\lra P \times F\\
(g,(p,f)) &\longmapsto g\cdot (p,f) := \left(p \cdot g^{-1}, g \cdot f \right).
}
and considering the orbit space $B:=(P \times F)/G$, we obtain a bundle $\pi': B \to M$ with fiber $F$, called the \emph{bundle associated} to the $G$-principal bundle $\pi: P \to M$ with fiber $F$.\footnote{In this formulation we neglect the dependence of the bundle structure on the actions $R$ and $L$.} Let us examine its structure:
The bundle projection is
\es{
\pi': B &\lra M\\
[(p,f)]:=G \cdot (p,f) &\longmapsto \pi(p)
}
and this is well-defined because $\pi(p)=\pi(p \cdot g^{-1})$ for all $p \in P, g \in G$. The bundle atlas $(U_i, \ph_i)_{i \in I}$ of the $G$-principal bundle $\pi: P \to M$ induces a bundle atlas $(U_i, \ph'_i)_{i \in I}$ of $\pi': B \to M$ by
\es{
\ph'_i: \pi'^{-1}(U_i) &\lra F\\
[(p,f)] &\longmapsto \ph_i(p) \cdot f
}
and this is well-defined because $[(p,f)]=[(q,e)]$ yields the existence of $g \in G$ such that $p \cdot g^{-1}=q$ and $g \cdot f=e$ and then
$\ph_i(q) \cdot e= \ph_i(p \cdot g^{-1}) \cdot g \cdot f= \ph_i(p) \cdot g^{-1} \cdot g \cdot f =\ph_i(p) \cdot f.$ Finally, each fiber $B_x=\pi'^{-1}(x)$ for $x \in M$ is \diffec to $F$:
\es{
\pi'^{-1}(x) &= \set{\left. [(p,f)] \in (P \times F)/G\right| \pi(p)=x} = \set{\left. [(p,f)] \in (P \times F)/G\right| p \in P_x} = (P_x \times F)/G \cong F.
}
\end{defi}
Now we can define a vector bundle as being a bundle associated to a principal bundle.

\begin{defi}\label{Defvb2}
Given a finite-dimensional \vs $V$ and a Lie subgroup $G$ of $\GL(V)$ acting on $V$ \footnote{This condition may be replaced by the existence of an smooth group morphism $\Xi: G \to \GL(V)$ (a so-called Lie group representation) and $G$ acting on $V$ as follows: $g.v := \Xi(g)(v)$.} and a $G$-principal bundle $\pi: P \to M$, the bundle $\pi': B \to M$ associated to the $G$-principal bundle $\pi: P \to M$ with fiber $V$ is called a \emph{vector bundle}.
\end{defi}

\begin{rem} \label{VecBunEquiv}
\begin{enumerate}
\item A vector bundle $\pi': B \to M$ in the sense of Definition \ref{Defvb2} allows a \vs structure on each fiber $B_x$ for $x \in M$, defined by the condition that all restrictions $\left(\ph'_i\right)_{|_{B_x}}: B_x \to V$ are vector space isomorphisms and this is well-defined because
\es{
\left(\ph'_i\right)_{|_{B_x}}^{-1}\left(\lambda \ph'_i(G \cdot (p,v))+ \ph'_i(G \cdot (q,w)) \right)&=
   \left(\ph'_i\right)_{|_{B_x}}^{-1}\left(\lambda \ph_i(p)(v) + \ph_i(q)(w) \right)\\
&= \left(\ph'_i\right)_{|_{B_x}}^{-1}\left(\ph_i(p)(\lambda v + w) \right)\\
&= \left(\ph'_i\right)_{|_{B_x}}^{-1}\left(\ph'_i(G \cdot (p,\lambda v +w )) \right)\\
&=G \cdot (p,\lambda v +w)
}
does not depend on the choice of the bundle chart $(U_i, \ph'_i)$, where $x \in U_i$, $\lambda \in \K$, $G \cdot (p,v)$, $G \cdot (q,w) \in B_x$. Therefore every vector bundle in the sense of Definition \ref{Defvb2} is a vector bundle in the sense of Definition \ref{Defvb1}.

\item On the other hand, let $\pi: E \to M$ be a vector bundle with fiber $V$ in the sense of Definition \ref{Defvb1}. By setting
\es{
P:= \dot{\bigcup_{x \in M}} \Iso\left(V, E_x\right),
}
we obtain a surjective submersion $\rho: P \to M$ by setting $\rho(f):=x$ if $f \in \Iso\left(V, E_x\right)$, where the smooth structure on $P$ is induced by the bundle atlas $(U_i, \ph_i)_{i \in I}$ of the bundle $\pi: E \to M$ by declaring each map
\es{
\rho^{-1}(U_i) &\lra U_i \times \GL(V)\\
f &\longmapsto \left(x, \left(\ph_i\right)_{|_{E_{\rho(f)}}} \circ f\right)
}
to be a \diffem. We obtain a $\GL(V)$-principal bundle $\rho: P \to M$ with bundle action
\es{
P \times \GL(V) &\lra P\\
(f,T) &\longmapsto f \circ T
}
and we define $\rho': B=(P \times V)/\GL(V) \to M$ to be the bundle associated to this principal bundle with fiber $V$. It is isomorphic to the vector bundle $\pi: E \to M$ via
\es{
\kappa: B &\lra E\\
\GL(V) \cdot (f,v) &\longmapsto f(v).
}
We will see that, under certain conditions, the structure group may be reduced to a closed subgroup $G$ of $\GL(V)$.
\end{enumerate}
\end{rem}

\begin{defi}
Let $\pi': B \to M$ be a vector bundle with fiber $V$ associated to the $G$-principal bundle $\pi: P \to M$ with bundle atlas $(U_i, \ph'_i)_{i \in I}$ and $(U_i, \ph_i)_{i \in I}$, respectively. We also define an index set as follows: $\I:=\set{\left.(i,j) \in I \times I \right|U_i \cap U_j \neq \emptyset}$. For $(i,j) \in \I$ and $x \in U_i \cap U_j$ we define an element $g_{ji}(x) \in G$ by
\es{
\left(\pi,\ph_j \right) \circ \left(\pi,\ph_i \right)^{-1}: \left(U_i \cap U_j\right) \times G &\lra \left(U_i \cap U_j\right) \times G\\
(x,g) &\longmapsto \left(x,g_{ji}(x) g\right).
}
Note that, if $G \le GL(V)$ or $G$ is injectively representated in $GL(V)$, this element can also be defined by
\es{
\left(\pi',\ph'_j \right) \circ \left(\pi',\ph'_i \right)^{-1}: \left(U_i \cap U_j\right) \times V &\lra \left(U_i \cap U_j\right) \times V\\
(x,v) &\longmapsto \left(x,g_{ji}(x)(v)\right).
}
We obtain smooth functions $g_{ji}: U_i \cap U_j \to G$ for $(i,j) \in \I$, called \emph{transition maps}. The family $(g_{ji})_{(i,j) \in \I}$ is called \emph{cocyle} relative to the bundle atlas $(U_i, \ph_i)_{i \in I}$. As a shorter formulation we will say ``$\pi': B \to M$ is a vector bundle with fiber $V$ and cocycle $(g_{ji})_{(i,j) \in \I}$.'' We have:
\begin{equation}\label{cocyle}
g_{ji}(x) \cdot g_{ik}(x) \cdot g_{kj}(x) = 1 \text{ for all } x \in U_i \cap U_j \cap U_k.
\end{equation}
Given another bundle atlas $(U_i, \psi_i)_{i \in I}$ (the cover $\U$ is the same!) and the cocyle $(k_{ji})_{(i,j) \in \I}$ relative to this bundle atlas, we define smooth mappings $h_i: U_i \to G$ for $i \in I$ by
\es{
\left(\pi,\psi_i \right) \circ \left(\pi,\ph_i \right)^{-1}: U_i \times G &\lra U_i \times G\\
(x,g) &\longmapsto \left(x,h_i(x) g\right).
}
The two cocyles are then linked by the formula
\begin{equation}\label{cohomological}
k_{ji}(x) = h_j(x) \cdot g_{ji}(x) \cdot h_i(x)^{-1} \text { for all } x \in U_i \cap U_j.
\end{equation}
\end{defi}

\v Cech cohomology deals with abstract cocyles fulfilling the conditions
\eqref{cocyle} and \eqref{cohomological} and the equivalence classes of
cohomological cocyles. We will see details in Subsection \ref{Cech}.

\begin{rem}\label{RedStrucGrp}
If $G \le \GL(V)$ is a closed subgroup (and hence a Lie subgroup) containing the images of the
transition maps of a vector bundle $\pi: E \to M$, then one can show (cf. Theorem
I.6.4.1 of \cite{Hu75}) that $\pi: E \to M$ is isomorphic to the bundle associated to a
$G$-principal bundle with base space $M$ and fiber $V$. \end{rem}

Another equivalent way to define vector bundles is given by Proposition
I.5.5.1 of \cite{Hu75}.
\begin{prop}\label{vectobundlebycocyles}
Let $V$ be a finite-dimensional \vs, $G$ a closed subgroup of $\GL(V)$, $M$ a manifold, $(U_i)_{i \in I}$ an open cover of $M$ and $g_{ji}: U_i \cap U_j \to G$ for $(i,j) \in \I$ a familiy of maps meeting the condition \eqref{cocyle}, where $\I:=\set{\left.(i,j) \in I \times I \right|U_i \cap U_j \neq \emptyset}$. Then there is a vector bundle $\pi: E \to M$ with fiber $V$, bundle atlas $(U_i, \ph_i)_{i \in I}$ and cocyle $(g_{ji})_{(i,j) \in \I}$ and it is unique up to isomorphism.
\end{prop}

Knowing different ways to define and construct vector bundles, it is now time to give some important examples of vector bundles.

\begin{exa} \
\begin{enumerate}
\item Let $V$ be a finite-dimensional \vs and $M$ a manifold. Then, by setting $E:=M \times V$, we obtain a vector bundle $\pi: E \to M, (m,v) \mapsto m$ with fiber $V$, called the \emph{trivial} vector bundle with base space $M$ and fiber $V$.
\item Let $M$ be a $m$-dimensional smooth manifold. Then the tangent bundle $TM$ is the disjoint union of tangent spaces $\bigcup_{x \in M} T_x M$ provided with a topology and a smooth structure like in \cite{Ne05}, Definition II.1.8. If $\pi: TM \to M$, $\pi(T_xM) = \set{x}$ is the natural projection, then $\pi: TM \to M$ is a vector bundle with fiber $\R^m$. The set of smooth sections of this bundle is denoted by $\mathcal{V}(M)$ or $\Gamma(TM)$ and these sections are called \emph{vector fields}.
\item Let $\pi_\ell: E_\ell \to M$ be a vector bundle with fiber $V_\ell$, bundle atlas $(U_i, \ph^\ell_i)_{i \in I}$ and cocyle $(g^\ell_{ji})_{(i,j) \in \I}$ for $\ell=1,2$. Theorem \ref{vectobundlebycocyles} yields the existence of the vector bundle $\Pi: \Hom(E_1,E_2) \to M$ with fiber $\Hom(V_1,V_2)$, bundle atlas $(U_i, \Phi_i)_{i \in I}$ and cocyle $(G_{ji})_{(i,j) \in \I}$, where we define the mappings $G_{ji}: U_i \cap U_j \to \GL(\Hom(V_1,V_2))$ by
\es{
G_{ji}(x)(f):= g^2_{ji}(x) \circ f \circ g^1_{ji}(x)^{-1}
}
for $f \in \Hom(V_1,V_2)$, $x \in U_i \cap U_j$ and $(i,j) \in \I$.
\end{enumerate}
\end{exa}

In order to discuss sections of bundles locally, we introduce the following notation.
\begin{defi}Let $(U,\ph)$ be a bundle chart of a vector bundle $\pi: E \to M$ with fiber $V$. The
local form of a section $X \in \Gamma^k(\pi^{-1}(U))$, where $k \in \Ni$, is $X^\ph \in
\operatorname{C}^k(U,V)$, defined by claiming the following diagram to be commutative: \kD{
\pi^{-1}(U)\ar[r]^{(\pi,\ph)} & U \times V \\ U\ar[u]^{X} \ar[ru]_{\left(\id_{U}, X^\ph \right)}.
}
If $N$ is a set and $\eta: \Gamma^k(\pi^{-1}(U)) \to N$ is a map, then the corresponding local form
of $\eta$ is denoted by $\eta^{\ph}: \operatorname{C}^k(U,V) \to N$. \end{defi}

A very useful type of cover of the base space of a bundle is the so-called Palais cover.
\begin{defi}\label{Palais}
If $\pi: E \to M$ is a bundle, then $\left(\left(V_i, \psi_i, \xi_i, \rho_i\right)_{i \in J}, (J_t)_{t=1}^r\right)$ is called \emph{Palais cover}, if
\begin{enumerate}
\item [(a)] $(V_i)_{i \in J}$ is a locally finite cover of $M$,
\item [(b)] $(V_i,\psi_i)_{i \in J}$ is a bundle atlas of $E$,
\item [(c)] $(V_i,\xi_i)_{i \in J}$ is an atlas of $M$,
\item [(d)] $(\rho_i: M \to [0,1])_{i \in J}$ is a partition of unity such that $\supp (\rho_i)$ is a compact subset of $V_i$ for all $i \in J$,
\item [(e)] $\set{J_1, \ldots , J_r}$ is a partition of $J$ such that $i,j \in J_t$ and $V_i \cap V_j \ne \emptyset$ already implies $i=j$ for all $t \in \set{1, \ldots , r}$.
\end{enumerate}
\end{defi}

If the base space of a bundle is paracompact, then there is always a Palais cover.
\begin{thm}\label{Palaisexi}
A bundle with paracompact base space possesses a Palais cover.
\end{thm}
\begin{proof}
Since $M$ is paracompact, there is surely $\left(V_j, \psi_j, \xi_j, \rho_j\right)_{j \in J}$ fulfilling the conditions (a)-(d). By Theorem I in Chapter 1.2 of \cite{GrHaVa72}, there is a finite set $S$ and a refinement $(V_{sk})_{s\in S, k \in \N}$ of $(V_j)_{j \in J}$, such that for each $s \in S$ we have $V_{sk} \cap V_{s\ell} = \emptyset$ if $k \ne \ell \in \N$. By restriction of the $\psi_j$, $\xi_j$ and $\rho_j$ we obtain induced mappings $\psi_{sk}, \xi_{sk}, \rho_{sk}$ and a Palais cover $\left(\left(V_{sk}, \psi_{sk}, \xi_{sk}, \rho_{sk}\right)_{s\in S, k \in \N}, (\set{s0, s1, s2, \ldots})_{s \in S}\right)$.
\end{proof}

The following lemma will be very useful in some future proofs.

\begin{lem}\label{Taylorlemma}
If we have a vector bundle $\pi: E \to M$ with fiber $V$, a section $A \in \Gamma(E)$, a point $x \in M$ and an integer $s \in \N$ such that $j^s_x(A)=0$, then there is an integer $r \in \N$ and there are sections $A_1, \ldots , A_r \in \Gamma(E)$ and functions $a_1, \ldots, a_r \in \Ci(M,\R)$ with $a_1(x)=\ldots=a_r(x)=0$ such that
\begin{equation}\label{Taylorlemmaeq}A=\sum_{i=1}^r a_i^{s+1} A_i.\end{equation}
\end{lem}
\begin{proof}
Let $n$ be the vector space dimension of $V$ and $m$ the dimension of $M$. The lemma will be proven in two steps.
\begin{enumerate}
\item We consider the case of $M$ being a convex open neigbourhood $U \sbs \R^m$ of $x=0$ and $E = U \times \K^n$ and $A \in \Ci(U,\K^n)$ with compact support $F \sbs U$. Then, by the Taylor Formula, there is a family of functions $A_{\alpha} \in \Ci\left(U, \K^n \right)$, $\alpha \in \N^m$, $\abs{\alpha} = s+1$, such that
\es{
A(y) = \sum_{\abs{\alpha}=s+1} \frac{1}{\alpha !} \cdot y_1^{\alpha_1} \cdots  y_m^{\alpha_m} \cdot A_{\alpha}(y)
}
for all $y \in U$. Let $W \sbs U$ be an open set with $F \sbs W$ and $\rho: U \to [0,1]$ be a smooth function with compact support such that $\rho_{|W} \equiv 1$. Then we have
\begin{equation}\label{Taylorlemmastep1}
A(y) = \rho(y)^{s+2} \cdot A(y) = \sum_{\abs{\alpha}=s+1} (\rho(y)y_1)^{\alpha_1} \cdots (\rho(y)y_m)^{\alpha_m} \cdot \frac{1}{\alpha !} \cdot \rho(y) \cdot A_{\alpha}(y)
\end{equation}
for all $y \in U$. By the Multinomial Theorem, we have
\[
(t_1  + \ldots + t_m)^k = \sum_{\abs{\alpha}=k} \begin{pmatrix} k \\ \alpha ! \end{pmatrix}  t_1^{\alpha_1} \cdots t_m^{\alpha_m}
\]
for $t_1, \ldots , t_m \in \R$ and, by the inversion formula presented in \cite{MoHeRaCo94}, there are, for $\alpha \in \N^m$, scalars $\lambda_{ij} = \lambda_{ij}(\alpha) \in \R$ and $\mu_j =\mu_j(\alpha) \in \R$, where $i\in \set{1, \ldots , m}$ and $j \in \set{1, \ldots , p}$ for an integer $p = p(\alpha) \in \N$, such that
\[
t_1^{\alpha_1} \cdots t_m^{\alpha_m} = \sum_{j=1}^{p} \mu_j \left(\lambda_{1j} t_1  + \ldots + \lambda_{mj} t_m \right)^{\abs{\alpha}}.
\]
Equation \eqref{Taylorlemmastep1} then turns into:
\es{
A(y) &= \sum_{\abs{\alpha}=s+1}  \sum_{j=1}^{p(\alpha)} \left(\underbrace{\lambda_{1j} \rho(y) y_1  + \ldots + \lambda_{mj} \rho(y) y_m}_{=: a_{\alpha}(y)} \right)^{s+1}  \cdot \underbrace{\frac{\mu_j}{\alpha !} \cdot \rho(y) \cdot A_{\alpha}(y)}_{=:\wt A_\alpha(y)}\\
& =\sum_{\abs{\alpha}=s+1} a_{\alpha}(y)^{s+1} \cdot \wt A_\alpha(y).
}
Since $\wt A_\alpha$ is of compact support and $a_{\alpha}(0)=0$ for all $\alpha \in \N^m$ with $\abs{\alpha} = s+1$, we have shown \eqref{Taylorlemmaeq} for this local case.
\item For the general case, we take an $x$-neigbourhood $U \sbs M$, a bundle chart $(U,\ph)$, a chart $(U,\xi)$ of $M$ such that $\xi(x)=0$ and $\xi(U)$ is convex and a smooth function $\rho: M \to [0,1]$ such that $\supp (\rho)$ is a compact subset of $U$ and $\rho_{|W} \equiv 1$ for an $x$-neighbourhood $W \sbs U$. By the argument in the first step, we my write
\es{
(\rho \cdot A)^\ph \circ \xi^{-1} = \sum_{i=1}^r {\wt a_i}^{s+1} \wt A_i
}
for smooth functions $\wt A_1, \ldots \wt A_r \in \Ci(\xi(U),V)$ and $\wt a_1, \ldots , \wt a_r \in \Ci(\xi(U),\R)$ with $\wt a_1(x) = \ldots = \wt a_r(x) = 0$. For $i \in \set{1, \ldots r}$, we define $A_i \in \Gamma(E)$ by $A_i := \rho \cdot A'_i$, where $(A'_i)^\ph (y):= \wt A_i(\xi(y))$ for all $y \in U$ and we define $a_i \in \Ci(M,\R)$ with $a_1(x)=\ldots=a_r(x)=0$ by $a_i := \rho \cdot \left(\wt a_i \circ \xi\right)$. This yields $A=\sum_{i=1}^r a_i^{s+1} A_i$ on $M$.
\end{enumerate}
\end{proof}

\subsection{Bundles of Lie algebras and the Lie algebra of sections}
Now we are ready to define a central objects of this diploma thesis: Bundles of
Lie algebras and the corresponding Lie algebra of $\Ck$-sections. The
definitions of this subsection are taken from \cite{Le80}.
\begin{defi}\label{bundleofLiealgebras} Let $\k$ be a Lie algebra and $M$ a
manifold. A \emph{bundle of Lie algebras} $\pi: \L \to M$ with fiber $\k$ is a
vector bundle $\pi: \L \to M$ with fiber $\k$ which has an $\Aut(\k)$-valued
cocyle $(g_{ji})_{(i,j) \in \I}$. \end{defi}

We now define additional structure on a bundle of Lie algebras and on the
corresponding space of $\Ck$-sections.

\begin{defi} Let $\pi: \L \to M$ be a bundle of Lie algebras with fiber $\k$ and $k \in \Ni$.
\begin{enumerate}
\item We justify the name ``bundle of Lie algebras'' by defining a Lie bracket on each fiber $\L_x$ for $x \in M$ by a bundle chart $(U_i, \ph_i)$ in $x$ and the Lie bracket $\lie{\cdot,\cdot}$ on $\k$ (cf. Example \ref{exampAlg}.2):
\es{
\lie{\cdot,\cdot}: \L_x \times \L_x &\lra \L_x\\
(v,w) &\longmapsto \lie{v,w}_{\ph_i}=\left(\ph_i\right)_{|_{\L_x}}^{-1}\lie{\ph_i(v),\ph_i(w)}.
}
In order to see that this is well-defined, we make use of the fact that all maps $g_{ij}(x)$ for $(i,j)\in \I$ are isomorphisms of Lie algebras:
\es{
\lie{v,w}_{\ph_j}&=\left(\ph_j\right)_{|_{\L_x}}^{-1}\lie{\ph_j(v),\ph_j(w)}=
\left(\ph_i\right)_{|_{\L_x}}^{-1} \left(g_{ij}(x) \lie{\ph_j(v),\ph_j(w)}\right)\\
&=\left(\ph_i\right)_{|_{\L_x}}^{-1} \lie{g_{ij}(x)(\ph_j(v)),g_{ij}(x)(\ph_j(w))}=\left(\ph_i\right)_{|_{\L_x}}^{-1}\lie{\ph_i(v),\ph_i(w)}\\
&=\lie{v,w}_{\ph_i}.
}

\item With the help of the Lie algebra structure on each fiber of $\L$, we pointwisely define a Lie bracket on the vector space $\Gamma^k(\L)$ by
\es{
\lie{X,Y}(x):=\lie{X(x),Y(x)} \text{ for }X,Y \in \Gamma^k(\L), x \in M.
}

\item $\Gamma^k(\L)$ turns out into a \textsl{topological} Lie algebra by embedding it into the space $\Ck(M,\L)$ equipped with the $\Ck$-compact-open topology, i.e. embed it into the topological product $\prod_{i=0}^k \text{\textit{C}}\left(T^iM,T^i\L\right)$ by $X \mapsto \left(T^iX\right)_{i=0}^k$ , where each $\text{\textit{C}}\left(T^iM,T^i\L\right)$ is equipped with the compact-open topology.
\end{enumerate}
\end{defi}

\begin{defi}\label{Liealgebraoforderzero}
The Lie algebra $\Gamma^k(\L)$ is called \emph{Lie algebra of $\Ck$-sections}.
In the case of $k=\infty$ we call $\Gamma^\infty(\L)=\Gamma(\L)$ also \emph{Lie
algebra of  smooth sections}. \end{defi}

\begin{rem} \
\begin{enumerate}
\item In the case of $M:=\set{0}$ we have $\L=\set{0}\times \k \cong \k$ and
$\Gamma^k(\L)= \Ck\left(\set{0},\set{0}\times \k\right)\cong \k$. So the
finite-dimensional Lie algebras are special \La{}s of $\Ck$-sections.
\item The
Lie algebra $\Gamma^k(\L)$ is a Lie algebra of order zero in the following sense: For sections $X, Y
\in \Gamma^k(\L)$ and $x \in M$, the expression $\lie{X,Y}(x)$ depends on $X,Y$ only via their
``zeroth order parts'', namely $X(x)$, $Y(x)$. The Lie algebra of smooth vector fields
$\mathcal{V}(M)=\Gamma(TM)$ is a Lie algebra of order 1. \end{enumerate}
\end{rem}

There is a useful lemma about Lie algebras of $\Ck$-sections, which we want to
prove right now.

\begin{lem}\label{billigLemma}
Let $\pi: \L \to M$ be a bundle of Lie algebras with fiber $\k$ and $k \in \Ni$. Then for all $x \in M$ the evaluation map
\es{
\ev_x: \Gamma^k(\L) & \lra \L_x\\
X &\longmapsto X_x
}
is a surjective morphism of Lie algebras.
\end{lem}
\begin{proof}
The function $\ev_x$ is clearly linear. It is also compatible with the Lie bracket:
\es{
\ev_x\lie{X,Y}=\lie{X,Y}_x = \lie{X_x, Y_x} = \lie{\ev_x(X), \ev_x(Y)}.
}
For the proof of the surjectivity, choose an arbitrary $u \in \L_x$, a bundle chart $(U,\ph)$ such that $x\in U$ and let $\rho: U \to \R$ be a smooth map with compact support contained in $U$ and $\rho(x)=1$. We define $X \in \Gamma(\L) \sbs \Gamma^k(\L)$ via
\[
X_y:=
\begin{cases}
\ph_{|\L_x}^{-1}\left(\rho(y) \cdot \ph(u)\right) & \text{ for }y \in U.\\
0 & \text{ for }y \in M\backslash U.
\end{cases}
\]
Then $\ev_x(X)=X_x=u$.
\end{proof}

\begin{rem}\label{dirderi}
Let $\pi: \L \to M$ be a bundle of Lie algebras with fiber $\k$ and $(U,\ph)$ a bundle chart of $\L$ and $(U,\xi)$ a  corresponding chart of $M$. Any directional derivative $\d_{\xi_i}$ for $i \in \set{1, \ldots , \dim(M)}$ is a derivation of the poinwise Lie algebra structure on $\Ci(U,\k)$, the local forms of the sections in  $\Gamma\left(\pi^{-1}(U)\right)$. This follows from the Chain rule, the fact that any bilinear map $\beta: V \times V \to V$ for $\dim(V) < \infty$ is smooth and the rule $T_{(a,b)}\beta(v,w)=\beta(v,b)+\beta(a,w)$. Thus we have the equation:
\es{
\d_{\xi_i} \lie{A^{\ph},B^{\ph}} = \lie{\d_{\xi_i}A^{\ph}, B^{\ph}} + \lie{A^{\ph},\d_{\xi_i}B^{\ph}}
}
for all $A,B \in \Gamma\left(\pi^{-1}(U)\right)$.\end{rem}

There are certain bundles associated to a bundle of Lie algebras which will help
us to analyze the structure of a Lie algebra of $\Ck$-sections.

 \begin{defi}\label{specialsubbundles}Let $\pi: \L \to M$ be a bundle of Lie algebras with fiber $\k$ associated to an $\Aut(\k)$-principal bundle $\rho: P \to M$.
\begin{enumerate}
\item
Let $I \ide \k$ be a \emph{characteristic} ideal of $\k$, i.e. an ideal which is $\Aut(\k)$-stable.
We define $\varpi: \L[I] \to M$ to be the bundle associated to $\rho: P \to M$
with fiber $I$. It can be identified with a subbundle of $\pi: \L \to M$ as
follows: A typical element of $\L[I]_x \sbs \L[I]$ takes the form \es{
[(\psi,i)]:=\set{(\psi f^{-1}, f(i)): f \in \Aut(\k) } }
for a Lie algebra isomorphism $\psi: \k \to \L_x$ and an element $i \in I$. The element $\psi(i) \in \L_x$
is clearly well-defined and so we may embed $\L[I]$ into $\L$.
In the cases of $I$ being $\lie{\k,\k}$ and $\z(\k)$ \footnote{
These ideals are characteristic:
\begin{enumerate}
\item $i=\sum_j \lie{x_j,y_j} \implies f(i) = \sum_j \lie{fx_j,fy_j}$ for all $f \in \Aut(\k)$, $i \in \lie{\k,\k}$, $x_j, y_j \in \k$.
\item $\lie{f^{-1}x,i}=0 \implies f\lie{f^{-1}x,i} = 0 \implies \lie{x,f(i)}= 0$ for all $f \in \Aut(\k)$, $i \in \z(\k)$, $x \in \k$.
\end{enumerate}
} we also write $\lie{\L,\L}$ instead of $\L[\lie{\k,\k}]$, called the
\emph{commutator} of $\L$ and $\z(\L)$ instead of $\L[\z(\k)]$, called the
\emph{center} of $\L$. Note that $\lie{\L,\L}_x \cong \lie{\L_x,\L_x}$ and
$\z(\L)_x \cong \z(\L_x)$ via $[(\psi,i)] \mapsto \psi(i)$ for $i \in
\lie{\k,\k}$ and $i \in \z(\k)$, respectively.

\item Let $J \sbs \End(\k)$ be a subalgebra which is invariant under the left action
\es{
\Aut(\k) \times \End(\k) &\lra \End(\k)\\
(f,g) &\longmapsto f \circ g \circ f^{-1}.
}
We define $\varpi: \L(J) \to M$ to be the bundle associated to $\rho: P \to M$ with fiber $J$. It can be identified
with a subbundle of $\Pi: \Hom(\L,\L) \to M$ as follows: A typical element of
$\L(J)_x \sbs \L(J)$ takes the form \es{ [(\psi,j)]:=\Aut(\k)\cdot
(\psi,j)=\set{(\psi f^{-1}, f j f^{-1}): f \in \Aut(\k) } } for a Lie algebra
isomorphism $\psi: \k \to \L_x$ and an endomorphism $j: \k \to \k$, $j \in J$.
The map $\psi \circ j \circ \psi^{-1}: \L_x \lra \L_x$ is clearly well-defined
and, by applying this construction to all classes $[(\psi,j)]$ for $\psi \in
\Iso(\k,\L_x)$, where $x$ runs through $M$ and $j \in J$ is fixed, we obatin an
element of $\Hom(\L,\L)$. In the cases of $J$ being $\Der(\k)$ and $\Cent(\k)$
\footnote{ These algebras are $\Aut(\k)$-invariant: \begin{enumerate} \item
$fgf^{-1}\lie{x,y}=fg(\lie{f^{-1}x,f^{-1}y})=f(\lie{gf^{-1}x,f^{-1}y}+\lie{f^{-1
}x,gf^{-1}y}) = \lie{fgf^{-1}x,y}+\lie{x,fgf^{-1}y}$\\ for all $f \in \Aut(\k)$,
$g \in \Der(\k)$, $x,y \in \k$. \item $(fgf^{-1} \circ
\ad_x)(y)=fgf^{-1}\lie{x,y}=fg(\lie{f^{-1}x,f^{-1}y})=(fg \circ
\ad_{f^{-1}x})(f^{-1}y)=(f \circ
\ad_{f^{-1}x})(gf^{-1}y)=f\lie{f^{-1}x,gf^{-1}y}=\lie{x,fgf^{-1}y} = (\ad_x
\circ fgf^{-1})(y)$ for all $f \in \Aut(\k)$, $g \in \Cent(\k)$, $x,y \in \k$.
\end{enumerate} } we also write $\Der \left(\L\right)$ instead of $\L(\Der(\k))$
and $\Cent(\L)$ instead of $\L(\Cent(\k))$. Note that $\Der(\L)_x \cong
\Der(\L_x)$ and $\Cent(\L)_x \cong \Cent(\L_x)$ via $[(\psi,j)] \mapsto \psi
\circ j \circ \psi^{-1}$ for $j \in \Der(\k)$ and $j \in \Cent(\k)$,
respectively. \end{enumerate} \end{defi}

\subsection{Lie connections}
We briefly repeat the definition of a covariant derivative and define the Lie connection corresponding to a fixed bundle of Lie algebras.

\begin{defi}\label{covderi}
Let $\pi: E \to M$ be a vector bundle with fiber $V$. A \emph{covariant derivative} or \emph{connection} of $E$ is a linear map
\es{
\nabla: \Gamma(TM) &\lra \End(\Gamma(E))\\
X &\longmapsto \nabla_X
}
satisfying the following conditions:
\begin{enumerate}
\item [(a)] $\nabla_X(aA) = (X.a)A + a \nabla_X A$ for all $X \in \Gamma(TM), a\in \operatorname{C}^\infty(M,\K), A \in \Gamma(E)$.
\item [(b)] $\nabla_{(aX + bY)}A = a\nabla_X A + b \nabla_Y A$ for all $X,Y \in \Gamma(TM), a,b\in \operatorname{C}^\infty(M,\K), A \in \Gamma(E)$.
\end{enumerate}
\end{defi}

\begin{rem}\label{pointwisedepen}
Given a point $x \in M$, a vector field $X \in \Gamma(TM)$ and sections $A,B \in \Gamma(E)$ it is, by using property (b) of Definition \ref{covderi}, easy to see that $((\nabla_X A)B )(x)$ depends on $X$ only via $X_x$:

Firstly we show that $((\nabla_X A)B )(x)$ depends on $X$ at most via $X_{|U}$, where $U \sbs M$ is any open neighbourhood of $x$. In fact, if $\rho: M \to [0,1]$ is a smooth map with $\rho_{|M \backslash U} \equiv 0$ and $\rho_{|W} \equiv 1$ for a smaller $x$-neighbourhood $W \sbs U$, then $(X-X')_{|U} \equiv 0$ implies
\es{
(\nabla_{X}Y)(x) - (\nabla_{X'}Y)(x)=(\nabla_{X-X'}Y)(x)= \rho(x) \cdot (\nabla_{X-X'}Y)(x) = (\nabla_{\rho \cdot (X-X')}Y)(x) = (\nabla_{0}Y)(x)=0.
}
Thus we can now reduce the problem to a problem in a chart $(U,\xi)$: Let $X, X'$ be two vector fields identical in $x$ and $\rho: M \to [0,1]$ as above. We can write $X-X'= \sum_{i=1}^m Z^i \d_{\xi_i}$ for certain functions $Z^1, \ldots Z^m \in \Ci(U,\R)$.  We note that $Z^i(x)=0$ for all $i$ and conclude by calculating:
\es{
(\nabla_{X}Y)(x) - (\nabla_{X'}Y)(x)&= (\nabla_{\rho \cdot (X-X')}Y)(x) = \left(\nabla_{\rho \cdot \left( \sum_{i=1}^m Z^i \d_{\xi_i}\right)}Y\right)(x)\\
&= \left(\sum_{i=1}^m \left( \rho \cdot Z^i \right) \cdot \nabla_{\d_{\xi_i}}Y  \right)(x) = \sum_{i=1}^m Z^i(x) \cdot \left(\nabla_{\d_{\xi_i}}Y\right)(x) = 0.
}
\end{rem}

Locally, a covariant derivative ist given by the so-called Christoffel symbols.
\begin{defi}
Let $(U,\ph)$ be a bundle chart of the vector bundle $\pi: E \to M$ with fiber $V$ and $(U,\xi)$ a corresponding chart of $M$ and let $\nabla$ be a covariant derivative of $E$. Then there are unique mappings $\Gamma^{\ph}_i \in \Ci(U,\End(V))$ for $i \in \set{1, \ldots , \dim(M) }$ such that
\es{
\left(\nabla_X A\right)^{\ph} = \sum_{i=1}^{\dim(M)} X^i \left(\d_{\xi_i}A^{\ph}+ \Gamma^{\ph}_i \cdot A^\ph \right)
}
for all sections $A$ and all vector fields $X$ with local form $\sum_{i=1}^m X^i \d_{\xi_i}$. The maps $\Gamma^\ph_i$ are called \emph{Christoffel symbols} of $\nabla$ \wrt $(U,\ph)$.
\end{defi}
For a given bundle of Lie algebras we now define a Lie connection.
\begin{defi}
Let $\pi: \L \to M$ be a bundle of Lie algebras with fiber $\k$ and $\nabla$ a covariant derivative of $\L$. We call $\nabla$ a \emph{Lie connection} if for all vector fields $X \in \Gamma(TM)$ and for all sections $A,B \in \Gamma(\L)$ we have
\es{
\nabla_X\lie{A,B} = \lie{\nabla_X A,B} + \lie{A,\nabla_X B},
}
i.e. $\nabla$ is a map $\Gamma(TM) \to \Der(\Gamma(\L))$.
\end{defi}
Whether a covariant derivative of a given bundle of Lie algebras is a Lie connection, can be decided by examining its Christoffel symbols.
\begin{lem}\label{localadapted}
Let $\pi: \L \to M$ be a bundle of Lie algebras with fiber $\k$ and $\nabla$ a covariant derivative of $\L$. Then $\nabla$ is a Lie connection \iff for all bundle charts $(U,\ph)$ of $\L$ and all corresponding charts $(U,\xi)$ of $M$ the Christoffel symbols are in in $\Ci(U,\Der(\k))$.
\end{lem}
\begin{proof}
The covariant derivative $\nabla$ is a Lie connection \iff for all bundle charts $(U,\ph)$ of $\L$, all corresponding charts $(U,\xi)$ of $M$, all sections $A,B \in \Gamma\left(\pi^{-1}(U)\right)$ and any integer $i \in \set{1, \ldots , \dim(M)}$ we have
\begin{equation}\label{chrislocaldericondition}
\d_{\xi_i}\lie{A^{\ph},B^{\ph}}+ \Gamma^{\ph}_i \lie{A^{\ph},B^{\ph}}
= \lie{\d_{\xi_i}A^{\ph}+\Gamma^{\ph}_i\left(A^{\ph}\right) ,B^{\ph}}
+ \lie{A^{\ph} ,\d_{\xi_i}B^{\ph}+\Gamma^{\ph}_i\left(B^{\ph}\right)}.
\end{equation}
Due to Remark \ref{dirderi} we have
\es{
\d_{\xi_i} \lie{A^{\ph},B^{\ph}} = \lie{\d_{\xi_i}A^{\ph}, B^{\ph}} + \lie{A^{\ph},\d_{\xi_i}B^{\ph}}.
}
So the equation \eqref{chrislocaldericondition} turns into
\es{
\Gamma^{\ph}_i (x)\lie{A^{\ph}(x),B^{\ph}(x)} =  \lie{\Gamma^{\ph}_i (x) \left(A^{\ph}(x)\right),B^{\ph}(x)} + \lie{A^{\ph}(x),\Gamma^{\ph}_i (x) \left(B^{\ph}(x)\right)}
}
for all $x \in U$. But the last equation is, due to Lemma \ref{billigLemma}, equivalent to $\Gamma^{\ph}_i \in \Ci(U,\Der(\k))$.
\end{proof}

Finally we want to show the existence of Lie connections.
\begin{prop}
Let $\pi: \L \to M$ be a bundle of Lie algebras with fiber $\k$. Then there exists a covariant derivative $\nabla$ of $\L$.
\end{prop}
\begin{proof}
Let $\pi: \L \to M$ be a bundle of Lie algebras with fiber $\k$ and $(U_j,\ph_j)_{j \in J}$ a locally finite bundle atlas of $\L$ with corresponding atlas $(U_j,\xi_j)_{j\in J}$ of $M$. For each $j \in J$ we define a covariant derivative $\nabla^{U_j}$ on $\pi^{-1}(U_j)$ by:
\es{
\left(\nabla^{U_j}_X A \right)^{\ph_j} := \sum_{i=1}^{\dim(M)} X^i \d_{\left(\xi_j\right)_i} A^{\ph}
}
for $X \in \Gamma\left(TU\right)$ and $A \in \Gamma\left(\pi^{-1}(U_j)\right)$. This covariant derivative is a Lie connection (cf. Lemma \ref{localadapted}). After choosing a smooth partition $\left( \rho_j: M \to [0,1] \right)_{j \in J}$ of 1 subordinated to the cover $(U_j)_{j \in J}$, i.e. $\supp(\rho_j)$ is a compact subset of $U_j$ for all $j \in J$, we define a covariant derivative $\nabla$ of $\L$ by
\es{
\nabla := \sum_{j \in J}\rho_j \cdot \nabla^{U_j}.
}
Since $\nabla$ is locally the finite sum of Lie connections, it is also a Lie connection.
\end{proof}

\subsection{Local operators and the Peetre Theorems}
\begin{defi}
Let $\pi_1: E_1 \to M$, $\pi_2: E_2 \to M$ be two vector bundles and $k_1,k_2 \in \Ni$. A linear operator $T: \Gamma^{k_1}(E_1) \to \Gamma^{k_2}(E_2)$ is called \emph{local}, if for any $X \in \Gamma^{k_1}(E_1)$ and any open set $U \sbs M$ the condition $X_x=0 \in (E_1)_x$ for all $x \in U$ implies $(TX)_x=0 \in (E_2)_x$ for all $x \in U$.
\end{defi}
The meaning of an operator to be local is shown in the following lemma.
\begin{lem}
If $Y,Z \in \Gamma^{k_1}(E_1)$ are identical sections on an open set $U \sbs M$ and we have a local operator $T: \Gamma^{k_1}(E_1) \to \Gamma^{k_2}(E_2)$, then $TY, TZ \in \Gamma^{k_1}(E_2)$ are identical on $U$.
\end{lem}
\begin{proof}
The section $X:=Y-Z \in \Gamma^{k_1}(E_1)$ vanishs on $U$. The locality of $T$ yields that $TX=TY-TZ$ also vanishs on $U$. So $TY_{|U} = TZ_{|U}$.
\end{proof}

\begin{exa}
Let $\nabla: \Gamma(TM) \lra \End(\Gamma(E))$ be a covariant derivative on a vector bundle $\pi: E \to M$ with fiber $V$ and $X \in \Gamma(TM)$ a vector field. Then $\nabla_X$ is a local operator:

Let $Y \in \Gamma(E)$ be a section zero on an open set $U \sbs M$, let $x \in U$ be a point and let $\rho: M \to [0,1]$ be a smooth map such that $\rho_{|M \backslash U} \equiv 0$ and $\rho_{|W} \equiv 1$ for a smaller $x$-neighbourhood $W \sbs U$. We easily calculate:
\es{
(\nabla_{X}Y)(x) = \rho(x) (\nabla_{X}Y)(x) = (\nabla_{X} (\rho \cdot Y))(x) - \left((X.\rho)Y\right)(x) = 0-0 = 0.
}
\end{exa}

The Peetre Theorem, proven in \cite{Na68}, p. 175-176, can be stated as follows:
\begin{thm}\label{Peetre}\textbf{\upshape(Peetre Theorem, Version 1)}
Let $M$ be a smooth $m$-dimensional manifold and $\pi_i: E_i \to M$ a vector bundle with fiber $V_i$ for $i=1,2$. If $T: \Gamma\left(E_1\right) \to \Gamma\left(E_2\right)$ is a local operator, then for any point $x \in M$ there exists an open neighbourhood $U \sbs M$, bundle charts $(U,\ph_1)$, $(U,\ph_2)$ of $E_1$, $E_2$, respectively, a chart $(U,\xi)$ of $M$, a number $n \in \N$ and a familiy of functions $f_{\alpha} \in \Ci\left(U, \Hom(V_1,V_2) \right)$, $\alpha \in \N^m$, $\abs{\alpha} \leq n$, such that for all $X\in \Gamma\left(\pi_1^{-1}(U)\right)$ we have
\begin{equation}\label{Diffop}
(TX)^{\ph_2}=\sum_{\abs{\alpha} \leq n} f_{\alpha} \cdot \left(\partial_{\xi}^{\alpha}X^{\ph_1}\right).
\end{equation}
\end{thm}
\begin{defi}
The formula \eqref{Diffop} says that $T$ is a \emph{differential operator of order at most $n$} on $U$.
\end{defi}
By a modification of the proof of Theorem \ref{Peetre} one gets the following result (cf. \cite{Le79}, p.52):
\begin{thm}\label{PseudoPeetre}\textbf{\upshape(Peetre Theorem, Version 2)}
Let $M$ be a smooth $m$-dimensional manifold and $\pi_i: E_i \to M$ a vector bundle with fiber $V_i$ and $k_i \in \N$ for $i=1,2$. If $T: \Gamma^{k_1}(E_1) \to \Gamma^{k_2}(E_2)$ is a local operator, then $T$ is a differential operator of order at most $k_1 - k_2$. In particular, if $k_1=k_2$ then $T$ is a differential operator of order 0 and if $k_1 < k_2$ then $T=0$. Furthermore, we obtain  $T=0$ if $T: \Gamma^{k_1}(E_1) \to \Gamma(E_2) \sbs \Gamma^{k_1 + 1}(E_2) $ is local for $k_1 \in \N$.
\end{thm}

\begin{rem}\label{diffopsection}
In the situations of the above theorems, if $T: \Gamma^{k}(E_1) \to
\Gamma^{k}(E_2)$ is a differential operator of order 0 for $k \in \Ni$, then it
can be identified with a $\operatorname{C}^{k}$-section of the bundle
$\Hom(E_1,E_2)$ as follows: Fix an $x \in M$ and an open neighbourhood $U\sbs
M$ of $x$ such that $\pi_1^{-1}(U)$ is trivial. For vectors $v \in (E_1)_x$ we
choose sections $X: U \to \pi_1^{-1}(U)$ such that $X_x=v$ and define the linear
map \es{ \tau_x: (E_1)_x  &\lra (E_2)_x\\ v &\longmapsto \left(TX\right)_x. }
\kommentar{
$\widehat{Tu_1}=a \widehat{u_1}$ for smooth $a: U \cong \Omega \to M_{r_2,
r_1}(\K)$, then define a linear map $\tau_x: (E_1)_x \to (E_2)_x$ by making the
following diagram commutative: \kD{ (E_1)_x \ar[d]_{\tau_x}			 & V_1
\ar[l]_{\cong}			& \K^{r_1} \ar[l]_{\cong} \ar[d]^{a(x)}\\ (E_2)_x	 & V_2
\ar[l]^{\cong}			& \K^{r_2} \ar[l]^{\cong}. } }
depends on $X$ only via $X_x$ because $T$ is an operator of order $0$. The
required section is the map $M \to \Hom(E_1,E_2)$, $x \mapsto \tau_x$. \end{rem}

\subsection{\v Cech cohomology}\label{Cech}

\v Cech cohomology is one of the important cohomology theories in algebraic topology. It can be defined for any topological space and a presheaf of groups on this space. However, we will only consider it for a manifold $M$ and a discrete group $G$, which can be understood as a Lie group. The following definitions are due to \cite{tD91}.

\begin{defi} Let $\Co(M):=\set{\left. \V \in \P(\O(M)) \right| \bigcup \set{U \in \V}=M}$ denote the set of collections of open subsets of $M$ who cover $M$ and for $\V \in \Co(M)$ we define $\V * \V :=\set{\left.(U,V) \in \V \times \V \right| U \cap V \ne \emptyset}$.

\begin{enumerate}
\item A familiy of smooth functions $(g_{VU}: U \cap V \to G)_{(U,V) \in \V * \V}$ is called a \emph{cocyle} related to $\V \in \Co(M)$, if it satisfies the equation
\es{
g_{VU}(x) \cdot g_{UW}(x) \cdot g_{WV}(x) = 1
}
for all $x \in U \cap V \cap W$ for all $ U,V,W \in \V$. The set of these families is denoted by $\Zc^1(M,\V,G)$.

\item Two cocycles $(g_{VU})_{(U,V) \in \V * \V}, (k_{VU})_{(U,V) \in \V * \V} \in \Zc^1(M,\V,G)$ are called \emph{cohomologous}, if there exists a familiy of smooth mappings $(h_U: U \to G)_{U \in \V}$ such that the relation
\es{
k_{VU}(x) = h_V(x) \cdot g_{VU}(x) \cdot \left(h_U(x)\right)^{-1}
}
is satisfied for all $x \in U \cap V$ for all $U,V \in \V$. The relation ``cohomologous'' is an equivalence relation on $\Zc^1(M,\V,G)$ and the set of the corresponding equivalence classes is denoted by $\Hc^1(M,\V,G)$.

\item If $G$ is abelian, two cocycles $(g_{VU})_{(U,V) \in \V * \V}, (k_{VU})_{(U,V) \in \V * \V} \in \Zc^1(\V,G)$ can be multiplied pointwisely and we get a group structure on $\Zc^1(M,\V,G)$. The set of cocyles $x \mapsto h_V(x) \cdot \left(h_U(x)\right)^{-1}$ is a normal subgroup of $\Zc^1(M,\V,G)$, denoted by $\Bc^1(M,\V,G)$ and this gives rise to the group structure on the quotient $\Zc^1(M,\V,G) / \Bc^1(M,\V,G) = \Hc^1(M,\V,G)$.

\item $\Co(M)$ is a directed set by the refinement relation
\es{
\V \prec \W \quad :\Longleftrightarrow \quad \forall W \in \W \, \exists V \in \V: \, W \sbs V
}
and so we can define the \emph{first \v Cech cohomology group} of $M$ with values in $G$ to be\footnote{This definition is not rigorous. If $(W_j)_{j \in J}$ is a refinement of $(V_i)_{i \in I}$, then, for the definition of the direct limit, it is necessary to have a function $f: J \to I$ such that $f(j)=i$ implies $W_j \sbs V_i$. But one can show that $\Hc^1(M,G)$ does not depend on the choice of such functions.}
\es{
\Hc^1(M,G):=\varinjlim_{\V \in \Co(M)} \Hc^1(M,\V,G).
}
In general, if $G$ is not necessarily abelian, we obtain the \emph{first \v Cech cohomology set} $\Hc^1(M,G)$ without group structure.
\end{enumerate}
\end{defi}

If $G$ is abelian and discrete, then the \v Cech cohomology groups $\Hc^1(M,G)$ are isomorphic to the singular cohomology groups $\operatorname{H}^1(M,G)$ (cf. e.g \cite{EiSt57}, Chapter IX), which can be easily ``calculated'' in some cases. There are also group isomorphies $\Hc^1(M,G) \cong \Hom(\operatorname{H}_1(M),G) \cong \Hom(\pi_1(M),G)$, where $\operatorname{H}_1(M)$ is the first singular homology group of $M$ with values in $\Z$ and $\pi_1(M)$ is the fundamental group of $M$.

\begin{rem}
We give some first cohomology groups which are calculated in any elementary algebraic topology lecture:
\begin{enumerate}
\item If $G$ is an abelian group and $M$ is a simply connected manifold, then $\operatorname{H}^1(M,G)=0$.
\item If $G$ is an abelian group, then $\operatorname{H}^1(\S^n,G)= \begin{cases}G & \text{ if } n=1 \\0 & \text{ if } n>1. \end{cases}$
\item $\operatorname{H}^1(\C P^n,\Z/2\Z)=\operatorname{H}^1(\H P^n, \Z/2\Z)=0$.
\end{enumerate}
\end{rem}

\subsection{Structure of Lie algebras of $\Ck$-sections}

From now on, facts from \cite{Le80} are explained in detail and generalized. We have $k \in
\Ni$ and a bundle of Lie algebras $\pi: \L \to M$ with fiber $\k$. The dimensions of $M$ and $\k$
are denoted by $m$ and $d$, respectively. Lecomte only considered the case of $\z(\k) = 0$ and we
will replace this condition with weaker ones.

\subsubsection{Center and commutator} The center of
$\Gamma^k(\L)$ is easy to calculate. \begin{prop} For $k \in \Ni$ we have
$\z\left(\Gamma^k(\L)\right) = \Gamma^k\left(\z(\L) \right)$. \end{prop}
\begin{proof}
If $X \in \Gamma^k\left(\z(\L) \right)$ then, for all $x \in M$, we have $X_x \in \z(\L_x)$ and thus $\lie{X,Y}_x = \lie{X_x,Y_x} = 0$ for all $Y \in \Gamma^k(\L)$ and $x \in M$. This proves $\z\left(\Gamma^k(\L)\right) \sps \Gamma^k\left(\z(\L) \right)$. If $X \in \z\left(\Gamma^k(\L)\right)$, $x \in M$ and $u \in \L_x$, then, by Lemma \ref{billigLemma}, there is a section $Y \in \Gamma^k(\L)$ such that $Y_x = u$, thus $\lie{X_x,u}=\lie{X,Y}_x=0$ proving $X_x \in \z(\L_x)$. This proves $\z\left(\Gamma^k(\L)\right) \sbs \Gamma^k\left(\z(\L) \right)$.
\end{proof}
\begin{cor}
$\Gamma^k(\L)$ is centerfree \iff $\k$ is centerfree.
\end{cor}

For the calculation of the commutator of $\Gamma^k(\L)$ we first have to prove a technical lemma.
\begin{lem}\label{commutatorlemma}
There exists an $r \in \N$ such that, for each bundle chart $(U,\ph)$ and each $X \in \Gamma^k(\lie{\L,\L})$ with compact support contained in $U$, there are sections $Y_1, Z_1, \ldots , Y_r, Z_r \in \Gamma^k(\L)$ with compact supports contained in $U$ with
\es{
X= \sum_{t=1}^r \lie{Y_t,Z_t}.
}
\end{lem}
\begin{proof}
Let $(\lie{u_1,w_1}, \ldots \lie{u_r,w_r})$ be a basis of $\lie{\k,\k}$. Let $F
\sbs M$ be compact such that \[\supp X \sbs F^\circ \sbs F \sbs U\]
and $\rho: M \to [0,1]$ a smooth function with compact support contained in $U$
such that $\rho_{|F} \equiv 1$. Then there are functions $f_1, \ldots , f_r \in
\Ck(M,\K)$ with supports contained in $F$ such that we have, for each $x \in U$:
\es{ X^\ph(x) = \sum_{t=1}^r f_t(x) \lie{u_t,w_t}. } For $t \in \set{1, \ldots ,
r}$ we define $Y_t, Z_t \in \Gamma^k(\L)$ by \es{ Y_t (x) = Z_t (x) = 0 &\text{
for }x \in M \backslash U\\ Y_t^\ph (x) = f_t(x) u_t \text{ and } Z_t^\ph (x) =
\rho(x) w_t &\text{ for }x \in U.\\ }
For $x \in M \backslash F$ we have $X(x) = 0 = \sum_{t=1}^r 0 = \sum_{t=1}^r \lie{Y_t(x),Z_t(x)}$. For $x \in F$ we have
\es{
X^\ph (x) = \sum_{t=1}^r \lie{f_t(x) u_t, 1 \cdot w_t} = \sum_{t=1}^r \lie{Y^\ph_t(x),Z^\ph_t(x)}.
}
Thus $X= \sum_{t=1}^r \lie{Y_t,Z_t}$.
\end{proof}
\begin{prop}
For $k \in \Ni$ we have $\lie{\Gamma^k(\L),\Gamma^k(\L)} = \Gamma^k\left(\lie{\L,\L} \right)$.
\end{prop}
\begin{proof}
Let $\left(\left(V_i, \psi_i, \xi_i, \rho_i\right)_{i \in J}, (J_t)_{t=1}^r\right)$ be a Palais cover (cf. Definition \ref{Palais} and Theorem \ref{Palaisexi}). Any section $X \in \Gamma^k(\L)$ takes the form $X=\sum_{t=1}^r X_t$ for $X_t=\sum_{i\in J_t} \rho_i X$, where $t \in \set{1, \ldots , r}$. We fix $X \in \Gamma^k(\lie{\L,\L})$ and $t \in \set{1, \ldots , r}$. Then, by Lemma \ref{commutatorlemma}, there are, for each $i \in J_t$, sections $Y^i_1, Z^i_1, \ldots , Y^i_s, Z^i_s \in \Gamma^k(\L)$ with compact supports contained in $V_i$, such that $\rho_i X = \sum_{p=1}^s \lie{Y^i_p, Z^i_p}$. By setting $Y_p:=\sum_{i \in J_t}Y^i_p \in \Gamma^k(\L)$ and $Z_p:=\sum_{i \in J_t}Z^i_p \in \Gamma^k(\L)$ for each $p \in \set{1, \ldots , s}$ and recalling that the supports of $Y^i_p$, $Z^j_p$ for $i \neq j$ are disjoint, we obtain $X_t= \sum_{p=1}^s \lie{Y_p,Z_p}$, yielding $X = \sum_{t=1}^r X_t$ and $\lie{\Gamma^k(\L),\Gamma^k(\L)} \sps \Gamma^k(\lie{\L,\L})$.

In order to show $\lie{\Gamma^k(\L),\Gamma^k(\L)} \sbs \Gamma^k(\lie{\L,\L})$ we consider $X = \sum_{s=1}^p \lie{Y_p, Z_p}$ for appropriate sections $Y_p, Z_p \in \Gamma^k(\L)$. For arbitary $x \in M$ we have $X(x) = \sum_{s=1}^p \lie{Y_p(x), Z_p(x)} \in \lie{\L_x,\L_x} = \lie{\L,\L}_x$.
\end{proof}

\begin{cor}\label{perfectcoro}
$\Gamma^k(\L)$ is perfect \iff $\k$ is perfect.
\end{cor}

\subsubsection{Derivations}
Now we want to calculate the derivations of the Lie algebras of $\Ck$-sections by applying Version 2 of the Peetre Theorem. \begin{thm}
Let $k \in \N$.  If $\k$ is perfect or centerfree, then $\Der\left(\Gamma^k(\L)\right)\cong\Gamma^k\left(\Der(\L)\right)$ as Lie algebras.
\end{thm}
\begin{proof}
Let $D$ be a derivation of $\Gamma^k(\L)$. We want to show that $D$ is automatically a local operator. Let $X \in \Gamma^k(\L)$ be zero on an open set $U \sbs M$.
\begin{itemize}
\item Suppose $\z(\k)=0$. Then for all $x \in U$ and $Y \in \Gamma^k(\L)$ with $\supp(Y) \sbs U$ we have $\lie{X,Y}=0$ and therefore:
\es{
\lie{(DX)_x, Y_x} = (D\lie{X,Y})_x - \lie{X_x,(DY)_x} = 0 - 0 = 0.
}
By Lemma \ref{billigLemma}, this implies $\lie{(DX)_x,v}=0$ for all $v \in \L_x$, yielding $(DX)_x \in \z(\L_x) = 0$ for all $x \in U$.
\item Suppose $\lie{\k,\k}=\k$. Since $\Gamma^k(\L)$ is also perfect, there exist sections $Y_1, Z_1, \ldots , Y_r, Z_r \in \Gamma^k(\L)$ such that $X=\sum_{t=1}^r \lie{Y_t,Z_t}$. Now, if $x \in U$ then there is an $x$-neigbourhood $W \sbs U$ and a smooth map $\rho: M \to [0,1]$ such that $\rho_{|W} \equiv 1$ and $\rho_{M\backslash U} \equiv 0$. We set $X':=(1-\rho)^2 \cdot X$. Thus $X = X'$ on $M$ and we obtain $X= \sum_{t=1}^r \lie{(1-\rho) \cdot Y_t,(1-\rho) \cdot Z_t}$ yielding
\es{
(DX)_x = \sum_{t=1}^r \lie{D ((1-\rho) \cdot Y_t)(x),(1-\rho)(x) \cdot Z_t(x)} + \lie{(1-\rho)(x)\cdot Y_t(x),D((1-\rho) \cdot Z_t)(x)} = 0.
}
\end{itemize}
In both cases we have $(DX)_{|U} \equiv 0$, thus $D$ is local.

By applying Theorem \ref{PseudoPeetre} and Remark \ref{diffopsection} we have a linear map $\Psi: \Der\left(\Gamma^k(\L)\right) \to \Gamma^k\left(\Hom(\L,\L)\right)$ well-defined by $\Psi(D)(X_x):=(DX)_x$ for $X \in \Gamma^k(\L)$, $x \in M$. Since we have
\es{
\Psi(D)\lie{X_x,Y_x}=(D\lie{X,Y})_x = (\lie{DX,Y})_x + (\lie{X,DY})_x = \lie{\Psi(D)(X_x),Y_x} + \lie{X_x,\Psi(D)(Y_x)}
} for $X,Y \in \Gamma^k(\L)$ and $x \in M$, the image of $\Psi$ is contained in $\Gamma^k(\Der(\L))$. Furthermore, the map $\Psi$ is compatible with the Lie brackets on $\Der\left(\Gamma^k(\L)\right)$ and $\Gamma^k(\Der(\L))$:
\es{
\Psi\lie{D,D'}(X_x) = \left(\lie{D,D'}X\right)_x =  \left(D(D'X)\right)_x - \left(D'(DX)\right)_x = \lie{\Psi(D),\Psi(D')}(X_x).
}
Finally, the function $\Phi: \Gamma^k(\Der(\L)) \to \Der\left(\Gamma^k(\L)\right)$ defined by $\left(\Phi\left(\mathfrak{D}\right)(X)\right)_x:=\mathfrak{D}_x(X_x)$ for a section $\mathfrak{D} \in  \Gamma^k(\Der(\L))$, $X \in \Gamma^k(\L)$, $x \in M$, is the inverse of $\Psi$. We conclude that $\Psi$ is an isomorphism of Lie algebras.
\end{proof}

In order to analyze $\Gamma(\L)=\Gamma^\infty(\L)$ we define the notion of an $x$-derivation of $\L$.

\begin{defi}\label{xderi}
Let $x \in M$.
\begin{enumerate}
\item An \emph{$x$-derivation} of $\L$ is a linear map $\delta: \Gamma(\L) \to \L_x$ such that the following condition is satisfied for all $X,Y \in \Gamma(\L)$:
\es{
\delta\lie{X,Y}=\lie{\delta(X),Y_x}+\lie{X_x,\delta(Y)}.
}
The vector space of all $x$-derivations of $\L$ is denoted by $\D_x(\L)$, the union $\bigcup_{x \in M} \D_x(\L)$ is denoted by $\D(\L)$ and we have a projection $\overline{p}: \D(\L) \lra M$ defined by $\overline{p}(\D_x(\L))=\set{x}$ for $x \in M$.

\item Note that $\Der(\L)$ can be embedded into $\D(\L)$ via the \emph{natural injection} $i: \Der(\L) \lra \D(\L)$, where $D \in \Der(\L_x)$ is mapped to the $x$-derivation $i(D): \Gamma(\L) \to \L_x$, $X \mapsto D(X_x)$.
\end{enumerate}
\end{defi}

\begin{thm}\label{dertheo}
Let $\k$ be perfect or centerfree. Then every $x$-derivation of $\L$ is a differential operator of order at most 1 and the triple $\left( \D(\L), \overline{p}, M\right)$ admits a vector bundle structure such that $\Der \left(\Gamma(\L)\right)$ can be naturally identified with $\Gamma(\D(\L))$. In addition, there exists a short exact sequence of vector bundles as follows:
\es{
0 \lra \Der \left(\L\right) \overset{i}{\lra} \D(\L) \overset{\sigma}{\lra} TM \otimes \Cent(\L) \lra 0.
}
\end{thm}
\begin{proof}
Let $\delta \in \D_x(\L)$ be an $x$-derivation. By Lemma \ref{Taylorlemma}, every section $X \in \Gamma(\L)$ with $j^1_x(X)=0$ takes the form
\begin{equation}
X=\sum_{i=1}^r f_i^2 X_i
\end{equation}
for functions $f_i \in \Ci(M,\K)$, $f_i(x)=0$ and $X_i \in \Gamma(\L)$.
\begin{itemize}
\item Suppose $\z(\k)=0.$ If $Y \in \Gamma(\L)$ is an arbitrary smooth section, then:
\es{
\lie{\delta(X),Y(x)}&=\sum_{i=1}^r \lie{\delta\left(f_i^2 X_i \right),Y(x)} = \sum_{i=1}^r\left(\delta\lie{f_i^2 X_i,Y} - \underbrace{\lie{f_i(x)^2 X_i(x), \delta(Y)}}_{=0} \right)\\
&= \sum_{i=1}^r \delta\lie{f_i X_i, f_i Y}= \sum_{i=1}^r \lie{\delta(f_iX_i), f_i(x)Y(x)} + \lie{f_i(x)X_i(x), \delta(f_i Y)} \\&= \sum_{i=1}^r 0+0 = 0,
}
thus $\delta(X) \in \z(\L_x)=0$.
\item Suppose $\lie{\k,\k}=\k$. Since $\Gamma(\L)$ is also perfect, there exist, for each $i \in \set{1, \ldots , r}$, sections $Y_{i1}, Z_{i1}, \ldots , Y_{is}, Z_{is} \in \Gamma(\L)$ such that $X_i=\sum_{t=1}^s \lie{Y_{it},Z_{it}}$. Then:
\es{
\delta(X) &= \sum_{i=1}^r \delta \left(f_i^2 X_i \right) = \sum_{i=1}^r \sum_{t=1}^s \delta \lie{f_i Y_{it}, f_i Z_{it}}\\
& = \sum_{i=1}^r \sum_{t=1}^s \lie{\delta(f_iY_{it}), f_i(x)Z_{it}(x)} + \lie{f_i(x)Y_{it}(x), \delta(f_i Z_{it})} = \sum_{i=1}^r \sum_{t=1}^s 0= 0.
}
\end{itemize}
For both cases, we conclude:
\begin{equation}\label{order01only}
\delta(X) =0 \text{ if } j^1_x(X)=0.
\end{equation}

Let $(U,\ph)$ be a bundle chart in $x$ and $(U,\xi)$ a corresponding chart of $M$.
Since $\pi: \L \to M$ is defined by an $\Aut(\k)$-valued cocyle, $\ph$ restricted to any fiber is an isomorphism of Lie algebras, therefore the Lie bracket on $\Gamma\left(\pi^{-1}(U)\right)$ corresponds to the natural pointwise Lie bracket on $\Ci(U,\k)$ meaning $\lie{X,Y}^\ph =[X^\ph, Y^\ph]$ for $X,Y \in \Gamma\left(\pi^{-1}(U)\right)$. Note, by relation \eqref{order01only}, that $\delta$ has a local form without terms of order greater than 1:
\begin{equation}\label{deltalocal}
X^\ph \overset{{\delta^\ph}}{\longmapsto} D\left({X^\ph}(x)\right) +
\sum_{i=1}^m S^i\left(\partial_{\xi_i} {X^\ph}(x)\right)
\end{equation}
for certain $D, S^1, \ldots , S^m \in \End(\k)$.

Let us show now that a local form of the type \eqref{deltalocal} is satisfied for an $x$-derivation $\delta^\ph$ \iff $D \in \Der(\k)$ and $S^1, \ldots , S^m \in \Cent(\k)$. Equation \eqref{deltalocal} implies the following two equations:
\begin{equation}\label{partone}
\begin{split}
\delta^\ph \lie{ X^\ph,  Y^\ph }
&= D\left(\lie{ X^\ph(x),  Y^\ph(x)} \right) + \sum_{i=1}^m S^i \left(\d_{\xi_i} \lie{ X^\ph,  Y^\ph}(x) \right)\\
&= D\left(\lie{ X^\ph(x),  Y^\ph(x)} \right) + \sum_{i=1}^m S^i \left(\lie{\d_{\xi_i} X^\ph(x),  Y^\ph(x)} + \lie{ X^\ph(x),  \d_{\xi_i}  Y^\ph(x)} \right).
\end{split}
\end{equation}
\begin{equation}\label{parttwo}
\begin{split}
\lie{\delta^\ph\left( X^\ph\right),  Y^\ph (x)} + \lie{ X^\ph (x),\delta^\ph\left( Y^\ph\right)} &=\lie{D\left({X^\ph}(x)\right), Y^\ph (x)} + \lie{ X^\ph (x), D\left({Y^\ph}(x)\right)}\\
&+ \sum_{i=1}^m \lie{S^i\left(\partial_{\xi_i} {X^\ph}(x)\right), Y^\ph (x)}+
\lie{ X^\ph (x), S^i\left(\partial_{\xi_i} {Y^\ph}(x)\right)}.
\end{split}
\end{equation}
By comparing the equations \eqref{partone} and \eqref{parttwo} for constant sections $X,Y$ we obtain
\es{
D\left(\lie{ X^\ph(x),  Y^\ph(x)} \right) = \lie{D\left({X^\ph}(x)\right), Y^\ph (x)} + \lie{ X^\ph (x), D\left({Y^\ph}(x)\right)},
}
thus $D$ is a derivation of $\k$. Knowing this we compare the equations \eqref{partone} and \eqref{parttwo} for constant $X$ and arbitrary $Y$ and obtain, after the cancelation of the derivation part:
\es{
\sum_{i=1}^m S^i \lie{ X^\ph(x),  \d_{\xi_i}  Y^\ph(x)} = \sum_{i=1}^m \lie{ X^\ph (x), S^i\left(\partial_{\xi_i} {Y^\ph}(x)\right)}.
}
Since for any $i \in \set{1, \ldots , m}$ we have
\es{
\partial_{\xi_i} {Y^\ph}(x)=\left. \frac{d}{dt} \right|_{t=0}  {Y^\ph}\left(\xi^{-1}(\xi(x) + t e_i)\right)=\left(T_x Y^\ph\right) \left[\left(T_{\xi(x)}\xi^{-1} \right)(e_i)\right]
}
and $T_{\xi(x)}\xi^{-1}$ is a linear bijection, there is, for any $x \in U$ and $v \in \k$, a smooth map $Y^\ph: U \to \k$ with $\partial_{\xi_i} {Y^\ph}(x)= \delta_{ij} v$ for $j \in \set{1, \ldots , m}$. Therefore we may conclude that each $S^i$ is contained in the centroid of $\k$.

We obtain a bijective correspondence:
\begin{equation}\label{bijcorr}
\begin{split}
\bigcup_{x \in U} \D_x\left(U\times \k\right) &\lra  U \times \left(\Der(\k) \times \Cent(\k)^m \right)\\
\delta^\ph &\longmapsto \left(\overline{p}\left(\delta^\ph\right),\left(D,\left(S^1, \ldots , S^m \right) \right) \right).
\end{split}
\end{equation}
By extending this construction to all charts of a maximal bundle atlas of $M$, we construct trivializations of $\D(\L)$. In order to show that we obtain a bundle structure on $\D(\L)$ we will show that the transitions are linear: Let $(U,\ph)$ and $(V,\psi)$ be bundle charts of $\pi: \L \to M$ in $x$ with corresponding charts $(U,\xi)$ and $(V,\eta)$ of $M$ with $\xi(x) = \eta(x)=0$ and let $D, \wt D  \in \Der(\k)$, $S^1, \ldots , S^m, \wt S^1, \ldots , \wt S^m \in \Cent(\k)$ be the endomorphisms of $\k$ such that for all $X^\ph, X^\psi \in \Ci(U \cap V, \k)$ we have
\es{
\delta^\ph \left(X^\ph \right) =  D\left({X^\ph}(x)\right) + \sum_{i=1}^m S^i\left(\partial_{\xi_i} {X^\ph}(x)\right)
}
and
\es{
\delta^\psi \left(X^\psi \right) =   \wt D\left({X^\psi}(x)\right) + \sum_{i=1}^m \wt S^i\left(\partial_{\eta_i} {X^\psi}(x)\right).}
By the definition of the local forms, we have
\es{
\left(\left(\ph_{|{\L_x}}\right)^{-1} \circ \delta^{\ph} \right)\left(X^\ph \right) = \delta(X) = \left(\left(\psi_{|{\L_x}}\right)^{-1} \circ \delta^{\psi} \right)\left(X^\psi \right)
}
and
\es{
\left(\ph_{|{\L_x}}\right)^{-1}\left(X^\ph(x) \right) = X(x) = \left(\psi_{|{\L_x}}\right)^{-1}\left(X^\psi(x) \right)
,}
thus
\es{
\delta^{\psi}\left(X^\psi \right) = g(x) \left( \delta^\ph \left(X^\ph \right) \right)
\text{ and }
X^\ph(x) = g(x)^{-1}\left(X^\psi \right)
}
for a transition map $g: U \cap V \to \Aut(\k)$ of the principal bundle to which $\pi: \L \to M$ is associated. By considering the commutative diagram
\kD{
\k = T_{\ph(x)} \k \ar[d]_{g(x)} & & T_x M \ar[ll]_{T_x X^{\ph}} \ar[r]^{T_x \xi}& \R^m\\
\k = T_{\psi(x)} \k & & T_x M \ar[ru]_{T_x \eta} \ar[ll]^{T_x X^{\psi}} & ,
}
we may calcuate
\es{
\partial_{\xi_i}X^\ph(x) &= T_x X^\ph \left(T_0 \xi^{-1} (e_i) \right) = \left(g(x)^{-1} \circ T_x X^\psi \circ T_0 \eta^{-1} \circ T_x \xi \right) \left(T_0 \xi^{-1} (e_i) \right)\\
&=\left(g(x)^{-1} \circ T_x X^\psi \right) \left(T_0 \eta^{-1} (e_i) \right) = g(x)^{-1} \left(\partial_{\eta_i}X^\psi(x) \right)
}
and obtain, by the definition of $D, \wt D, S^1, \ldots , S^m, \wt S^1, \ldots , \wt S^m $:
\begin{equation}\label{schlangen}
\begin{split}
\wt D \left( X^\psi(x) \right) + \sum_{i=1}^m \wt S^i\left(\partial_{\eta_i} {X^\psi}(x)\right) &= \delta^\ph \left(X^\ph \right) = g(x) \left( \delta^\ph \left(X^\ph \right) \right) = g(x) \left( D\left({X^\ph}(x)\right) + \sum_{i=1}^m S^i\left(\partial_{\xi_i} {X^\ph}(x)\right) \right)\\
&=\left\{g(x) \circ D \circ g(x)^{-1} \right\} \left( X^\psi(x) \right) + \sum_{i=1}^m \left\{ g(x) \circ S^i \right\}\left(\partial_{\xi_i} {X^\ph}(x)\right)\\
&=\left\{g(x) \circ D \circ g(x)^{-1} \right\} \left( X^\psi(x) \right) + \sum_{i=1}^m \left\{ g(x) \circ S^i \circ g(x)^{-1} \right\}\left(\partial_{\eta_i} {X^\psi}(x)\right)
.
\end{split}
\end{equation}
For contstant sections $X$ this equation implies
\es{
\wt D \left( X^\psi(x) \right) = \left\{g(x) \circ D \circ g(x)^{-1} \right\} \left( X^\psi(x) \right),
} thus $\wt D =g(x) \circ D \circ g(x)^{-1}$. Knowing this, equation \eqref{schlangen} also implies
\es{
\sum_{i=1}^m \wt S^i\left(\partial_{\eta_i} {X^\psi}(x)\right) = \sum_{i=1}^m \left\{ g(x) \circ S^i \circ g(x)^{-1} \right\}\left(\partial_{\eta_i} {X^\psi}(x)\right)
} and, by inserting appropriate sections for $X$, we may conclude that $\wt S^i = g(x) \circ S^i \circ g(x)^{-1}$ for all $i \in \set{1, \ldots , m}$. So the transitions $D \mapsto \wt D$ and $S^i \to \wt S^i$ are linear isomorphisms and we have, by Theorem \ref{vectobundlebycocyles}, a vector bundle $\overline p: \D(\L) \to M$ with fiber $\Der(\k) \times \Cent(\k)^m$. There is an isomorphism of Lie algebras $\Phi: \Der(\Gamma(\L)) \to \Gamma(\D(\L))$, defined as follows:
If $D \in \Der(\Gamma(\L))$, then $\Phi(D)$ is the map $M \to \D(\L)$, $x \mapsto \ev_x \circ D$. If $\mathfrak{X} \in \Gamma(\D(\L))$, $X \in \Gamma(\L)$ and $x \in M$, then $\left(\Phi^{-1}(\mathfrak{X})X\right)_x = \left(\mathfrak{X}_x\right)X$.

Let us define the map $\sigma: \D(\L) \to TM \otimes \Cent(\L)$. If $\delta \in \D_x(\L)$, then we define $\sigma(\delta) \in T_xM \otimes \End(\L_x)$ as follows: We identify $T_xM \otimes \End(\L_x) = \Hom(T_xM^*,\End(\L_x))$ in the natural way, i.e. the element $\sum_i v_i \otimes A_i$ corresponds to the endomorphism $\alpha \mapsto \sum_i \alpha(v_i) \cdot A_i$. For any $\alpha \in T_xM^*$ and $w \in \L_x$ let $a \in C^\infty(M,\R)$ and $A \in \Gamma(\L)$ such that $T_x a = \alpha$ and $A(x) = w$. Then $\sigma(\delta)$ is pointwisely defined by
\es{
\left\{\sigma(\delta) (\alpha) \right\}(w) := \delta(aA)-a(x)\delta(A).
}
Now let us check that $\sigma$ is well-defined and valued in $TM \otimes \Cent(\L)$ by using \eqref{order01only} and \eqref{bijcorr}:
\es{
\left(\left\{\sigma(\delta) (\alpha) \right\}(w)\right)^\ph &= \left(\delta(aA)-a(x)\delta(A)\right)^\ph\\
&= D(a(x) A^\ph(x)) + \sum_{i=1}^m S^i(\d_{\xi_i}(aA^\ph)(x)) - a(x) \left(D(A^\ph(x)) +  \sum_{i=1}^m S^i(\d_{\xi_i}A^\ph(x)) \right)\\
&= \sum_{i=1}^m S^i \left(\d_{\xi_i}a(x) \cdot A^\ph(x) +a(x)\d_{\xi_i}A^\ph(x) -a(x)\d_{\xi_i}A^\ph(x) \right)\\
&= \sum_{i=1}^m S^i \left(\d_{\xi_i}a(x) \cdot \ph(w)\right) = \sum_{i=1}^m \d_{\xi_i}a(x) \cdot  S^i \left( \ph(w)\right).
}
By leaving the local coordinates we see that $\left\{\sigma(\delta) (\alpha) \right\}$ only depends on $\alpha$ and $w$ and we also obtain $\sigma(\delta) \in T_xM \otimes \Cent(\L_x)$. The above calculation also shows that $\sigma$ maps exactly the elements without local part in $\Cent(\k)^m$ to zero, i.e. $\im i = \ker \sigma$.

We finally show the surjectivity of $\sigma: \D(\L) \to TM \otimes \Cent(\L)$: Let $x \in M$ and $(U,\ph)$ a bundle chart in $x$ with corresponding chart $(U,\xi)$ of $M$ and $\rho: M \to [0,1]$ a smooth map with $\supp \rho \sbs U$ and $\rho_{|W} \equiv 1$ for an $x$-neighbourhood $W \sbs U$ and $(S^1, \ldots , S^m) \in \Cent(\k)^m$. We define $f_i:= \left(\ph_{|\L_x}\right)^{-1} \circ S^i \circ \ph \in \Cent (\L_x)$ and $\delta \in \D_x(\L)$ is defined by $\delta(X):= \rho \cdot \left(\ph_{|{\L_x}}\right)^{-1} \left( \sum_{i=1}^m S^i\left(\partial_{\xi_i} {X^\ph}(x)\right) \right)$ for $X \in \Gamma(\L)$. Then the local part of $\sigma(\delta)$ is $(S^1, \ldots , S^m)$.
\end{proof}

The short exact sequence in Theorem \ref{dertheo} naturally induces another short exact sequence.
\begin{cor}\label{Gammaexactsequence}
The following induced sequence is exact:
\es{
0 \lra \Gamma(\Der \left(\L\right)) \overset{i}{\lra} \Gamma(\D(\L)) \overset{\sigma}{\lra} \Gamma(TM \otimes \Cent(\L)) \lra 0,
}
where $(i(X))(x):=i(X(x))$ and $(\sigma(X))(x):=\sigma(X(x))$.
\end{cor}

\begin{defi}
The map $\sigma: \D(\L) \to TM \otimes \Cent(\L)$ from Theorem \ref{dertheo} is called \emph{symbol map}.
\end{defi}

For a global description of $\Der(\Gamma(\L))$ we firstly extend the notion of
Lie connections. Fix a Lie connection $\nabla: \Gamma(TM) \to \Der(\Gamma(\L))$.

\begin{rem}
The symbol of a derivation $\nabla_X$ of $\Gamma(\L)$ for a vector field $X \in \Gamma(TM)$ is $X \otimes \1$: Let $x \in M$ be a point, $a \in C^\infty(M,\R)$ a smooth map and $A \in \Gamma(\L)$ a section of $\L$. Then we identify $\nabla_X \in \Der(\Gamma(\L))$ with $\left(y \mapsto \ev_y \circ \nabla_X\right) \in \Gamma(\D(\L))$ and calculate:
\es{
\left[\sigma\left(\nabla_X (x) \right) (T_x a) \right](A_x)
&= \left(\nabla_X (x) \right)(aA)-a(x)\left(\nabla_X (x) \right)(A) = \left(\nabla_X (aA) \right)(x) - a(x)\left(\nabla_X A \right)(x)\\
&= \left((X.a)A + a \nabla_X (A) \right)(x) - \left(a \nabla_X A \right)(x) = \left((X.a)A\right)(x) = (T_x a)(X_x) \cdot A_x
}
and therefore
\es{
\sigma\left(\nabla_X (x) \right) = X_x \otimes \1_x.
}
Thus, by the short exact sequence of Corollary \ref{Gammaexactsequence}, we see that not all derivations of $\Gamma(\L)$ can be described like that. It is necessary to extend $\nabla$ in an appropriate way.
\end{rem}

\begin{defi}
The \emph{extension} of $\nabla$ to $\Gamma(TM \otimes \Cent(\L))$ is the linear map
\es{
\nabla: \Gamma(TM \otimes \Cent(\L)) &\lra \Der(\Gamma(\L))=\Gamma(\D(\L))\\
\Y &\longmapsto \nabla_{\Y}
}
defined as follows: Let $\Y \in \Gamma(TM \otimes \Cent(\L))$ be a section and $(U,\xi)$ a chart of $M$. Then there are unique elements $S^1, \ldots  , S^m \in \Gamma(\Cent(\pi^{-1}(U)))$ such that
\es{
\Y_{|U} = \sum_{i=1}^m \d_{\xi_i} \otimes S^i.
}
We define $\nabla_{\Y}$ locally by
\begin{equation}\label{defext}
\left(\nabla_{\Y}\right)_{|U}:= \sum_{i=1}^m S^i \circ \nabla_{\d_{\xi_i}}.
\end{equation}
We show that this is well-defined: Let $(U,\xi)$ and $(V,\eta)$ be two charts of $M$ and $x \in U \cap V$. Then $\Y_x$ takes the form $\sum_{i=1}^m \left(\d_{\xi_i}\right)_x \otimes s^i = \sum_{j=1}^m \left(\d_{\eta_j}\right)_x \otimes r^j$ for endomorphisms $s^1, \ldots  , s^m, r^1, \ldots , r^m \in \Cent(\L_x)$. Note that, by the uniqueness of these forms, we can write
\es{
r^j = \Y_x(T_x \eta_j).
}
By writing the matrix $T_{\eta(x)}\left(\xi \circ \eta^{-1}\right)$ as $A=(a_{ij})$ we find the following relation between the basis vectors $\left(\d_{\xi_i}\right)_x$ and $\left(\d_{\eta_j}\right)_x$:
\es{
\left(\d_{\eta_j}\right)_x &= T_x \left(\xi^{-1} \circ \xi \right) T_{\eta(x)}\eta^{-1} (e_j)=T_{\xi(x)}\xi^{-1} (Ae_j) = T_{\xi(x)}\xi^{-1} \left(\sum_{k=1}^m a_{kj} e_k\right)
=\sum_{k=1}^m a_{kj} T_{\xi(x)}\xi^{-1} (e_k)\\
&=\sum_{k=1}^m a_{kj} \left(\d_{\xi_k}\right)_x.
}
We write $T_{\xi(x)}\left(\eta \circ \xi^{-1}\right) = A^{-1}=(a^{ij})$ and calculate:
\es{
\sum_{j=1}^m r^j \circ \nabla_{\left(\d_{\eta_j}\right)_x}
&=\sum_{j=1}^m \Y_x(T_x \eta_j) \circ \sum_{k=1}^m a_{kj} \nabla_{\left(\d_{\xi_k}\right)_x}=\sum_{j,k=1}^m \left(\sum_{i=1}^m \left(\d_{\xi_i}\right)_x \otimes s^i\right)(T_x \eta_j) \circ a_{kj}\nabla_{\left(\d_{\xi_k}\right)_x}\\
&=\sum_{i,j,k=1}^m (T_x \eta_j)\left(\left(\d_{\xi_i}\right)_x \right) \cdot s^i \circ a_{kj} \nabla_{\left(\d_{\xi_k}\right)_x}
=\sum_{i,j,k=1}^m a^{ij}\cdot a_{kj} \cdot s^i \circ \nabla_{\left(\d_{\xi_k}\right)_x}\\
&=\sum_{i=1}^m s^i \circ \nabla_{\left(\d_{\xi_i}\right)_x}.
}
So we have a well-defined extension.
\end{defi}

\begin{lem}
The extension of $\nabla$ satisfies the following properties for all $\Y \in \Gamma(TM \otimes \Cent(\L))$, $a \in \Ci(M,\R)$, $A,B \in \Gamma(\L)$, $X \in \Gamma(TM)$, $S \in \Gamma(\Cent(\L))$ and $x \in M$:
\begin{enumerate}
\item $\nabla_{\Y} \lie{A,B} = \lie{\nabla_{\Y} A , B} + \lie{A, \nabla_{\Y} B}$, i.e. $\nabla_{\Y} \in \Der(\Gamma(\L))$,
\item $\left(\nabla_{\Y} (aA) \right)_x = \left( \Y_x (T_x a) \right) \{A\} + \left(a \nabla_{\Y} A \right)_x $, i.e. $\nabla_\Y (aA) = (\Y.a)\{A\} + a \nabla_\Y A$, i.e. $\sigma\left(\nabla_{\Y}\right)=\Y$,
\item $\nabla_{a\Y} A = a \nabla_{\Y} A$,
\item $\nabla_{X \otimes S}A = S \circ \nabla_{X}A $. In particular, we have $\nabla_{X \otimes \1} = \nabla_X$.
\end{enumerate}
\end{lem}
\begin{proof} Let $(U,\xi)$ be a chart of $M$. We show the properties by concluding locally:
\begin{enumerate}
\item
\es{
\nabla_{\Y}\lie{A, B} &= \sum_{i=1}^m S^i \circ \nabla_{\d_{\xi_i}} \lie{A, B}
 = \sum_{i=1}^m S^i \circ \left( \lie{\nabla_{\d_{\xi_i}} A, B} +  \lie{A,  \nabla_{\d_{\xi_i}} B} \right)\\
&=\sum_{i=1}^m S^i \circ  \lie{\nabla_{\d_{\xi_i}} A, B}  + \sum_{i=1}^m S^i \circ  \lie{A,  \nabla_{\d_{\xi_i}} B}  = \lie{\nabla_{\Y}A, B} + \lie{A, \nabla_{\Y} B}.
}
\item
\es{
\left(\nabla_{\Y} (aA) \right)_x &= \sum_{i=1}^m S^i(x) \circ \left( \nabla_{\d_{\xi_i}} aA \right)_x
=  \sum_{i=1}^m S^i(x) \circ \left(\left(\d_{\xi_i}a\right) A \right)_x + \sum_{i=1}^m S^i(x) \circ \left(a \nabla_{\d_{\xi_i}} A \right)_x\\
&= \Y_x (T_x a) [A] + \left(a \nabla_{\Y} A \right)_x.
}
\item If $\Y = \sum_{i=1}^m \d_{\xi_i} \otimes S^i$ then $a\Y = \sum_{i=1}^m \d_{\xi_i} \otimes a S^i$, thus:
\es{
\nabla_{a \Y}= \sum_{i=1}^m a S^i \circ \nabla_{\d_{\xi_i}} = a \sum_{i=1}^m S^i \circ \nabla_{\d_{\xi_i}} = a \nabla_{ \Y}.
}
\item If $X = \sum_{i=1}^m X^i \d_{\xi_i}$ then
\es{
X \otimes S = \left(\sum_{i=1}^m X^i \d_{\xi_i} \right) \otimes S = \sum_{i=1}^m \d_{\xi_i} \otimes X^i S
}
implying
\es{
\nabla_{X \otimes S}A = \sum_{i=1}^m X^i S \circ \d_{\xi_i} A = S \circ \nabla_{X}A.
}
\end{enumerate}
\end{proof}

\begin{rem} \
\begin{enumerate}
\item It is possible to define the extension of a Lie connection $\nabla: \Gamma(TM) \to \Der(\Gamma(\L))=\Gamma(\D(\L))$ without the help of charts of $M$: For $x \in M$ and $X \in \Gamma(TM)$ the $x$-derivation $\left(\nabla_X\right)_x$ depends on $X$ only via $X_x$ and so we can understand $\nabla$ as a map $TM \to \D(\L)$. We now extend $\nabla$ in a natural way to a map $\nabla: TM \otimes \Cent(\L) \to \D(\L)$ by mapping $X_x \otimes S_x$ to $S_x \circ \left(\nabla_X\right)_x$ (cf. Equation \eqref{defext}).
\item It is not difficult to show that any linear map $L: \Gamma(TM \otimes \Cent(\L)) \to \Der(\Gamma(\L))$ satisfying the four conditions of the preceding lemma is the extension of exactly one Lie connection $\nabla$.
\end{enumerate}
\end{rem}

We can now formulate and prove a theorem about the global structure of $\Der(\Gamma(\L))$.
\begin{thm}\label{globalderi}
Fix a Lie connection $\nabla$ on $\L$. If $\k$ is perfect or centerfree,
then a linear function $D: \Gamma(\L) \to \Gamma(\L)$ is a derivation of
$\Gamma(\L)$ \iff there are sections $\Y \in \Gamma(TM \otimes \Cent(\L))$ and
$D_0 \in \Gamma(\Der(\L))$ such that $D=\nabla_{\Y} + i(D_0)$ and this
decomposition is unique. \end{thm} \begin{proof} Obviously, for any $\Y \in
\Gamma(TM \otimes \Cent(\L))$ and $D_0 \in \Gamma(\Der(\L))$ the linear map
$\nabla_{\Y} + i(D_0)$ is a derivation of $\Gamma(\L)$.

On the other hand, fix $D \in \Der(\Gamma(\L))=\Gamma(\D(\L))$ and set $\Y:=\sigma(D) \in \Gamma(TM \otimes \Cent(\L))$. Since $\sigma\left( \nabla_{\Y} \right) = \Y$, the linear map $D-\nabla_{\Y}: \Gamma(\L) \to \Gamma(\L)$ has symbol 0 and thus identifies with a differential operator of order zero which is a section of $\Hom(\L,\L)$. Since $D$ and $\nabla_{\Y}$ are derivations of $\Gamma(\L)$, this section takes its values in $\Der(\Gamma(\L))$. Therefore $D-\nabla_{\Y}$ can be written as $i(D_0)$ for a $D_0 \in \Gamma(\Der(\L))$.

The uniqueness of the decomposition follows from the fact that $\Y= \sigma(D)$ is the only element in $\Gamma(TM \otimes \Cent(\L))$ such that $\sigma(\nabla_{\Y}) = \Y$.
\end{proof}

\subsubsection{Centroid}
In this subsection we will calculate the centroid of $\Gamma^k(\L)$. We begin 
with a technical lemma.

\begin{lem}\label{centroidtheorem1lemma}
Let $X \in \Gamma(\L)$ be a section which is zero on an open set $U \sbs M$ and $x \in U$. Then there exists a compact $x$-neigbourhood $F$ and there are functions $D_t \in \Der (\Gamma(\L))$, $X_t \in \Gamma(\L)$ with ${X_t}_{|F} \equiv 0$ for $t \in \set{1, \ldots , r}$ such that
\es{
X = \sum_{t=1}^r D_t(X_t).
}
\end{lem}

\begin{proof}
Let $F$ be a compact $x$-neigbourhood contained in $U$. Since $M$ is paracompact, there exists a Palais cover $\left(\left(V_i, \psi_i, \xi_i, \rho_i\right)_{i \in J}, (J_t)_{t=1}^r\right)$ refining the open cover $\set{U,M \backslash F}$ of $M$ (cf. Definition \ref{Palais} and Theorem \ref{Palaisexi}).
For any $i \in J$ and $Y \in \Gamma\left(\pi^{-1}(V_i)\right)$ recall that $Y^{\psi_i}: \R^m \sps \xi_i(V_i) \to \k$ denotes the mapping such that the following diagram commutes:
\kD{
**[l]\L \sps \pi^{-1}(V_i) \ar[r]^{(\xi_i \circ \pi, \psi_i)} &  **[r] \xi_i(V_i) \times \k \\
**[l]M \sps V_i \ar[r]_{\xi_i} \ar[u]^{Y} 							& **[r] \xi_i(V_i) \ar[u]_{\left(\id,Y^{\psi_i}\right)}.
}
The Lie bracket on $\Gamma\left(\pi^{-1}(V_i)\right)$ corresponds to the natural pointwise Lie bracket on $\Ci(\xi_i(V_i),\k)$, i.e. $\lie{Y,Z}^{\psi_i} =[Y^{\psi_i}, Z^{\psi_i}]$ for $Y,Z \in \Gamma\left(\pi^{-1}(V_i)\right)$, for all $i \in J$. We set $Y:=\rho_i X$ and choose a section $X^i \in \Gamma\left(\pi^{-1}(V_i)\right)$ by claiming $\partial_{(\xi_i)_1} (X^i)^{\psi_i} = Y^{\psi_i}$. With $D^i: \Gamma\left(\pi^{-1}(V_i)\right) \to \Gamma\left(\pi^{-1}(V_i)\right)$ defined by $(D^i)^{\psi_i}:=\partial_{(\xi_i)_1}$ one obtains
\es{
\rho_i X = Y = D^i(X^i).
}
Note that in the case of $\supp(\rho_i) \sbs U$ we can choose $X^i \equiv 0$ because $X_{|U} \equiv 0$ and each $V_i$ is a subset of exactly one, $U$ or $M\backslash F$. By setting $D_t := \sum_{i \in J_t} D^i$ and $X_t := \sum_{i \in J_t} X^i$ we obtain a well-defined derivation $D_t$ and a well-defined section $X_t$ for any $t \in \set{1, \ldots , r}$ and
\es{
\sum_{t=1}^r D_t(X_t) = \sum_{t=1}^r \left( \sum_{i \in J_t} D^i \left(\sum_{i \in J_t} X^i \right) \right) = \sum_{t=1}^r \left(\sum_{i\in J_t} \rho_i X \right)=X.
}
\end{proof}

\begin{thm}\label{centroidtheorem1}
Any endomorphism $T: \Gamma(\L) \to \Gamma(\L)$ with $T \circ \Der(\Gamma(\L)) \sbs \Der(\Gamma(\L))$ is a differential operator of order zero and can be identified with a smooth section of $\Hom(\L,\L)$.
\end{thm}
\begin{proof}
We want to show that $T$ is local. Let $X \in \Gamma(\L)$ be zero on an open set $U \sbs M$. Lemma \ref{centroidtheorem1lemma} yields that there exist $D_1, \ldots , D_r \in \Der(\Gamma(\L))$ and $X_1, \ldots , X_r \in \Gamma(\L)$, each $X_t$ zero on an $x$-neighbouhood $F$, such that $X=\sum_{t=1}^r D_t(X_t)$. We obtain
\begin{equation}\label{TAnull}
TX_{|F} = \sum_{t=1}^r \left((T \circ D_t)(X_t)\right)_{|F} = \sum_{t=1}^r 0 = 0
\end{equation}
because any  $D \in \Der(\Gamma(\L))$ can be identified with a section of $\D(\L)$ by $D \mapsto X_D$, where $X_D(x):= \ev_x \circ D$, and hence is local. By \eqref{TAnull} we conclude that $T$ is local.

Theorem \ref{Peetre} yields that $T$ is a differential operator. By Theorem \ref{dertheo}, any derivation $D \in \Der(\Gamma(\L))$ is a differential operator of order at most one and so is any $T \circ D$. If $T$ were of order $n > 0$, then there would be a bundle chart $(U,\ph)$ with corresponding chart $(U,\xi)$ of $M$ and functions $f_{\alpha} \in \Ci\left(U, \End(\k) \right)$, $\alpha \in \N^m$, $\abs{\alpha} \leq n$, such that for all $X\in \Gamma\left(\pi^{-1}(U)\right)$ we would have
\es{
(TX)^{\ph}=\sum_{\abs{\alpha} \leq n} f_{\alpha} \cdot \left(\partial_{\xi}^{\alpha}X^{\ph}\right)
}
and there would be $x' \in U$, $u \in \k$ and $\alpha'=(\alpha'_1, \ldots , \alpha'_m) \in \N^m$ with $\abs{\alpha'} = n$ such that $f_{\alpha'}(x')(u) \ne 0$. We may assume, without loss of generality, that $\xi(x')=0$. By defining $D \in \Der\left(\Gamma\left(\pi^{-1}(U)\right)\right)$ by $D^{\ph}:=\d_{\xi_1}$ and $X \in \Gamma\left(\pi^{-1}(U)\right)$ by $X^{\ph}(x):=\xi_1(x) \cdot \prod_{i=1}^m\xi_i(x)^{\alpha'_i}  \cdot u$ we would obtain
\es{
(T(DX))^{\ph}
= \sum_{\abs{\alpha} \leq n} f_{\alpha} \cdot \left(\partial_{\xi}^{\alpha}(DX)^{\ph}\right)
= \sum_{\abs{\alpha} \leq n} f_{\alpha} \cdot \underbrace{\left(\partial_{\xi}^{\alpha}\d_{\xi_1} X^{\ph}\right)}_{= 0 \text{ in } x' \text{, if }\alpha \ne \alpha'}\\
\implies (T(DX))^{\ph}(x')= f_{\alpha'}(x') \left(\alpha' !  \cdot u \right) \ne 0,
}
contradicting to the facts that $T \circ D$ is a differential operator of order at most one and $j_{x'}^1(X)=0$. So $T$ is of order zero.
\end{proof}

\begin{thm}\label{Gammacentroid}
If $\k$ is perfect or centerfree, then $\Cent\left(\Gamma^k(\L)\right)\cong\Gamma^k\left(\Cent(\L)\right)$ as associative algebras.
\end{thm}
\begin{proof}
In the smooth case we use Lemma \ref{Liealglemma}.2 and Theorem \ref{centroidtheorem1} to see that all $T \in \Cent(\Gamma(\L))$ identify with smooth sections of $\Hom(\L,\L)$.

For $k \in \N$ we will firstly show that any $A \in \Cent(\Gamma^k(\L))$ is local: Let $X \in \Gamma^k(\L)$ be zero on an open set $U \sbs M$.
\begin{itemize}
\item Suppose $\z(\k)=0$. Then for all $x \in U$ and $Y \in \Gamma^k(\L)$ we have:
\es{
\lie{(AX)_x, Y_x} = \lie{X_x,(AY)_x} = 0.
}
This implies $(AX)_x \in \z(\L_x) = 0$ for all $x \in U$.
\item Suppose $\lie{\k,\k}=\k$. Then, by Corollary \ref{perfectcoro}, the Lie algebra $\Gamma^k(\L)$ is also perfect. So there exist sections $Y_1, Z_1, \ldots , Y_r, Z_r \in \Gamma^k(\L)$ such that $X=\sum_{t=1}^r \lie{Y_t,Z_t}$. Now, if $x \in U$ then there is an $x$-neigbourhood $W \sbs U$ and a smooth map $\rho: M \to [0,1]$ such that $\rho_{|W} \equiv 1$ and $\rho_{M\backslash U} \equiv 0$. We set $X':=(1-\rho) \cdot X$, thus $X = X'$ on $M$ and we obtain $X= \sum_{t=1}^r \lie{Y_t, (1-\rho) \cdot  Z_t}$, yielding
\es{
(AX)_x = \sum_{t=1}^r \lie{(A Y_t)(x),(1-\rho)(x) \cdot Z_t(x)} = 0.
}
\end{itemize}
In both cases we have $(AX)_{|U} \equiv 0$, thus $A$ is local. So we can use Theorem \ref{PseudoPeetre} to see that all $T \in \Cent(\Gamma^k(\L))$ identify with $\Ck$-sections of $\Hom(\L,\L)$.

We know now that, for $k \in \Ni$, there is a linear map $\Psi: \Cent\left(\Gamma^k(\L)\right) \to \Gamma^k\left(\Hom(\L,\L)\right)$ well-defined by $\Psi(A)(X_x):=(AX)_x$ for $X \in \Gamma^k(\L)$, $x \in M$. Since we have
\es{
\Psi(A)\lie{X_x,Y_x}=(A\lie{X,Y})_x = (\lie{AX,Y})_x  = \lie{\Psi(A)(X_x),Y_x}
} for $X,Y \in \Gamma^k(\L)$ and $x \in M$, the image of $\Psi$ is contained in $\Gamma^k(\Cent(\L))$. Furthermore, the map $\Psi$ is compatible with the composition of functions on $\Cent\left(\Gamma^k(\L)\right)$ and $\Gamma^k(\Cent(\L))$:
\es{
\Psi(A \circ A')(X_x) = \left((A \circ A')X\right)_x =  \left(A(A'X)\right)_x = (\Psi(A) \circ \Psi(A'))(X_x).
}
Finally, the function $\Phi: \Gamma^k(\Cent(\L)) \to \Cent\left(\Gamma^k(\L)\right)$ defined by $\left(\Phi\left(\mathfrak{A}\right)(X)\right)_x:=\mathfrak{A}_x(X_x)$ for a section $\mathfrak{A} \in  \Gamma^k(\Cent(\L))$, $X \in \Gamma^k(\L)$, $x \in M$, is the inverse of $\Psi$. We conclude that $\Psi$ is an isomorphism of associative algebras.
\end{proof}

We can now deduce an important structure theorem.

\begin{thm}\label{Gammaindecomposable}
Let $\k$ be indecomposable. Then $\Gamma^k(\L)$ is so, if $\Hom(\k/\lie{\k,\k},\z(\k))=0$.
\end{thm}
\begin{proof}
Let $\Gamma^k(\L)=I_1 \oplus I_2$ be a decomposition into a direct sum of ideals and $P^1$ be the projection onto $I_1$ parallel to $I_2$. Note that $P^1 \in \Cent(\Gamma^k(\L)) = \Gamma^k(\Cent(\L))$ because for any $X,Y \in \Gamma^k(\L)$ decomposing into $X_1,Y_1 \in I_1$ and $X_2,Y_2 \in I_2$ we have:
\begin{equation}\label{projtheta}
P^1\lie{X,Y}=P^1(\lie{X_1,Y_1}+\lie{X_2,Y_2})=\lie{X_1,Y_1}=\lie{X_1+X_2,Y_1}=\lie{X,P^1Y}.
\end{equation}
Since $P^1 \circ P^1 = P^1$ on $\Gamma^k(\L)$, each $P^1_x \in \Cent(\L_x)$, where $x \in M$, is a projection in $\Cent(\L_x)$ and, by indecomposability of $\L_x \cong \k$, this implies $P^1_x = 0$ or $P^1_x = \1_x$. The map $M \to \Cent(\L_x) \sbs \Cent(\L)$, $x \mapsto P^1_x$  is continuous and for each bundle atlas $(U_i,\ph_i)_{i \in I}$, where each $U_i$ is connected, the mapping $x \longmapsto \ph_i \circ P^1_x$ has discrete image. But all this is only possible if $P^1=0$ or $P^1=\1$, i.e. $I_1=0$ or $I_2=0$. So $\Gamma^k(\L)$ is indecomposable.
\end{proof}

\begin{rem}
If $\Hom(\k/\lie{\k,\k}, \z(\k))=0$, then, by Lemma \ref{Liealglemma}.3, we have $\Cent(\k) = \operatorname{N}(\k) \oplus \operatorname{S}(\k)$ and, since $\Aut(\k)$ acts on $\operatorname{N}(\k)$ and $\operatorname{S}(\k)$ by conjugation, 
 we also have the following decomposition: $\Gamma^k(\Cent(\L)) = \Gamma^k(\operatorname{N}(\L)) \oplus \Gamma^k(\operatorname{S}(\L))$, where $\operatorname{N}(\L)$, $\operatorname{S}(\L)$ denote the corresponding subbundles $\L(\operatorname{N}(\k))$, $\L(\operatorname{S}(\k))$ of $\Cent(\L)=\L(\Cent(\k))$ (cf. Definition \ref{specialsubbundles}.2), respectively.
Consider the case $\operatorname{S}(\k) = \K \cdot \1$. Then
$\operatorname{S}(\L)$ is a bundle over $M$ with fiber $\K \cdot \1$ and, for any $x \in M$, the fiber $\operatorname{S}(\L)_x$ is naturally isomorphic to $\operatorname{S}(\L_x)$. Thus any section $X \in \Gamma^k\left(\operatorname{S}(\L)\right)$ takes the form $M \to \operatorname{S}(\L)$, $x\mapsto X_x = f(x) \1_x$ for some $f \in \operatorname{C}^k(M,\K)$, hence:
\begin{equation}\label{SGammakL}
\Gamma^k(\operatorname{S}(\L)) = \operatorname{C}^k(M,\K) \cdot \1.
\end{equation} \end{rem}

\begin{rem}
Let $\k$ be reductive and $\dim \z(\g) \le 1$, e.g., $\k= \gl_n(\K)$ or $\Hom(\k / \lie{\k,\k}, \z(\k)) = 0$.
If $\k= \bigoplus_{i=1}^n \k_i$ is a decomposition into a direct sum of indecomposable non-zero ideals, then we define $\operatorname{H}(\k):=\bigcap_{i=1}^n \set{\left. f \in \Aut(\k) \right| f(\k_i) = \k_i}$, the normal subgroup of $\Aut(\k)$ stabilizing this decomposition. This is well-defined because the decomposition into a direct sum of indecomposable non-zero ideals is, by Proposition \ref{uniquedecompo}, unique except for the order. We denote the quotient group $\Aut(\k) / \operatorname{H}(\k)$ by $\operatorname{G}(\k)$.

Let $\norm{\cdot}: \End(\k) \to \R_{+}$ be the Frobenius norm on $\End(\k)$ \wrt a basis $(v_1, \ldots , v_d)$ corresponding to the decomposition $\k = \bigoplus_{i=1}^n \k_i$, i.e. $\norm{f}:=\sqrt{\sum_{i,j=1}^d a_{ij}^2}$, where $f(v_j)=a_{ij} v_i$ for $i,j \in \set{1, \ldots , d}$. This norm induces, by $\dim \k = d < \infty$, the compact-open topology. So the subspace $\Aut(\k) \sbs \End(\k)$ becomes a topological group.

By Remark \ref{Hopen}, the matrices corresponding to the automorphisms in $\Aut(\k)$ are ``permutated block diagonal matrices''.  Therefore, if $f \in \operatorname{H}(\k)$ with $f(v_j)=a_{ij} v_i$ for $i,j \in \set{1, \ldots , d}$ and we set $\ep:= 0.5 \cdot \min_{i,j \in \set{1, \ldots , d}} a_{ij}$, then the $\ep$-ball around $f$ in $\Aut(\k)$ \wrt $\norm{\cdot}$ is entirely contained in $\operatorname{H}(\k)$. So $\operatorname{H}(\k) \sbs \Aut(\k)$ is open and $\operatorname{G}(\k)$ equipped with the corresponding quotient topology is discrete.
\end{rem}

The following general theorem about first \v Cech cohomology sets is Theorem V.5 of \cite{Ne04}. We will use it to show Lemma \ref{firstCechlemma} which will be useful for two structure theorems.

\begin{thm}\label{NeebsCechTheorem}
If $M$ is a smooth manifold and $q: A \to G$ is a smooth and surjective morphism of Lie groups with kernel $H$, then there exists an exact sequence of morphisms of pointed sets as follows:
\es{
1 \lra \Ci(M,H) \lra \Ci(M,A) \lra \Ci(M,G) \lra \Hc^1(M,H) \lra \Hc^1(M,A) \lra \Hc^1(M,G).
}
\end{thm}

\begin{lem}\label{firstCechlemma}
Let $A$ be a Lie group, $H \ide A$ open and $G:=A/H$ such that $\Hc^1(M,G)=1$. Then $\Hc^1(M,A)\cong\Hc^1(M,H)$.
\end{lem}
\begin{proof}
Since $\Hc^1(M,G)=1$, the map $\Hc^1(M,H) \to \Hc^1(M,A)$ is, by Theorem \ref{NeebsCechTheorem}, surjective.
\end{proof}

\begin{thm}\label{decotheo}
Let $\k= \bigoplus_{i=1}^n \k_i$ be a decomposition into a direct sum of indecomposable non-zero ideals of a perfect or centerfree \La $\k$ and let $\Hc^1(M,\operatorname{G}(\k))$ be trivial (e.g. $M$ simply connected).
Then there is a decomposition into a direct sum of indecomposable non-zero ideals $\Gamma^k(\L)=\bigoplus_{i=1}^n \Gamma^k(\L^i)$, where each $\pi_{|\L^i}: \L^i \to M$ is a subbundle of $\pi: \L \to M$ with fiber $\k_i$ and this decomposition is unique except for the order.
\end{thm}
\begin{proof}
Lemma \ref{firstCechlemma} yields $\Hc^1(M,\Aut(\k))\cong\Hc^1(M,\operatorname{H}(\k))$, so the bundle $\pi: \L \to M$ admits a cocycle in $\operatorname{H}(\k)$ and we can construct the subbundles $\pi_{|\L^i}: \L^i \to M$ with fiber $\k_i$. We also have \es{\Gamma^k(\L)=\bigoplus_{i=1}^n \Gamma^k(\L^i),} which is a decomposition into a direct sum of indecomposable non-zero ideals because each $\k_i$ is indecomposable (cf. Theorem \ref{Gammaindecomposable}). We conclude by showing the uniqueness of this decomposition (except for the order).

Let $\Gamma^k(\L)=I_1 \oplus I_2$ be a decomposition into a direct sum of ideals and $P^1$ be the projection onto $I_1$ parallel to $I_2$ and for $i \in \set{1, \ldots , n}$ let $P_i$ be the projection onto $\Gamma^k(\L^i)$ associated to the decomposition $\Gamma^k(\L)=\bigoplus_{i=1}^n \Gamma^k(\L^i)$. Calculations totally analogous to \eqref{projtheta} show that $P^1$ and each $P_i$ are in $\Cent(\Gamma^k(\L))=\Gamma^k(\Cent(\L))$. For all $x \in M$ we have $\lie{P_i,P^1}(x)=\lie{P_i(x),P^1(x)}=0$ because $\Cent(\L_x)$ is abelian by $\Hom(\k / \lie{\k,\k} , \z(\k))=0$ and Lemma \ref{Liealglemma}.3. Therefore $P^1$ commutes with $P_i$ for any $i\in\set{1, \ldots, n}$. This implies $P^1\left(\Gamma^k(\L^i)\right) = P_i(I_1)$ and, due to the indecomposability of the ideals, $P^1_{|\Gamma^k(\L^i)} = 0$ or $P^1_{|\Gamma^k(\L^i)} = \1$ for any $i\in\set{1, \ldots, n}$, resulting in
\es{
I_1 = \bigoplus_{\ell=1}^t \Gamma^k\left(\L^{j_{\ell}}\right)
}
for some $1 \le j_1 < \ldots < j_t \le n$. Thus the decomposition $\Gamma^k(\L)=\bigoplus_{i=1}^n \Gamma^k(\L^i)$ is unique.
\end{proof}

A statement concerning complex structures of real Lie algebras of $\Ck$-sections is given by the following proposition.

\begin{prop}
Let $\K = \R$ and $\k$ be indecomposable with $\Hom(\k/\lie{\k,\k},\z(\k))=0$.
Then $\Gamma^k(\L)$ admits at most two complex structures and, if $\Hc^1(M,\Z/2\Z)\cong \Hom(\pi_1(M), \Z / 2\Z)$ is trivial (e.g. $M$ simply connected), then $\Gamma^k(\L)$ admits a complex structure \iff $\k$ does.
\end{prop}

\begin{proof}
A complex structure $J \in \Cent\left(\Gamma^k(\L)\right)=\Gamma^k(\Cent(\L))$
induces a complex structure $J_x$ on each fiber $\Cent(\L_x)$, thus on $\k$.
Lemma \ref{indecomplexreallemma}.2 implies that $\k$ possesses at most two
complex structures, thus, by the connectedness of $M$, there can be at most two
different $\operatorname{C}^k$-complex structures, namely $J$ and $-J$, on
$\Gamma^k(\L)$.

Let $J_0$ be a complex structure on $\k$. We turn $\k$ into a complex Lie algebra by defining the multiplication by a complex scalar via $(a+bi) \cdot x := ax + b J_0(x)$ for $a, b \in \R$ and $x \in \k$. If $\sigma \in \Aut_\R\left(\k\right)$, $a, b \in \R$ and $x \in \k$, then
\es{
\sigma((a+bi)\cdot x) = \sigma(ax + b J_0(x)) = a \sigma(x) + b \sigma(J_0(x)).
}
Thus $\sigma \in \Aut_\R\left(\k\right)$ is in $\Aut_\C\left(\k\right)$ \iff $\sigma \circ J_0 = J_0 \circ \sigma$. Otherwise we have $\sigma \circ J_0 \circ \sigma^{-1} = -J_0$ because $\left(\sigma \circ J_0 \circ \sigma^{-1}\right)^2 = -\1$ and $J_0$, $-J_0$ are the only complex structures on $\k$ by Lemma \ref{indecomplexreallemma}.2. So $\Aut_\C\left(\k\right)$ is a subgroup of index 2 of $\Aut_\R\left(\k\right)$, thus normal and $\Aut_\R\left(\k\right) / \Aut_\C\left(\k\right) \cong \Z / 2\Z$. Since $\Hc^1(M,\Z/2\Z)=1$, we have $\Hc^1(M,\Aut_\R(\k))\cong\Hc^1\left(M,\Aut_\C\left(\k\right)\right)$ by Lemma \ref{firstCechlemma}, yielding the existence of an $\Aut_\C\left(\k\right)$-valued cocyle of $\L$ and this turns $\pi: \L \to M$ into a bundle of Lie algebras with complex fiber $\k$ and $\Gamma^k(\L)$ into a complex Lie algebra.

Conversely, if $\Gamma^k(\L)$ admits a complex structure, then, by Theorem \ref{Gammacentroid}, this induces a complex structure on each fiber $\k$, thus on $\k$.
\end{proof}

\subsubsection{Isomorphisms}
Now we will discuss the isomorphisms of Lie algebras of $\Ck$-sections. In
the light of Theorem \ref{decotheo} it suffices to do this for indecomposable
Lie algebras only. The main key to the analysis of the isomorphisms is Lemma
\ref{monsterlemma}. We want to show some lemmas before. In this
subsection $M,N$ are smooth manifolds with positive dimensions $m,n$, 
respectively

\begin{lem}\label{Cktop}
For $k \in \Ni$ the topology on $M$ is the same as the initial topology of the maps in $\operatorname{C}^k(M,\K)$.
\end{lem}
\begin{proof}
Let $\O_M$ denote the topology on $M$. The initial topology of the maps in $\operatorname{C}^k(M,\K)$, denoted by $\O_i$, is the coarsest topology on $M$ such that each $f \in \operatorname{C}^k(M,\K)$ is continuous, so that $\O_i \sbs \O_M$. Let $V$ be a neighbourhood of $x \in M$ \wrt $\O_M$. Then, by paracompactness of $M$, there is a smooth function $\rho: M \to [0,1] \sbs \K$ such that $\rho_{|M \backslash V} \equiv 0$ and  $\rho_{|W} \equiv 1$ for a smaller $x$-neighbourhood $W \sbs V$. But this yields $W \sbs \rho^{-1}(B(0.5;1)) \sbs V$, where $B(r;z)$ denotes the open ball with radius $r$ around $z$ in $\K$. Thus $V$ is also a neighbourhood of $x \in M$ \wrt $\O_i$ and we conclude $\O_M \sbs \O_i$, so $\O_i = \O_M$.
\end{proof}

\begin{lem}\label{prop}
There exists a non-constant proper map $\ph \in \Ci(M,\K)$, i.e. $\ph$-preimages of compact subspaces are compact.
\end{lem}
\begin{proof}
Let $U \subsetneq M$ be a non-empty open subset with compact closure, $K \sbs U$ a non-empty compact subset and $\ph \in \Ci(M,[0,1]) \sbs \Ci(M,\K)$ a map such that $\supp \rho \sbs U$ and ${\rho}_{|K} \equiv 1$. For any compact $C \sbs \K$, the set $\ph^{-1}(C)$ is a closed subset of $\overline U$, hence compact.
\end{proof}

\begin{lem}\label{monsterlemma}
Let $v: \operatorname{C}^k(M,\K) \to \operatorname{C}^\ell(N,\K)$ be an isomorphism of associative algebras for some $0 \ne k,\ell \in \Ni$. Then $k=\ell$ and there is a $\operatorname{C}^k$-diffeomorphism $\lambda: M \to N$ such that $v(f)=f \circ \lambda^{-1}$ for all $f \in \operatorname{C}^k(M,\K)$.
\end{lem}

\begin{proof}
We set $\A:=\A^M:=\operatorname{C}^k(M,\K)$ and
 $\A_U:=\A^M_U:=\set{\left. f \in \A \right| \supp
f \sbs U}$ for open $U \sbs M$. Furthermore, for $x_0 \in M$ we write
$\Nc_{x_0}:=\Nc^M_{x_0}:=\set{\left. f \in \A \right|f(x_0)=0}$. The symbols
$\A^N$, $\A^N_V$ for open $V \sbs N$ and $\Nc^N_{y_0}$ for $y_0 \in N$ are
understood in the obvious analogous manner.

Since any $\Nc_{x_0}$ obviously is a vector subspace of $\A$ and $\A \cdot \Nc_{x_0} \sbs \Nc_{x_0}$, it is even an ideal of $\A$. Its codimension\footnote{The ``codimension'' of an ideal $I \ide \A=\A^M$ or $J \ide \A^N$ is always meant to be the vector space dimension of the quotient algebra $\A / I$ or $\A^N / J$, respectively.} is 1 because, evidently, we have $\Nc_{x_0} \oplus \K \cdot \1 = \A$. We will show that in fact \textsl{every} ideal $I \ide \A$ of codimension 1 takes the form $\Nc_x$ for a some $x \in M$.\footnote{The proof will also show the analogous statemanet for the ideals $J \ide \A^N$ with codimension 1.} For this, it suffices to show $I \sbs \Nc_x$ for some $x \in M$ because all $\Nc_x$ are ideals of codimension 1. Since, for $x_0 \ne x_1$, we have $\left(\Nc_{x_0} \cap \Nc_{x_1}\right) \oplus \K \cdot \1 \ne \A$, the point $x \in M$ corresponding to the ideal $I$ is uniquely determined.

Assume there exists an ideal $I \ide \A$ of codimension 1 such that for all $x \in M$ we have $I \not\sbs \Nc_x$, i.e., there exists $f_0 \in I$ with $f_0(x)\ne0$. By continuity, there is even an admissible\footnote{We call an open set $U$ \emph{admissible}, if there is a function $f \in I$ such that $f(y) \ne 0$ for all $y \in U$.} open $x$-neighbourhood $U \sbs M$ with $f_0(x') \ne 0$ for all $x' \in U$. Note that $M$ can, in our case, be covered by admissible open sets, say $M = \bigcup_{\gamma \in \Gamma} U_{\gamma}$. If $U \sbs M$ is open and $f \in \A_U$ then $\frac{f}{f_0}$ is a well-defined function in $\A = \operatorname{C}^k(M,\K)$ and $f = f_0 \cdot \frac{f}{f_0} \in I$ because $f_0 \in I$ and $I$ is an ideal of $\A$. Thus $\A_U \sbs I$.

Now let $\ph \in \A$ be a non-locally constant proper map, which exists by Lemma \ref{prop}.
Since $\text{codim}_{\A} I =1$ and $\dim
\left(\K \cdot \ph \oplus \K \cdot \ph^2  \right) = 2$, there are $a,b \in \K$
such that $a\ph + b\ph^2=:f_1 \in I$ is not zero and
$f_1^{-1}\left(\set{0}\right)$ is compact because $f_1^{-1}\left(\set{0}\right)=
(b\ph)^{-1}\left(\set{0}\right) \cup \ph^{-1} \left(\set{-\frac a b}\right)$ and
$b\ph$ is proper in the case of $b \ne 0$ and $f_1^{-1}\left(\set{0}\right) =
(a\ph)^{-1}\left(\set{0}\right)$ and $a\ph$ is proper if $b=0$. We define an
open set $U_1:=f_1^{-1}(\K \backslash \set{0})$. The compact subset
$f_1^{-1}\left(\set{0}\right) \sbs M$ is covered by finitely many admissible
open sets $U_2, U_3, \ldots , U_r \in (U_\gamma)_{\gamma \in \Gamma}$. Let
$(V_\beta)_{\beta \in B}$ be a locally finite refinement of $(U_1, \ldots ,
U_r)$ such that there is a smooth partition of unity $\left(\rho_\beta: M \to
[0,1]\right)_{\beta \in B}$ subordinate to $(V_\beta)_{\beta \in B}$. Let $B_1,
\ldots , B_r \sbs B$ be subsets such that $B=\bigcup_{i=1}^r B_i$ and let
$\beta \in B_i$ for some $i \in \set{1, \ldots , r}$ imply $V_{\beta} \sbs
U_i$. Then, for any $i \in \set{1, \ldots r}$ and $\beta \in B_i$, we have $\rho
\cdot \A \sbs \A_{U_i}$. We obtain \es{ \A = \1 \cdot \A = \sum_{\beta \in B}
\left(\rho_{\beta} \cdot \A\right) \sbs \sum_{i=1}^r \left(\sum_{\beta \in
B_i} \rho_{\beta} \cdot \A \right) \sbs \sum_{i=1}^r \A_{U_i} \sbs I, }
but this contradicts the fact that $I$ has codimension 1. This shows that for
any ideal $I \ide \A$ with codimension 1 there exists $x \in M$ such that
$I=\Nc_x$.

Since $v$ is an isomorphism of associative algebras, $v(\Nc_x)$ is, for any $x \in M$, an ideal of $\A^N$ with codimension 1, thus equal to $\Nc^N_y$ for some unique $y \in N$. We write $y=:\lambda(x)$ and obtain a mapping $\lambda: M \to N$. A dual argumentation with $v^{-1}$ instead of $v$ leads to a mapping $\wt \lambda: N \to M$ such that $\lambda \circ \wt \lambda = \id_N$ and $\wt \lambda \circ \lambda = \id_M$, so $\lambda$ is bijective. Note that we can perform the following calculation for $f \in \A$ and $y \in N$:
\es{
f\left(\lambda^{-1}(y)\right)=0 \quad \implies \quad f \in \Nc_{\lambda^{-1}(y)} \quad \implies \quad v(f) \in \Nc^N_y  \quad \implies \quad v(f)(y) =0.
}
But this already implies that $f \circ \lambda^{-1} = v(f)$ for any $f\in \A$ because we may replace the mapping $f$ by $f-r \cdot \1$ for arbitrary $r \in \K$ in the above calculation. We will now show that $\lambda$ is a $\operatorname{C}^k$-diffeomorphism.

By Lemma \ref{Cktop}, the topologies on $M$ and $N$ can be described as the initial topologies of the maps in $\operatorname{C}^k(M,\K)$ and $\operatorname{C}^\ell(N,\K)$, respectively. Since $v^{-1}(g)= g \circ \lambda \in \operatorname{C}^k(M,\K)$ for all $g \in \operatorname{C}^\ell(N,\K)$, the map $\lambda: M \to N$ is continuous. Let $(U'_i,\ph'_i)_{i \in I}$ be a locally finite atlas of $M$ and $(V'_j,\psi'_j)_{j \in J}$ a locally finite atlas of $N$. We modify the chart maps by multiplying them with the maps $\rho_i: M \to [0,1]$ and $\varpi_j: N \to [0,1]$, respectively, of smooth partitions of unity subordinate to the atlases. Then ${\rho_i}_{|U_i} \equiv 1$ and ${\varpi_j}_{|V_j} \equiv 1$ for open sets $U_i \sbs U'_i$ and $V_j \sbs V'_j$ still covering the whole manifold. We obtain new atlases, denoted by $(U_i,\ph_i)_{i \in I}$ and $(V_j,\psi_j)_{j \in J}$, where each chart map is defined on the whole manifold. Now let $t \in \set{1, \ldots , n}$, $i \in I$, $j \in J$ and define $A:=\ph_i\left(U_i \cap \lambda^{-1}(V_i)\right) \sbs \R^m$ and $D:=e_t^* (\psi_j(V_j \cap \lambda(U_i))) \sbs \R$. The map $e_t^* \circ \psi_j \circ \lambda \circ \left(\ph_i\right)^{-1}: A \to D$ is equal to
\es{
\underbrace{v^{-1}\left(e_t^* \circ \psi_j \right)}_{\in \operatorname{C}^k(M,\R)} \circ \left(\ph_i\right)^{-1} \in \operatorname{C}^k\left(A,D\right),
}
thus, since $t$,$i$ and $j$ were arbirarily chosen, $\lambda: M \to N$ is $\operatorname{C}^k$.

By the symmetry of the arguments, $\lambda^{-1}$ is a $\operatorname{C}^\ell$-map and, by the Inverse Mapping Theorem, even a $\max(k,\ell)$-times continuously differentiable diffeomorphism. Since the composition with $\lambda$ turns $\operatorname{C}^\ell(N,\K)$-maps into $\operatorname{C}^k(M,\K)$-maps  and $\lambda^{-1}$ turns $\operatorname{C}^k(M,\K)$-maps into $\operatorname{C}^\ell(N,\K)$-maps, so $\operatorname{C}^k(M,\K)=\operatorname{C}^\ell(M,\K)$ and $\operatorname{C}^k(N,\K)=\operatorname{C}^\ell(N,\K)$, thus $k=\ell$, completing the proof.
\end{proof}

The following Lemma will be used in Theorem \ref{isotheo}.
\begin{lem} \
\begin{enumerate}
\item Let $V$ be a vector space of finite dimension, $(U,\xi)$ a chart of $M$ and $T:U \to \End(V)$, $f:U \to V$ smooth functions. Then, for any multi-index $\alpha \in \N^m$, we have:
\begin{equation}\label{GenLeibRul}
\partial_{\xi}^{\alpha}(T \cdot f) = \sum_{\gamma \le \alpha} \begin{pmatrix}\alpha \\ \gamma\end{pmatrix} \left(\partial_{\xi}^{\gamma}T \cdot \partial_{\xi}^{\alpha -\gamma}f \right).
\end{equation}
\item We have, for $\alpha \in \N^m$:
\begin{equation}\label{multinom}
\sum_{\gamma \le \alpha} \frac{(-1)^{\abs{\alpha-\gamma}}}{\gamma! (\alpha-\gamma)!} =  \begin{cases}1 \text{, if } \alpha = 0\\ 0 \text{ otherwise} \end{cases} = \frac{\delta_{\alpha,0}}{\alpha !}.
\end{equation}
\end{enumerate}
\end{lem}

\begin{proof}\
\begin{enumerate}
\item The proof is done by mathematical induction over $\abs{\alpha}$. The claim is trivially true for $\alpha = 0$. For $\abs{\alpha}=1$ there exists a canonical basis vector $e_i$ of $\R^m$ such that $\alpha = e_i$ and $\partial_{\xi}^{\alpha}=\partial_{\xi_i}$. Then
\es{
\partial_{\xi_i}(T \cdot f) =  \partial_{\xi_i}T \cdot f + T \cdot \partial_{\xi_i}f
}
and the claim is also true. Now consider the claim shown for all multi-indices $\beta \in \N^m$ with $\abs{\beta}=n > 0$ and fix $\alpha \in \N^m$ with $\abs{\alpha}=n+1$. There exists a canonical basis vector $e_i$ of $\R^m$ such that $\beta := \alpha - e_i \in \N^m$ and $\abs{\beta}=n$. By induction hypothesis, we have
\[
\d_{\xi}^\beta(T \cdot f)=\sum_{\gamma \le \beta} \begin{pmatrix}\beta \\ \gamma\end{pmatrix} \d_{\xi}^\gamma T \cdot \d_{\xi}^{\beta-\gamma}f,
\]
thus
\begin{align*}
\d_{\xi}^\alpha(T \cdot f)&=\sum_{\gamma \le \beta} \begin{pmatrix}\beta \\ \gamma\end{pmatrix} \d_{\xi_i}\left(\d_{\xi}^\gamma T \cdot \d_{\xi}^{\beta-\gamma}f\right)=
\sum_{\gamma \le \beta} \begin{pmatrix}\beta \\ \gamma\end{pmatrix} \left(\d_{\xi}^{\gamma+e_i} T \cdot \d_{\xi}^{\beta-\gamma}f + \d_{\xi}^\gamma T \cdot \d_{\xi}^{\beta-\gamma+e_i}f\right)\\
&=\sum_{\gamma \le \beta} \begin{pmatrix}\beta \\ \gamma\end{pmatrix} \d_{\xi}^{\gamma+e_i} T \cdot \d_{\xi}^{\beta-\gamma}f
+
\sum_{\gamma \le \beta} \begin{pmatrix}\beta \\ \gamma\end{pmatrix} \d_{\xi}^\gamma T \cdot \d_{\xi}^{\beta-\gamma+e_i}f\\
&=\sum_{e_i \le \delta \le \alpha} \begin{pmatrix}\alpha - e_i \\ \delta - e_i\end{pmatrix} \d_{\xi}^{\delta} T \cdot \d_{\xi}^{\alpha-\delta}f
+
\sum_{\gamma \le \alpha - e_i} \begin{pmatrix}\alpha - e_i \\ \gamma\end{pmatrix} \d_{\xi}^\gamma T \cdot \d_{\xi}^{\alpha-\gamma}f\\
&=
1 \cdot \d_{\xi}^0 T \cdot \d_{\xi}^{\alpha}f +
\sum_{e_i \le \delta \le \alpha-e_i} \left( \begin{pmatrix}\alpha - e_i \\ \delta - e_i\end{pmatrix} +  \begin{pmatrix}\alpha - e_i \\ \delta\end{pmatrix} \right) \d_{\xi}^{\delta} T \cdot \d_{\xi}^{\alpha-\delta}f
+
1 \cdot \d_{\xi}^{\alpha} T \cdot \d_{\xi}^0 f\\
&=
\sum_{\delta \le \alpha} \begin{pmatrix}\alpha \\ \delta\end{pmatrix} \left(\partial_{\xi}^{\delta}T \cdot \partial_{\xi}^{\alpha -\delta}f \right).
\end{align*}
So the claim is true for all multi-indices $\alpha \in \N^m$.
\item For a given multi-index $\alpha \in \N^m$ we define a mapping $F_\alpha: \R^m \to \R$, $x \mapsto \prod_{i=1}^m x_i^{\alpha_i}$. By the Taylor Formula and the fact that $\partial^{\gamma}F \equiv 0$ for each $\gamma \in \N^m$ with $\gamma_i > \alpha_i$  for an index $i \in \set{1, \ldots , m}$, we obtain:
\es{
F_\alpha(x-y) &= \sum_{\gamma \le \alpha} \frac{\prod_{i=1}^m (-y_i)^{\alpha_i}}{\gamma !} \cdot \partial^{\gamma}F_{\alpha}(x) + 0\\
&= \sum_{\gamma \le \alpha} \frac{1}{\gamma !} \cdot \left(\prod_{i=1}^m (-y_i)^{\alpha_i}\right) \cdot \left(\frac{\alpha !}{(\alpha-\gamma)!}\prod_{j=1}^m x_j^{\alpha_j-\gamma_j}\right)
.}
This yields, by setting $x=y=(1, \ldots , 1) \in \R^m$:
\es{
\frac{\delta_{\alpha,0}}{\alpha !} = \frac{F_{\alpha}(0)}{\alpha !}=
\sum_{\gamma \le \alpha} \frac{1}{\gamma !} \cdot \left(\prod_{i=1}^m (-1)^{\alpha_i}\right) \cdot \left(\frac{1}{(\alpha-\gamma)!}\prod_{j=1}^m 1^{\alpha_j-\gamma_j}\right) = \sum_{\gamma \le \alpha} \frac{(-1)^{\abs{\alpha-\gamma}}}{\gamma ! (\alpha-\gamma)!}
.}
\end{enumerate}
\end{proof}

\begin{thm}\label{isotheo}
Let $\pi: \L \to M$ and $\varpi: \E \to N$ be two bundles of Lie algebras with fiber $\k$ and $\g$, respectively, $\dim \k = d$, $\dim \g = e$, $\Hom(\k / \lie{\k,\k} , \z(\k))= 0$, $\Hom(\g / \lie{\g,\g}, \z(\g))= 0$ and $\operatorname{S}(\k)=\K \cdot \1$, $\operatorname{S}(\g)=\K \cdot \1$, so that both Lie algebras are indecomposable by Lemma \ref{indecomplexreallemma}. Suppose there exists an isomorphism of Lie algebras $\mu: \Gamma^k(\L) \to \Gamma^\ell(\E)$ for some $0 \ne k,\ell \in \Ni$.
Then $k=\ell$ and
\begin{enumerate}
\item[(a)] if $k,\ell \in \N$, then $\mu$ is induced in a natural way by a $\operatorname{C}^k$-isomorphism of vector bundles $\kappa: \L \to \E$, i.e., if $\kappa': M \to N$ is the map underlying to $\kappa$ (that means $\kappa' \circ \pi = \varpi \circ \kappa$), then for all $X \in \Gamma^k(\L)$ we have $\mu(X) = \kappa \circ X \circ \left(\kappa'\right)^{-1}$. In particular, the manifolds $M,N$ are $\Ck$-diffeomorphic and the Lie algebras $\k,\g$ are isomorphic.
\item[(b)] if $k,\ell = \infty$, then the manifolds $M,N$ are diffeomorphic and the Lie algebras $\k,\g$ are isomorphic. After identifying the manifolds and the Lie algebras, $\mu$ turns into a linear differential operator of order at most $e-1$ taking the following local form on a bundle chart $(U,\ph)$ of $\pi: \L \to M$ with corresponding chart $(U,\xi)$ of $M$:
\begin{equation}\label{locaformofisos}
A^{\ph} \overset{\mu^{\ph}}{\longmapsto} \sum_{\abs{\alpha} < d} \frac{1}{\alpha !}  N^\alpha \cdot \left( \mu_0  \cdot \partial_{\xi}^\alpha A^{\ph}\right),
\end{equation}
where $\mu_0$ is a smooth function $U \to \Aut(\g)$ and we use the notation $N^{\alpha}=N_1^{\alpha_1} \circ \ldots \circ N_m^{\alpha_m}$ for multi-indices $\alpha=(\alpha_1, \ldots , \alpha_m) \in \N^m$ and smooth functions $N_1, \ldots , N_m \in \Ci(U,\operatorname{N}(\g))$.
\end{enumerate}
\end{thm}

\begin{proof}
The isomorphism of Lie algebras $\mu: \Gamma^k(\L) \to \Gamma^\ell(\E)$ induces an isomorphism of associative algebras $\wt \mu: \Cent\left(\Gamma^k(\L)\right) \to \Cent\left(\Gamma^\ell(\E)\right)$ by $\wt \mu (T):= \mu \circ T \circ \mu^{-1}$ for $T \in \Cent\left(\Gamma^k(\L)\right)$. This identifies, by Theorem \ref{Gammacentroid}, with an isomorphism $\Gamma^k(\Cent(\L)) \to \Gamma^\ell(\Cent(\E))$ and, by applying Lemma \ref{Liealglemma}.3 to the fibers of $\Cent(\E)$, we deduce that $\wt \mu$ induces an isomorphism of associative algebras $v: \operatorname{C}^k(M,\K) \to \operatorname{C}^\ell(N,\K)$ as follows: By \eqref{SGammakL}, an arbitary element of $\operatorname{S}\left(\Gamma^k(\L)\right)$ takes the form $f \cdot \1$, where $f \in \operatorname{C}^k(M,\K)$. By another application of \eqref{SGammakL}, we see that there is a function $v(f) \in \operatorname{C}^\ell(N,\K)$ such that we have
\begin{equation}\label{formofmutilde}
\wt{\mu}(f \cdot \1) = v(f) \cdot \1 + N_f,
\end{equation}
where $N_f \in \operatorname{N}\left(\Gamma^\ell(\E)\right)=\Gamma^\ell(\operatorname{N}(\E))$ is nilpotent and this decomposition is unique. Note that for constant functions $f \equiv c$, the map $\wt \mu(f \cdot \1)= c \cdot \wt \mu(\1) = c \cdot \1$ is semisimple and thus, by uniqueness of the decomposition \eqref{formofmutilde}, we have $N_f = 0$ for constant functions $f$. The morphism property of $v$ is shown by the following calculations:
\es{
\wt{\mu}(f \cdot \1) \cdot \wt{\mu}(g\cdot \1) = \wt{\mu}((f \cdot \1)\cdot(g\cdot \1)) = \wt{\mu}(fg \cdot \1) = v(fg) \cdot \1 + N_{fg}
}
and
\es{
\wt{\mu}(f \cdot \1) \cdot \wt{\mu}(g\cdot \1) = (v(f) \cdot \1 + N_f) \cdot (v(g) \cdot \1 + N_g) = v(f) v(g) \cdot \1 + \underbrace{v(f) \cdot N_g + v(g) \cdot N_f + N_f N_g}_{\text{nilpotent}}
.}
Furthermore, the Binomial Theorem yields, for $f \in \Ck(M,\K)$ and $r \in \N$:
\begin{equation*}
N_{f^r} = \sum_{t=1}^r \begin{pmatrix}r\\t\end{pmatrix} v(f)^{r-t} \left(N_f\right)^t.
\end{equation*}
Note that $v$ is bijective because $\wt \mu_{|\operatorname{S}\left(\Gamma^k(\L) \right)}$ and $\wt \mu^{-1}_{|\operatorname{S}\left(\Gamma^k(\E) \right)}$ are injective. By applying Lemma \ref{monsterlemma} to $v: \operatorname{C}^k(M,\K) \to \operatorname{C}^\ell(N,\K)$, we obtain $k=\ell$ and the existence of a $\operatorname{C}^k$-diffeomorphism $\lambda: M \to N$ such that $v(f)=f \circ \lambda^{-1}$ for all $f \in \operatorname{C}^k(M,\K)$. We now identifiy the manifolds $M, N$ via $\lambda$ and the associative algebras $\operatorname{C}^k(M,\K)$, $\operatorname{C}^k(N,\K)$ via $v$, so that we have the isomorphisms $\mu: \Gamma^k(\L) \to \Gamma^k(\E)$ and $\wt{\mu}: \Gamma^k(\Cent(\L)) = \Cent\left(\Gamma^k(\L)\right) \lra  \Gamma^k(\Cent(\E)) = \Cent\left(\Gamma^k(\E) \right)$. For all sections $A \in \Gamma^k(\L)$, all points $x \in M$ and all mappings $f \in \operatorname{C}^k(M,\K)$ with $f(x)=0$, we calculate:
\begin{equation}\label{mulocal}
\begin{split}
\left(\mu\left(f^e A\right)\right)_x & =\left(\mu\left(\left(f^e \cdot \1 \right)(A)\right)\right)_x
=\left(\left(\wt \mu \left(f^e \cdot \1 \right) \right) (\mu(A)) \right)_x
=\left(\left(f^e \cdot \1 + N_{f^e} \right) (\mu(A)) \right)_x\\
&=\left(f^e \mu(A) \right)_x +  \left(N_{f^e}(\mu(A))\right)_x = f(x)^e \left(\mu(A)\right)_x +
\sum_{t=1}^e \begin{pmatrix}e\\t\end{pmatrix} f(x)^{e-t} \left(N_f\right)^t_{x}(\mu(A))_x = 0+0= 0.
\end{split}
\end{equation}
Suppose a section $A \in \Gamma^k(\L)$ be zero on an open set $U \sbs M$, $x \in U$ and let $\rho: M \to [0,1]$ be a smooth function and $W \sbs U$ a smaller $x$-neighbourhood such that $\rho_{|M \backslash U} \equiv 0$ and $\rho_{|W} \equiv 1$. Then $A = (1-\rho)^e A$ and \eqref{mulocal} shows $(\mu(A))_x = 0$. Since $x$ was arbitrarily chosen, $\mu(A)$ is also zero on $U$. So $\mu: \Gamma^k(\L) \to \Gamma^k(\E)$ is local and, by the Peetre Theorems \ref{Peetre} and \ref{PseudoPeetre}, a differential operator.

If $k \in \N$, then Theorem \ref{PseudoPeetre} even implies that $\mu$ is a differential operator of order $0$ inducing a bundle isomorphism $\kappa: \L \to \E$ (cf. Definition \ref{bundlemorph} and Remark \ref{diffopsection}). This proves (a).

Now assume $k = \infty$. Let $A \in \Gamma(\L)$ be a section with $j^{e-1}_x(A)=0$. By Lemma \ref{Taylorlemma}, the section $A$ locally takes the form
\es{
A=\sum_{i=1}^r f_i^e A_i
}
for functions $f_i \in \Ci(M,\K)$, $f_i(x)=0$ and $A_i \in \Gamma(\L)$. By using relation \eqref{mulocal}, we see that
\es{
\left(\mu(A)\right)_x = \sum_{i=1}^r \left[\mu\left(f_i^e A_i\right)\right]_x = \sum_{i=1}^r 0  = 0.
}
This proves that the order of the differential operator $\mu$ is at most $e-1$ because $j^{e-1}_x(A)=0$ already implies $\left(\mu(A)\right)_x =0$. Let $(U,\ph)$ be a bundle chart of $\pi: \L \to M$ with corresponding chart $(U,\xi)$ of $M$ such that $\xi(U)$ is convex. We have the local forms $\mu^{\ph}: \Ck(U,\k) \to \Ck(U,\g)$, ${\wt \mu}^{\ph}: \Ck(U,\Cent(\k)) \to \Ck(U,\Cent(\g))$ and $N_f^\ph \in \Ck(U,\operatorname{N}(\g))$ of $\mu$, $\wt \mu$ and $N_f \in \Gamma^k(\operatorname{N}(\E))$, respectively. Locally, the decomposition \eqref{formofmutilde}, turns into:
\es{
\wt{\mu}^\ph(f \cdot \1) = f \cdot \1 + N^\ph_f.
}
The Taylor Formula yields, for each $x \in U$ and $A \in \Gamma(\L)$:
\begin{equation}\label{smalljet}
j^{e-1}_x \left(y \longmapsto A^{\ph}(y) - \sum_{\abs{\alpha} < e} \frac{\prod_{i=1}^m (\xi_i(y)-\xi_i(x))^{\alpha_i}}{\alpha !} \cdot \partial_{\xi}^{\alpha}A^{\ph}(x)\right)=0.
\end{equation}
For $x,y \in U$, $i \in \set{1, \ldots , m}$ and $\alpha \in \N^m$ with $\abs{\alpha} < e$ we write:
\es{
\xi_{i,x}(y)&:=\xi_i(y) - \xi_i(x),\\
\Xi_{\alpha,x}(y)&:=\prod_{i=1}^m (\xi_i(y)-\xi_i(x))^{\alpha_i},\\
\phi_{\alpha,x}(y)&:=\prod_{i=1}^m (\xi_i(y)-\xi_i(x))^{\alpha_i} \cdot \partial_{\xi}^{\alpha}A^{\ph}(x).
}
Then we define smooth mappings $N_1, \ldots , N_m : U \to \operatorname{N}(\g)$ and $\mu_0: U \to \Hom(\k,\g)$ (in the sense of Lie algebra morphisms), by setting for $x \in U$, $u \in \k$ and the constant map $c_u: U \to \g$, $y \mapsto u$:
\es{
\mu_0 (x)(u) := \mu^{\ph}(c_u)(x)
}
and
\es{
N_i(x) &:= N_{\xi_{i,x}}^\ph (x) = \wt \mu^\ph(\xi_{i,x} \cdot \1)(x) - (\xi_{i,x} \cdot \1)(x)\\
&= \wt \mu^\ph(\xi_{i,x} \cdot \1)(x)\\
}
We calculate:
\es{
\mu^{\ph}\left(\phi_{\alpha,x} \right)
&= \mu^{\ph}\left(y \mapsto \left( \prod_{i=1}^m (\xi_i(y)-\xi_i(x))^{\alpha_i} \cdot \1 \right) \left( \partial_{\xi}^{\alpha}A^{\ph}(x) \right) \right) \\
&=\wt \mu^{\ph}\left(\Xi_{\alpha,x} \cdot \1 \right) \cdot \mu^\ph\left(y \mapsto \partial_{\xi}^{\alpha}A^{\ph}(x) \right)\\
&= \left( \prod_{i=1}^m  \underbrace{\wt \mu^{\ph}\left(\xi_{i,x}^{\alpha_i} \cdot \1 \right)}_{=N_i^{\alpha_i}} \right) \cdot \left(\mu_0 \cdot \partial_{\xi}^\alpha A^\ph \right)\\
&= \left(\prod_{i=1}^m N_i^{\alpha_i}\right) \cdot \left( \mu_0 \cdot \partial_{\xi}^{\alpha}A^{\ph}\right)  =
N^\alpha \cdot \left( \mu_0 \cdot \partial_{\xi}^{\alpha}A^{\ph}\right).
}
Thus, by \eqref{smalljet} and the fact that $\mu$ is of order at most $e-1$, we have the local form
\begin{equation}\label{muphilocal}
A^{\ph} \overset{\mu^{\ph}}{\longmapsto} \sum_{\abs{\alpha} < e} \frac{1}{\alpha !} N^\alpha \cdot \left( \mu_0  \cdot \partial_{\xi}^\alpha A^{\ph}\right).
\end{equation}
It remains to show that each $\mu_0 (x)$, where $x \in U$, is bijective. This will be done in two steps. First, we verify the following identitiy for $A \in \Gamma(\L)$ and $B \in \Gamma(\E)$ with $\mu(A)=B$ (implying $\mu^\ph\left( A^\ph \right) = B^\ph$):
\begin{equation}\label{muidentity}
P\left(A \right) := \sum_{\abs{\alpha} < e} \frac{(-1)^{\abs{\alpha}}}{\alpha !}  \partial_{\xi}^{\alpha}\left(N^{\alpha} \cdot \left(\mu_0 \cdot A^{\ph}\right)\right)
= \sum_{\abs{\alpha} < e} \frac{(-1)^{\abs{\alpha}}}{\alpha !}  \partial_{\xi}^{\alpha}\left(N^{\alpha} \cdot B^{\ph}\right)=:Q\left(B\right).
\end{equation}
Note:
\begin{equation}\label{enilpot}
\text{If } S_1, \ldots , S_m \in \operatorname{N}(\g) \text{ and } \abs{\alpha} \ge e, \text{ then }S^{\alpha} = S_1^{\alpha_1} \circ \ldots \circ S_m^{\alpha_m}=0
\end{equation}
because $S^{\alpha}$ can be written as a sum of $\abs{\alpha}$-th powers of linear combinations of the $S_i$ and the $e$-th power of a nilpotent morphism cotained in $\operatorname{N}(\g)$ is zero. Therefore we can perform the following manipulations:
\es{
\sum_{\abs{\alpha} < e} \frac{(-1)^{\abs{\alpha}}}{\alpha !}  \partial_{\xi}^{\alpha}\left(N^{\alpha} \cdot B^{\ph}\right)
&\overset{\eqref{muphilocal}}{=}
\sum_{\abs{\alpha} < e} \frac{(-1)^{\abs{\alpha}}}{\alpha !}  \partial_{\xi}^{\alpha}\left(N^{\alpha} \cdot \sum_{\abs{\beta} < e} \frac{1}{\beta !} N^\beta \cdot \left(\mu_0 \cdot \partial_{\xi}^\beta A^{\ph}\right)\right)
\\&=\sum_{\stackrel{\abs{\alpha} < e}{\abs{\beta} < e}} \frac{(-1)^{\abs{\alpha}}}{\alpha ! \beta !}  \partial_{\xi}^{\alpha}\left(N^{\alpha+\beta} \cdot \left(\mu_0 \cdot \partial_{\xi}^\beta A^{\ph}\right)\right)
\\&\overset{\eqref{enilpot}}{=}
\sum_{\abs{\alpha+\beta} < e} \frac{(-1)^{\abs{\alpha}}}{\alpha ! \beta !}  \partial_{\xi}^{\alpha}\left( \left(N^{\alpha+\beta} \circ \mu_0 \right) \cdot \partial_{\xi}^\beta A^{\ph} \right)
\\&\overset{\eqref{GenLeibRul}}{=}
\sum_{\stackrel{\abs{\alpha+\beta} < e}{\gamma \le \alpha}} \frac{(-1)^{\abs{\alpha}}}{\beta ! \gamma ! (\alpha - \gamma) !}  \partial_{\xi}^{\gamma}\left(N^{\alpha+\beta} \circ \mu_0 \right) \cdot \partial_{\xi}^{\alpha + \beta - \gamma} A^{\ph}.
}
Now we perform the following substitutions: $\alpha':=\alpha + \beta - \gamma$, $\beta':=\gamma$ and $\gamma':=\beta$, thus $\alpha+\beta = \alpha' + \beta'$, $\alpha - \gamma = \alpha' - \gamma'$ and $\gamma \le \alpha \, \Longleftrightarrow \, \gamma' \le \alpha'$. And so we can finally show \eqref{muidentity}:
\es{
\sum_{\abs{\alpha} < e} \frac{(-1)^{\abs{\alpha}}}{\alpha !}  \partial_{\xi}^{\alpha}\left(N^{\alpha} \cdot B^{\ph}\right) &=
\sum_{\stackrel{\abs{\alpha+\beta} < e}{\gamma \le \alpha}} \frac{(-1)^{\abs{\alpha}}}{\beta ! \gamma ! (\alpha - \gamma) !}  \partial_{\xi}^{\gamma}\left(N^{\alpha+\beta} \circ \mu_0 \right) \cdot \partial_{\xi}^{\alpha + \beta - \gamma} A^{\ph}\\
&= \sum_{\stackrel{\abs{\alpha'+\beta'} < e}{\gamma' \le \alpha'}} \frac{(-1)^{\abs{\alpha'-\gamma'}} \cdot (-1)^{\beta'}}{\gamma' ! \beta' ! (\alpha' - \gamma') !}  \partial_{\xi}^{\beta'}\left(N^{\alpha'+\beta'} \circ \mu_0 \right) \cdot \partial_{\xi}^{\alpha'} A^{\ph}\\
&= \sum_{\abs{\alpha'+\beta'} < e} \left(\sum_{\gamma' \le \alpha'} \frac{(-1)^{\abs{\alpha'-\gamma'}}}{\gamma' ! (\alpha' - \gamma') !} \right) \cdot  \frac{(-1)^{\beta'}}{\beta'!}  \partial_{\xi}^{\beta'}\left(N^{\alpha'+\beta'} \circ \mu_0 \right) \cdot \partial_{\xi}^{\alpha'} A^{\ph}\\
& \overset{\eqref{multinom}}{=}
\sum_{\abs{\beta'} < e} \frac{(-1)^{\abs{\beta'}}}{\beta' !}  \partial_{\xi}^{\beta'}\left(N^{\beta'} \cdot \left(\mu_0 \cdot A^{\ph}\right)\right).
}
By the definition of $P$ and $\mu_0$, if $A^\ph(x) = u = A'^\ph(x)$ for $A,A' \in \Gamma(\L)$ and $x \in U$, then:
\es{
P(A)(x)=\sum_{\abs{\alpha} < e} \frac{(-1)^{\abs{\alpha}}}{\alpha !}  \partial_{\xi}^{\alpha}\left(N^{\alpha}(x) \left(\mu_0(x) \left(A^{\ph}(x)\right)\right)\right)=\sum_{\abs{\alpha} < e} \frac{(-1)^{\abs{\alpha}}}{\alpha !}  \partial_{\xi}^{\alpha}\left(N^{\alpha}(x) \left(\mu^{\ph}(c_u)(x) \right)\right)=P(A')(x).
}
So $P(A)(x)\in \g$ depends on $A^\ph$ only via $A^\ph(x) \in \k$ and we can define a linear map $P_x: \k \to \g$ by setting $P_x(A^\ph(x)):=P(A)(x)$ for $x \in U$. We may also define linear maps $Q_x: \g \to \g$ for $x \in U$ by $Q_x(B^\ph(x)):=Q(B)(x)$ for $x \in U$ due to an analogous element as above with the $N^\alpha$ instead of $\mu_0$. Since any $Q_x$ is a sum of $\1$ and a nilpotent linear map (see the sum in the third term of \eqref{muidentity} evaluated in $x$), it is bijective. We define a smooth mapping $\eta_0: U \to \Hom(\g,\k)$ (in the sense of Lie algebra morphisms) by setting for $x \in U$, $v \in \g$ and the constant map $c_v: U \to \g$, $x' \to v$:
\es{
\eta_0 (x)(v) := \left(\mu^{-1}\right)^{\ph}(c_v)(x).
}
We now fix $x \in U$ and $v \in \g$. Since $\mu$ is surjective, there exists $A \in \Gamma(\L)$ such that $\mu^\ph\left(A^\ph\right)=c_v$, thus $A^\ph(x)=\eta_0(x)(v)$. By \eqref{muidentity}, we have $P_x\left(\eta_0(x)(v) \right)=Q_x(v)$. The injectivity of $Q_x$ implies the injectivity of $\eta_0(x): \g \to \k$, yielding $e \le d$. By the symmetry of the arguments, $\mu_0(x): \k \to \g$ is also injective, yielding $d \le e$. So $\mu_0(x)$ is an isomorphism of Lie algebras.
\end{proof}

\begin{rem}
If the Lie algebra $\k$ is complex simple, then it is central, i.e. $\Cent(\k)=
\C \cdot \1$, by the Schur Lemma. If the Lie algebra $\k$ is real simple, then
Proposition X.1.5 of \cite{He78} says that $\k$ satisfies exactly one
of the following two conditions:
\begin{enumerate}
\item[A.] $\k$ admits a complex structure and the complexification $\k_{\C}=\k
\otimes_{\R} \C$ is the direct sum of two simple isomorphic ideals, hence
$\k_{\C}$ is not a simple $\C$-\La.
\item [B.] $\k_{\C}$ is a simple $\C$-\La.
\end{enumerate}
If the real simple Lie algebra $\k$ is the Lie algebra associated
to a compact Lie group, then its complexificaton $\k_{\C}$ is complex simple
by Lemma X.1.3 of \cite{He78}, so we are in case B
and have $\Cent(\k)=\R \cdot \1$.
\end{rem}

\begin{cor}\label{AutKoro}
If we have \textit{one} of the following two cases:
\begin{enumerate}
\item the Lie algebra $\k$ is complex simple,
\item the Lie algebra $\k$ is real simple and associated to
a compact Lie group,
\end{enumerate}
then for any $\mu \in \Aut\left(\Gamma^k(\L)\right)$, where $k \in \Ni$,
there is a $\Ck$-diffeomorphism $\lambda: M \to M$ such that $\mu$ can be
identified with some $\mu_0 \in \Gamma^k\left(\Aut_\lambda(\L)\right)$, i.e. a $\Ck$-section in $\Gamma^k\left(\Hom(\L,\L)\right)$ where for all $x \in M$ the map $\mu_0(x): \L_x \to \L_{\lambda(x)}$ is an isomorphism of Lie algebras. The bundle $\Aut_\lambda(\L)$ is isomorphic to $\Aut(\L):=\Aut_{\id_M}(\L)$ by $\left(f: \L_x \rightarrow \L_{\lambda(x)} \right) \longmapsto \left(\mu^{-1} \circ f: \L_x \rightarrow \L_x\right)$.
\end{cor}

\begin{cor}
In both of the cases of Corollary \ref{AutKoro}, the Lie algebra of $\Ck$-sections, where $k \in \Ni$, of the trivial bundle $\L = M \times \k$ is naturally isomorphic to $\Ck(M,\k)$ and, since $\Diff^k(M)$ and $\Gamma^k\left(\Aut(\L)\right) \cong \Ck(M,\Aut(\k))$ can be naturally embedded into $\Aut(\Ck(M,\k))$ as a subgroup and a normal subgroup, respectively, we obtain the isomorphism
\es{
\Aut(\Ck(M,\k)) \cong \Ck(M,\Aut(\k)) \rtimes \Diff^k(M).
}
\end{cor}


\end{document}